\documentclass[a4paper,10pt]{article}

\usepackage[T1]{fontenc}
\usepackage[utf8]{inputenc}

\usepackage{latexsym}
\usepackage{mathrsfs}
\usepackage{amsfonts,amsmath,amssymb,amsthm}
\usepackage{bm}
\usepackage{stmaryrd}

\usepackage{indentfirst}
\usepackage{verbatim}
\usepackage{microtype}
\usepackage{float}
\usepackage{enumitem}

\usepackage{multirow}
\usepackage{diagbox}
\usepackage{booktabs}
\usepackage{array}
\usepackage{tabularx}
\usepackage[table]{xcolor}

\usepackage{graphicx}
\usepackage{epstopdf}
\usepackage{subcaption}
\usepackage{tikz}
\usepackage{pgfplots}
\usetikzlibrary{fit}
\pgfplotsset{compat=1.18}

\usepackage{pdfpages}
\usepackage{algorithm}
\usepackage{algorithmic}
\usepackage[title]{appendix}

\usepackage[unicode,colorlinks,linkcolor=blue,citecolor=red]{hyperref}

\numberwithin{equation}{section}
\allowdisplaybreaks

\newtheorem{thm}{Theorem}[section]
\newtheorem{lemma}[thm]{Lemma}

\newtheorem{rem}[thm]{Remark}

\newcommand{\safeincludegraphics}[2][]{%
  \IfFileExists{#2}{\includegraphics[#1]{#2}}{%
    \fbox{\parbox[c][0.18\textheight][c]{0.78\linewidth}{\centering
    Placeholder for Figure\\[0.5ex]\texttt{\detokenize{#2}}}}%
  }%
}

\renewcommand{\epsilon}{\varepsilon}
\newcommand{\mat}[1]{\bm{#1}}

\begin{document}

\begin{center}
\Large\bf
Adaptive-Distribution Randomized Neural Networks for PDEs:\\
A Low-Dimensional Distribution-Learning Framework
\end{center}

\begin{center}
{\large You Yang}\footnote{School of Mathematics and Statistics, Xi'an Jiaotong University, Xi'an, Shaanxi 710049, China. Email: {\tt yangyou@stu.xjtu.edu.cn}}
\quad {\rm and} \quad
 {\large Fei Wang}\footnote{School of Mathematics and Statistics \& State Key Laboratory of Multiphase Flow in Power Engineering, Xi'an Jiaotong University, Xi'an, Shaanxi 710049, China. The work of this author was partially supported by the National Natural Science Foundation of China (Grant No.\ 92470115). Email: {\tt feiwang.xjtu@xjtu.edu.cn}}
\end{center}

\medskip

\begin{quote}
{\bf Abstract. } 
Randomized neural networks (RaNNs) are attractive for partial differential equations (PDEs) because they replace expensive end-to-end training with a linear least-squares solve over randomized hidden features. Their practical performance, however, depends strongly on the sampling distribution of the hidden-layer parameters, which is usually chosen heuristically and problem by problem. This distribution sensitivity is a central bottleneck in randomized neural PDE solvers. In this work, we propose Adaptive-Distribution Randomized Neural Networks (AD-RaNN), a framework that promotes randomized feature generation from a fixed heuristic choice to a low-dimensional adaptive optimization problem. Instead of training all hidden weights and biases, AD-RaNN parameterizes the hidden-feature sampling distribution by a low-dimensional vector $\bm p$ and optimizes only $\bm p$, thereby preserving the least-squares structure of RaNNs while reducing manual distribution tuning. The method uses a two-stage strategy: ridge-regularized reduced training for stable distribution-parameter optimization, followed by an unregularized least-squares refit for final solution recovery. We develop two adaptive mechanisms, PDE-Driven Adaptive Distribution (PDAD) and Data-Driven Adaptive Distribution (DDAD), and deploy them in space-time solvers, discrete-time solvers, and operator-learning models. We also incorporate an adaptive layer-growth enhancement for localized structures. For the reduced optimization problem, we establish well-posedness of the reduced objectives, consistency of ridge-regularized minimizers, an efficient gradient formula, and a practical lower-bound estimate for the ridge parameter. Numerical experiments on benchmark problems show that AD-RaNN provides an effective distribution-level adaptation mechanism, reduces reliance on hand-crafted hidden-feature distributions, and achieves strong empirical accuracy.

\end{quote}


\medskip


\medskip

\section{Introduction}\label{sec:intro}

Traditional numerical approaches, including finite element methods, discontinuous Galerkin methods, finite volume methods, finite difference methods, and spectral methods, have long served as standard tools for PDE simulation. These methods are highly successful in many regimes, but they may also face major difficulties in settings involving complex geometries, moving structures, multiscale behavior, or high-dimensional state spaces.  
 In many computational mechanics and engineering applications, moreover, one is not only interested in solving a single well-resolved PDE instance, but also in repeatedly solving families of related problems under varying parameters, source terms, geometries, boundary conditions, or initial data. Such multi-query settings create a demand for solvers that combine efficient linear-algebraic training, mesh flexibility, and robust feature construction. 
In recent years, neural-network-based solvers have emerged as an alternative paradigm for PDE approximation, motivated by the expressive power of neural networks and their mesh-free nature (\cite{cybenko1989Approximation,hornik1991Approximation,pinkus1999Approximation,chen1995Universal,ainsworth2021Galerkin,raissi2019Physics-informed,sirignano2018dgm,wang2024Computing,wang2024Solving,lu2021Learning}).

Among neural PDE solvers, randomized neural networks (RaNNs) are particularly attractive because they avoid expensive end-to-end optimization of all network parameters. In a typical RaNN, the hidden-layer parameters are randomly sampled and then frozen, while only the output-layer coefficients are determined, usually by solving a linear least-squares problem. This structure yields substantial computational savings and has led naturally to a range of efficient PDE solvers built on classical numerical formulations, including Petrov--Galerkin, discontinuous Galerkin, finite-difference, and domain-decomposition-type methods (\cite{dong2021local,shang2023randomized,shang2024Randomized,sun2024local1,sun2024Local2,chen2022bridging,li2025Local,dang2024Adaptive}). 
In this sense, RaNNs occupy an appealing middle ground between traditional numerical methods and fully trained neural architectures: they retain a trial-space viewpoint and a linear-algebraic solution stage, while offering the mesh-free flexibility of neural representations.

However, the practical success of RaNN-based PDE solvers depends critically on a design choice that is often treated heuristically: the sampling distribution of the hidden-layer parameters. If the random weights are too small, the resulting features may be too flat to capture relevant solution structures; if they are too large, the features may become highly oscillatory or ineffective on the sampled domain. In anisotropic, time-dependent, or sharp-interface problems, this sensitivity becomes even more pronounced. As a result, the approximation quality of RaNNs often depends strongly on manually chosen scale parameters or problem-dependent initialization rules, which limits their robustness and reusability in repeated PDE simulations.

Existing attempts to mitigate this issue typically move in several directions.
One direction keeps the hidden representation fixed after construction, but seeks to improve that construction through transferable feature spaces (\cite{zhang2024transferable}) or problem-informed initialization heuristics, such as frequency-based designs (\cite{dang2024Adaptive}).
A second direction introduces outer search procedures, such as differential evolution, to tune distributional hyperparameters (\cite{dong2022Oncomputing}).
A third direction moves closer to fully trained neural networks by optimizing a much larger collection of hidden parameters (\cite{xu2025Subspace}).
The first two directions preserve much of the low-cost least-squares structure, but they either rely on fixed feature constructions or incur problem-dependent search overhead.
The third direction improves flexibility, but at the price of substantially increased optimization complexity.
This leaves a methodological gap between fully fixed and fully trained hidden representations.
The main idea of this work is to fill this gap by promoting random feature generation from a fixed heuristic choice to a low-dimensional adaptive optimization problem. Rather than training all hidden weights and biases, we parameterize the hidden-layer sampling distribution by a low-dimensional vector $\bm p$ and optimize only $\bm p$. In this way, we retain the least-squares structure and computational efficiency of RaNNs, while allowing the random feature family itself to adapt to the PDE, the numerical solution, or the operator-learning task at hand.

In this work, we propose \emph{Adaptive-Distribution Randomized Neural Networks} (AD-RaNN), a framework for PDE solvers and operator-learning models based on \emph{distribution-level adaptation} of randomized features. The central viewpoint is that, for randomized neural solvers, one need not choose between a fully fixed random hidden representation and a fully trained hidden representation. Instead, one may optimize only a small number of distribution parameters that govern how the random hidden features are generated. 
This yields an intermediate adaptation mechanism that is expressive enough to adapt the trial space to the problem, yet remains far cheaper than full hidden-parameter training.

To realize this idea, we develop two complementary adaptive mechanisms. The first, termed \emph{PDE-Driven Adaptive Distribution} (PDAD), learns the distribution parameters directly from the PDE residual least-squares system. The second, termed \emph{Data-Driven Adaptive Distribution} (DDAD), learns them from numerically computed or externally available solution data. These two mechanisms share the same reduced optimization structure and can therefore be embedded into several computational settings. 

The main contributions of this paper are as follows:
\begin{itemize}
    \item \textbf{Adaptive distribution learning for randomized neural PDE solvers.}
    We introduce AD-RaNN, in which the hidden-layer sampling distribution is parameterized by a low-dimensional vector $\bm p$ and adapted through outer-loop optimization, while the output-layer coefficients are still determined by least-squares. This creates a structural middle ground between fully fixed random-feature methods and fully trained neural PDE solvers.

    \item \textbf{A stable reduced optimization strategy and two adaptive mechanisms.}
    We propose a ridge-regularized reduced optimization stage for learning the distribution parameters, followed by an unregularized least-squares refit for final solution recovery. Within this framework, we formulate both PDE-driven and data-driven routes, namely PDAD and DDAD.

    \item \textbf{Deployment across multiple computational regimes.}
    We instantiate the adaptive-distribution idea in unified space-time PDE solvers, discrete-time time-stepping solvers, and RaNN-DeepONet operator-learning models. In addition, we incorporate a local layer-growth enhancement for cases in which global adaptive distribution learning alone is insufficient to resolve strongly localized structures.

    \item \textbf{Analysis of the reduced optimization problem.}
    We establish existence and continuity results for the reduced ridge objective, prove consistency of ridge-regularized minimizers as $\lambda\to 0$, derive an efficient gradient formula via the envelope theorem, and provide a practical lower-bound estimate for the ridge parameter.

    \item \textbf{Extensive empirical validation.}
    Through elliptic, nonlinear, time-dependent, sharp-layer, geometric-flow, high-dimensional, and operator-learning examples, we show that adaptive distribution learning can substantially reduce the sensitivity of RaNN solvers to manually chosen feature distributions and yield strong empirical accuracy across a broad range of problem types.
\end{itemize}

\begin{table}[!ht]
\centering
\small
\caption{Methodological positioning of AD-RaNN among representative randomized neural network solvers. Here ``low'' / ``high'' refer to the dimension or cost of hidden-parameter adaptation relative to the hidden-layer size. The methods are ordered roughly from fully fixed hidden representations to increasingly adaptive hidden-parameter strategies.}
\label{tab:positioning_ADRaNN}
\setlength{\tabcolsep}{5pt}
\renewcommand{\arraystretch}{1.15}
\begin{tabular}{@{}p{2.5cm}p{3.4cm}p{2.3cm}p{2.4cm}p{4.1cm}@{}}
\toprule
Method class 
& Hidden-layer parameters 
& Hidden-parameter optimization 
& Manual distribution tuning 
& Feature reuse across problem settings \\
\midrule

Fixed-distribution methods
& Fixed after random draw
& None
& High
& Possible, but often requires case-by-case retuning \\
\midrule

Designed fixed-feature methods
& Fixed after problem-informed, frequency-informed, or transferable feature design
& None
& Moderate to high
& Moderate; reusable for related PDE tasks, but still sensitive to solution scales. \\
\midrule

{\bf AD-RaNN}
& Randomized, with \emph{distribution-level adaptation}
& Low
& Substantially reduced
& Broad at the framework level; the same adaptation principle applies in space-time, discrete-time, and operator-learning settings\\
\midrule

Search-based parameter selection
& Randomized and fixed after outer search of distribution parameters
& Moderate
& Reduced, but replaced by search cost
& Limited; often requires a new search for each problem regime \\
\midrule

Fully trained hidden-parameter methods
& Trained
& High
& Low
& Broad, but usually requires retraining or substantial fine-tuning \\
\bottomrule
\end{tabular}
\end{table}

 To clarify the methodological position of the proposed framework, it is helpful to distinguish several broad levels of adaptivity for hidden-layer parameters in randomized or neural PDE solvers.
At one extreme are methods in which the hidden parameters are randomly generated and then kept fixed throughout the computation.
A mild refinement of this setting improves the fixed hidden representation through feature-design or initialization strategies, while preserving the same fixed-feature least-squares structure.
A further step is to introduce external search or meta-optimization procedures for selecting favorable distributional hyperparameters.
At the opposite extreme are methods that directly train a large collection of hidden-layer parameters by gradient-based optimization.
Table~\ref{tab:positioning_ADRaNN} summarizes this methodological picture.
In contrast to both fully fixed and fully trained hidden representations, AD-RaNN performs \emph{adaptive distribution learning}: the hidden features remain randomized, but the distribution that generates them is updated through a low-dimensional parameter vector. In this sense, AD-RaNN occupies a structural middle ground between fixed random-feature methods and fully trained hidden-parameter methods.

 \par
The rest of this paper is organized as follows. Section~\ref{sec:Review:RaNN} reviews RaNNs under the space-time and discrete-time frameworks. Section~\ref{sec:AD-RaNN} introduces AD-RaNN with two strategies, PDAD and DDAD. Section~\ref{AD-RaNN-dif-setting} presents the resulting ST/DT variants, the adaptive layer-growth strategy, and further extends the method to operator learning via AD-RaNN-DeepONet. Section~\ref{sec:Theoretical-Analysis} provides the theoretical analysis. Section~\ref{sec:Numerical_experiment} reports numerical results, and Section~\ref{sec:conclusion} concludes the paper.

\section{A Brief Review of Randomized Neural Networks}\label{sec:Review:RaNN}

The architecture of a randomized neural network (RaNN) is similar to that of a fully connected
neural network, except that the hidden-layer weights and biases are randomly initialized and then fixed.
Let $L$ denote the number of hidden layers, and let $m_\ell$ be the number of neurons in the
$\ell$-th hidden layer. The hidden-layer features are defined recursively by
\begin{equation}\label{eq:hidden_features_rann}
\bm{\phi}^{(0)}(\bm{x})=\bm{x},
\qquad
\bm{\phi}^{(\ell)}(\bm{x})
=
\rho^{(\ell)}
\left(
\bm{W}^{(\ell)}\bm{\phi}^{(\ell-1)}(\bm{x})
+
\bm{b}^{(\ell)}
\right),
\quad \ell=1,\ldots,L,
\end{equation}
where $\bm{W}^{(\ell)}$ and $\bm{b}^{(\ell)}$ denote the weight matrix and bias vector of the
$\ell$-th hidden layer, respectively, and $\rho^{(\ell)}$ is the activation function. In a RaNN,
the hidden-layer parameters $\{\bm{W}^{(\ell)},\bm{b}^{(\ell)}\}_{\ell=1}^L$ are randomly initialized
and kept fixed. Thus, the last hidden layer produces a fixed feature vector
\begin{equation}\label{eq:last_hidden_feature_vector}
\bm{\phi}^{(L)}(\bm{x})
=
\big(\phi^{(L)}_1(\bm{x}),\ldots,\phi^{(L)}_{m_L}(\bm{x})\big)^T
\in \mathbb{R}^{m_L}.
\end{equation}
The network output is obtained by linearly combining these fixed features. For a scalar-valued
approximation, the RaNN takes the form
\begin{equation}\label{eq:rann_linear_output}
	u_N(\bm{x})
	=
	\sum_{j=1}^{m_L} w_j^0 \phi^{(L)}_j(\bm{x})
	=
	\bm{w}^0\bm{\phi}^{(L)}(\bm{x}),
\end{equation}
where $\bm{w}^0=[w_1^0,\ldots,w_{m_L}^0]\in\mathbb{R}^{1\times m_L}$ denotes the output-layer
weight vector. The feature functions $\{\phi^{(L)}_j\}_{j=1}^{m_L}$ are fixed after random initialization,
whereas the output-layer coefficients are determined by solving a least-squares problem constructed
from the data or the governing equations.

\begin{figure}[!ht]
	\centering
	\includegraphics[width=0.55\textwidth]{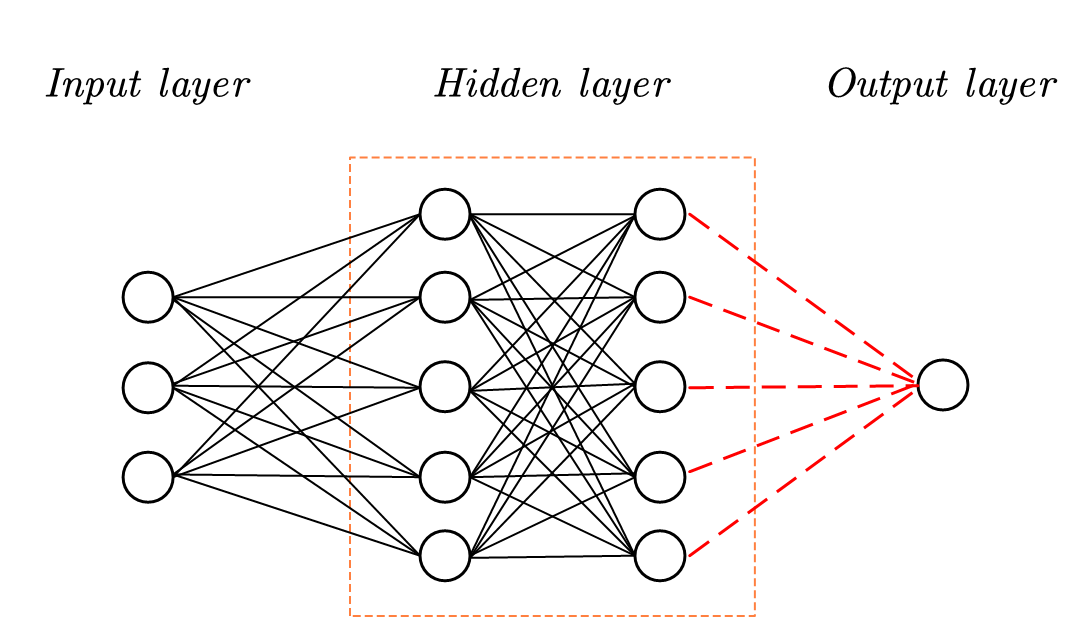}
	\caption{In the architecture of Randomized Neural Networks, black solid lines denote parameters that have been initialized and fixed, while red dashed lines denote parameters to be solved.}
	\label{fig:RNN-ref}
\end{figure}

\subsection{RaNNs with unified space-time framework}
In the unified space-time framework, RaNNs treat temporal and spatial variables on an equal footing, thereby avoiding errors introduced by time discretization. We note that the unified space-time framework is not restricted to
time-dependent problems. For time-independent equations, the formulation
naturally reduces to a purely spatial setting, while the overall procedure
remains unchanged.
 We consider the following problem \eqref{eq:RaNN_e} and approximate its solution using a shallow neural network with only one hidden layer
\begin{equation}\label{eq:RaNN_e}
\begin{aligned}
\mathcal{G}\!\left[u(\bm{x},t)\right]
  &= f(\bm{x},t), \quad (\bm{x},t)\in \Omega\times(0,T],\\
\mathcal{B}\!\left[u(\bm{x},t)\right]
  &= g(\bm{x},t), \quad  (\bm{x},t)\in \partial\Omega\times(0,T],\\
u(\bm{x},0) &= h(\bm{x}), \quad \bm{x}\in\overline{\Omega}.
\end{aligned}
\end{equation}
\par
 Let \(\bm{\phi}=[\phi_i]\in \mathbb{R}^{m_1}\) denote the neurons of the hidden layer, and let \(\bm{w}^{0} = [w_i^{0}]\in \mathbb{R}^{1\times m_1}\) be the vector of weights associated with these neurons (see the red line in Fig.~\ref{fig:RNN-ref}). The approximation takes the form:
\begin{align}\label{eqn:u}  
	u(\bm{x},t)\approx \bm{w}^{0}\bm{\phi},
\end{align}
where $\bm{w}^{0}\bm{\phi}$ denotes matrix multiplication.
By substituting equation \eqref{eqn:u} into equation \eqref{eq:RaNN_e} and evaluating the result at $n$ fixed collocation points, we obtain the least-squares problem:
\begin{equation}
	\left[ \begin{array}{c}
		A_o\\
		\eta_b A_{b}\\
		\eta_i A_{i}\\
	\end{array} \right] \bm{w}=\left[ \begin{array}{c}
		F\\
		\eta_b G_{b}\\
		\eta_i H_{i}\\
	\end{array} \right].
\end{equation}
\par
After solving the penalized least-squares problem (e.g., via a QR- or SVD-based solver) with the prescribed penalty coefficients $\eta_b$ and $\eta_i$, the weights from the last hidden layer to the output layer are replaced by the computed coefficients $\bm w$. 
Throughout this paper, for convenience of analysis, all coefficients from least-squares solutions are represented as column vectors. The weights of the final hidden layer are therefore defined as row vectors, i.e., $\bm{w}^0 = \bm{w}^\top$. Moreover, if the operator $\mathcal{G}\!\left[u(\bm{x},t)\right]$ is nonlinear, it must be linearized via iterative methods such as Picard's iteration or Newton's iteration.

\subsection{RaNNs with discrete-time framework}
Compared with a unified space-time nonlinear solve, the discrete-time formulation allows step-by-step linearization of the nonlinear operator using solutions from previous time steps, thereby requiring only the solution of a sequence of linear subproblems.
\par
In time-dependent partial differential equations, the choice of time discretization
scheme plays a crucial role in determining the stability, accuracy, and computational
efficiency of numerical methods. Common approaches include explicit, implicit, and
semi-implicit schemes~(\cite{frank1997On}),
 each offering distinct advantages and limitations depending on the problem characteristics.
\par
Assuming that the spatial approximation is sufficiently accurate,
the time-dependent problem can be written in the abstract form
\begin{equation}\label{eqn:dt-ode}
u'(t) = \mathcal{F}(t,u(t)) + \mathcal{L}(t,u(t)),
\end{equation}
where $\mathcal{F}$ denotes a nonlinear operator and
$\mathcal{L}$ denotes a linear (or stiff) operator.
\par
The time domain \((0,T]\) is discretized into \(N_t+1\) equally spaced time points
\begin{equation}
	0=t_0  < t_1 = \Delta t< t_2 = 2\Delta t \dots< t_{N_t} = T,
\end{equation}
where \(\Delta t = \frac{T}{N_t}\) is the uniform time step size.
\par
For the numerical 
solution of \eqref{eqn:dt-ode}, we consider a general $k$-step semi-implicit linear multistep scheme of the form
\begin{equation}
\sum_{j=0}^{k} a_j u_{n+1-j}
=
\Delta t \, \mathcal{L}(t_{n+1},u_{n+1})
+
\Delta t \mathcal{F}_{\mathrm{lin}}\!\left(t_{n+1};u_{n+1},u_{n},\dots,u_{n-k+1}\right),
\end{equation}
where $\mathcal{F}_{\mathrm{lin}}$ denotes a linearized approximation of the nonlinear operator,
constructed using extrapolation or coefficient freezing.
A representative choice of such schemes is obtained by combining
BDF time discretization with explicit extrapolation of the nonlinear term,
leading to the classical IMEX-BDF methods.
\par
If the spatial terms are solved by RaNNs, the above equation yields the following least-squares formulation, where the initial condition $u(t_0)=h$ is prescribed.
\begin{equation}\label{eqn:LSP}
	\left[ \begin{array}{c}
		 A_o\\
		\eta_b A_b\\
	\end{array} \right] \bm{w}_i=\left[ \begin{array}{c}
		 F_i\\
		\eta_b G_{b_i}\\
	\end{array} \right],
\end{equation}
where \(\bm{w}_i, F_i, G_{b_i}\) represent the corresponding matrices at the $i$-th step.

\section{Adaptive Distribution of Randomized Neural Networks}\label{sec:AD-RaNN}

Before introducing the concrete constructions, we briefly clarify the technical role of AD-RaNN. The central object is not the optimization of all hidden-layer parameters, but the optimization of a low-dimensional parameter vector $\bm p$ that governs how randomized hidden features are generated. In this way, the hidden representation remains randomized, while the feature family itself becomes adaptive to the problem. The following subsections make this idea concrete through parameterized random feature generation, reduced optimization, and the two adaptive mechanisms PDAD and DDAD.

\subsection{Random weight parameterization and bias construction}
In Randomized Neural Networks, the bias vector $\mat{b}$ can be determined from the domain and weights $\mat{W}\in \mathbb{R}^{m_1 \times d}$, 
\begin{equation}
    \mat{b}=-\left(\bm{W} \odot \bm{B}\right)\bm{1}_{d\times 1},
\end{equation}
where $\bm{B}\in \mathbb{R}^{m_1 \times d}$ is randomly generated from the uniform distribution $\mathcal{U}(\bm{q}_1, \bm{q}_2)$. Here we assume that
$\Omega=\prod_{j=1}^{d}\,[q_{1,j},\,q_{2,j}]\subset \mathbb{R}^{d}$
is a hyper-rectangle, and denote
$\bm{q}_1=\big(q_{1,1},\ldots,q_{1,d}\big)^{\top},
\bm{q}_2=\big(q_{2,1},\ldots,q_{2,d}\big)^{\top}$.
 Moreover, $\odot$ denotes element-wise multiplication, and $\bm{1}_{d\times 1}$ denotes a column vector of ones with dimension $d \times 1$ (right-multiplication by $\bm{1}_{d\times 1}$ realizes column-wise summation of the matrix $\bm{W} \odot \mat{B}$).
\par
Selecting appropriate weights $\bm{W}$ is of paramount importance. Excessively small $\bm{W}$ can impair the network's approximation power, whereas overly large $\bm{W}$ may render many neurons ineffective on the sampled points, thereby reducing the number of effective basis functions. To address this, we introduce a method in this section that automatically determines $\bm{W}$ based on the PDE information or the available data.
\par
For a shallow neural network with only one hidden layer, the weights $\bm{W}$ are usually selected from a specific distribution $\mathcal{H}(\bm{p})$, such as a uniform distribution or a Gaussian distribution. If the random seed is fixed, the selection of weights $\bm{W}$ turns into defining the parameters $\bm{p}$ that are related to these distributions.
The uniform distribution $\mathcal{U}(\bm{q}_1,\bm{q}_2)$
is one of the most common sampling strategies in randomized neural networks. It can be simplified to $\mathcal{U}(-\bm{r},\bm{r})$ with a vector parameter $\bm{r}=[r_1, r_2, \dots, r_d]$, where $r_i$ represents the input scale in the $i$-th dimension. Here $\mathcal{U}(-\bm{r},\bm{r})$ denotes the product uniform distribution with independent components.
\par
Initially, we fix the random seed and select $\bm{W}_0$ from $\mathcal{U}(-\bm{1},\bm{1})$. Therefore, $\bm{W}$ satisfying uniform distribution over the interval $[-\bm{r},\bm{r}]$ can be denoted as 
\begin{equation}
    \bm{W}=\bm{r}\odot \bm{W}_0. 
\end{equation}
\par
As for $\bm{b}$, we also randomly generate a $\bm{B}$, and then $\mat{b}=-\left(\bm{W} \odot \bm{B}\right)\bm{1}_{d\times 1}$. This method significantly reduces the degrees of freedom.

\subsection{Sensitivity with respect to the distribution parameter \texorpdfstring{$\bm p$}{p}}\label{sec:example}

This subsection provides a simple but representative example showing that the approximation quality of RaNNs can be highly sensitive to the distribution parameter $\bm p$, even when all other ingredients are kept fixed.

 We consider a simple linear Poisson equation with Dirichlet boundary conditions
    \begin{equation}\label{eq:poisson}
\begin{cases}
-\Delta u = f(x,y), & (x,y)\in\Omega=(-1,1)^2,\\[2pt]
u(x,y) = h(x,y),  & (x,y)\in\partial\Omega,
\end{cases}
\end{equation}
where the exact solution is $\sin(\pi x)\sin(5\pi y)$.
This solution exhibits markedly different oscillatory behavior in the two coordinate directions. Specifically, the solution changes slowly in the $x$-direction, while it oscillates rapidly in the $y$-direction. In our neural network architecture, this directional anisotropy is controlled through the parameters $\bm{p}=(r_x,r_y)$. To investigate the importance of selecting a suitable $\bm{p}$, we conduct a manual comparison of five distinct choices of $\bm{p}$
\begin{equation}
   (0.3,5),\quad (0.8,8),\quad (0.9,9),\quad (1,10),\quad (3,20),
\end{equation}
\par

Specifically, the activation function is the Gaussian function. In Table~\ref{tab:poisson_exam}, ``time" denotes the total computational time, $m_1$ represents the number of basis functions, and accuracy is evaluated by the relative $\ell_2$ error.
 To further enhance numerical stability, all experiments are implemented in MATLAB, and derivatives are obtained analytically.
As shown in Table~\ref{tab:poisson_exam}, the choice $\bm{p}=(0.3,5)$ yields a
relatively large $\ell_2$ error of $1.17e-05$.
In contrast, the choices $\bm{p}=(0.8,8)$, $\bm{p}=(0.9,9)$, and $\bm{p}=(1,10)$ achieve
significantly higher accuracy, with $\ell_2$ errors of $1.22e-08$, $3.25e-08$ and $1.32e-07$, respectively.
However, overly large values of $\bm{p}$ (e.g., $\left(3,20\right)$) also lead to
degraded performance.
These results indicate that good performance can be achieved as long as $\bm{p}$ lies within a reasonable range. This motivates an adaptive strategy for selecting $\bm{p}$.
\begin{table}[!ht]
\centering
\caption{Comparison of different selected parameters $\bm{p}=(r_x,r_y)$ for equation \eqref{eq:poisson}.}
    \label{tab:poisson_exam}
\setlength{\tabcolsep}{18pt} 
\begin{tabular}{cccc}
\hline
 $\bm{p}=(r_x,r_y)$& $m_1$  & time (s) & $\ell_2$ error \\ \hline
 \multirow{1}{*}{$(0.3,5)$} & 600 & 0.21 & 1.17e-05 \\ 
\multirow{1}{*}{$(0.8,8)$} & 600 &  0.22 & 1.22e-08 \\ 
\multirow{1}{*}{$(0.9,9)$} & 600 &  0.22 & 3.25e-08 \\ 
 \multirow{1}{*}{$(1,10)$}  & 600 & 0.22 & 1.32e-07 \\ 
 \multirow{1}{*}{$(3,20)$}  & 600 & 0.21 & 1.01e-03 \\ 
\hline
\end{tabular}
\end{table}

\subsection{PDE-Driven Adaptive Distribution (PDAD)}
For a linear problem, the least-squares system constructed from the governing PDE together with its initial and boundary conditions can be written in the matrix form
\begin{equation}\label{eq:matrix_lsq}
\begin{bmatrix}
A_o\\
\eta_b A_b\\
\eta_i A_i
\end{bmatrix}\boldsymbol w
=
\begin{bmatrix}
F\\
\eta_b G_{b}\\
\eta_i H_{i}
\end{bmatrix},
\end{equation}
which we abbreviate as
\begin{equation}
U(\boldsymbol p)\boldsymbol w = \tilde{\bm y},
\end{equation}
here \(U(\boldsymbol p)\in\mathbb{R}^{n\times m_1}\) is the coefficient matrix determined by the parameter vector \(\boldsymbol p\), while the sampling points
\(\{(\boldsymbol x_i,t_i)\}_{i=1}^{n}\) in the spatio-temporal domain are fixed in advance.
\par
We consider the reduced objective function obtained by minimizing the residual with respect to \(\boldsymbol w\)
\begin{equation}\label{eq:min_f(p)}
f_u(\boldsymbol p)
=
\min_{\boldsymbol w}\;
\tfrac12\big\|U(\boldsymbol p)\boldsymbol w-\tilde{\bm y}\big\|_2^2 .
\end{equation}

\par
 When $U(\boldsymbol{p})$ has full column rank, the matrix $U(\boldsymbol{p})^\top U(\boldsymbol{p})$ is symmetric positive definite; see Remark~\ref{rem:On-the-full-rank-assumption} for discussion of this assumption in the randomized setting. Consequently, the inner least-squares problem admits a unique minimizer. In practice, however, the matrix \(U(\boldsymbol p)\) can be severely ill-conditioned, or even effectively
rank-deficient in finite-precision arithmetic.
In such cases, the unregularized least-squares solution (e.g., obtained via QR or SVD-based solvers)
may have a large norm and be highly sensitive to perturbations in \(U(\boldsymbol p)\) or in the data
\(\tilde{\bm y}\).
This sensitivity leads to an excessive dynamic range in \(\boldsymbol w_u(\boldsymbol p)\), potential
overfitting, and numerical instability in the outer optimization with respect to \(\boldsymbol p\).
\par
To stabilize the training stage, we introduce ridge regularization with a parameter \(\lambda>0\) and consider
the regularized reduced objective
\begin{equation}\label{eq:min_ridge_regression}
f_\lambda(\boldsymbol p)
=
\min_{\boldsymbol w}\;
\tfrac12\big\|U(\boldsymbol p)\boldsymbol w-\tilde{\bm y}\big\|_2^2
+
\tfrac{\lambda}{2}\|\boldsymbol w\|_2^2 .
\end{equation}
\par
For any fixed \(\boldsymbol p\) and \(\lambda>0\), the inner problem is strongly convex in \(\boldsymbol w\),
since the matrix \(U(\boldsymbol p)^{\top}U(\boldsymbol p)+\lambda I_{m_1}\) is symmetric positive definite with
condition number
\begin{equation}
\kappa
=
\frac{\sigma_{\max}^2+\lambda}{\sigma_{\min}^2+\lambda},
\end{equation}
where \(\sigma_{\max}\) and \(\sigma_{\min}\) denote the largest and smallest singular values of
\(U(\boldsymbol p)\), respectively.
The unique minimizer is given by
\begin{equation}
\boldsymbol w_\lambda(\boldsymbol p)
=
\big(U(\boldsymbol p)^{\top}U(\boldsymbol p)+\lambda I_{m_1}\big)^{-1}
U(\boldsymbol p)^{\top}\tilde{\bm y} .
\end{equation}
\par
The regularization significantly reduces the sensitivity of \(\boldsymbol w_\lambda(\boldsymbol p)\) to
perturbations in \(\boldsymbol p\) and improves numerical robustness, making the reduced objective
\(f_\lambda(\boldsymbol p)\) suitable for gradient-based optimization methods.
Accordingly, the outer optimization problem
\begin{equation}
\min_{\boldsymbol p}\; f_\lambda(\boldsymbol p),
\end{equation}
is optimized using Adam, SGD, or L-BFGS~(\cite{kingma2014adam,robbins1951stochastic,liu1989limited}).
 Once an optimized parameter vector \(\boldsymbol p\) is obtained, the final unregularized least-squares
system is solved using a QR-based method to recover a high-accuracy solution.

\begin{itemize}
\item Ridge-regularized least squares is employed during parameter training to ensure numerical stability,
while an unregularized QR-based solver is used in the final stage to achieve high accuracy.
\item Since the random seed and training points are fixed, the optimization depends solely on the parameter
vector \(\bm{p}\), and only a small number of optimization steps is typically required.
\end{itemize}
\par
Although the ridge-regularized solution is written in closed form as
\(
\big(U(\bm{p})^{\top}U(\bm{p})+\lambda I_{m_1}\big)^{-1}
U(\bm{p})^{\top}\tilde{\bm y},
\)
no explicit matrix inversion is performed. In practice, $\bm{w}_\lambda(\bm{p})$ is computed by solving an equivalent linear system.
We solve this system either via an augmented least-squares formulation solved by QR factorization, or via MATLAB's backslash operator ({mldivide}), which automatically selects a numerically stable matrix factorization.
\par
For nonlinear problems, the linearized system
\(U(\boldsymbol p)\boldsymbol w=\tilde{\bm y}\) obtained from a nonlinear iterative scheme depends on
the current numerical approximation \(\bm{\hat{u}}\).
Consequently, the accuracy of \(\bm{\hat{u}}\) directly influences the optimization of \(\boldsymbol p\), and
repeated nonlinear iterations and parameter updates are generally unavoidable.

\subsection{Data-Driven Adaptive Distribution (DDAD)}
Suppose that some data (e.g., numerical solution, measurement data) are available. 
Our goal is to determine suitable distribution parameters based on this data.
Consider a numerical solution $\bm{\hat{u}}$
\begin{equation}
\bm{\hat{u}}(\bm{x},t) \approx\bm{w}^{0}\, \bm{\phi}(\bm{p}_0,\bm{x},t),
\end{equation}
where $\bm{p}_0$ is a given initial parameter vector and $\bm{w}^0$ is obtained by a least-squares solver.
Although $\bm{\hat{u}}$ may deviate from the exact solution $\bm{u}$, it typically captures sufficient structural
information to guide the selection of improved parameter values.
\par
Let $A(\bm{p})\in\mathbb{R}^{n\times m_1}$ denote the feature matrix formed by evaluating the basis
functions $\bm{\phi}(\bm{p},\bm{x},t)$ at the prescribed spatio-temporal sampling points.
The goal is to update $\bm{p}$ such that the resulting basis functions provide a better approximation
to the numerical solution $\bm{\hat{u}}$.
This leads to the reduced optimization problem:
\begin{equation}
f_u(\bm{p})
=
\min_{\bm{w}}\; \|A(\bm{p})\bm{w}-\bm{\hat{u}}\|_2^2 .
\end{equation}
\par
To improve numerical stability, we incorporate a ridge regularization term with parameter $\lambda>0$ and
consider the regularized objective:
\begin{equation}\label{eq:ridge_numerical}
f_\lambda(\bm{p})
=
\min_{\bm{w}}\;
\tfrac{1}{2}\|A(\bm{p})\bm{w}-\bm{\hat{u}}\|_2^2
+
\tfrac{\lambda}{2}\|\bm{w}\|_2^2 .
\end{equation}
\par
For any fixed $\bm{p}$, the corresponding minimizer is given by:
\begin{equation}
\bm{w}_\lambda{(\bm{p})}
=
\big(A(\bm{p})^{\top}A(\bm{p})+\lambda I_{m_1}\big)^{-1}
A(\bm{p})^{\top}\bm{\hat{u}}.
\end{equation}
 As in the PDAD case, this expression is used only for analysis and algorithm design; in practical computation, the corresponding linear system is solved directly rather than by forming an explicit inverse.

\par
This property also enables the use of gradient-based optimization algorithms such as Adam to determine a suitable parameter vector.
\par
For any numerically computed solution $\hat{u}_i$ with sufficient accuracy, an appropriate parameter
vector $\bm{p}_i$ can be obtained by minimizing problem~\eqref{eq:ridge_numerical} with respect to $\bm{p}$.
The resulting procedure, referred to as the \emph{Data-Driven Adaptive Distribution} (DDAD),
is particularly effective for problems in which data are naturally available, such as
nonlinear equations solved by iterative schemes, time-dependent PDEs with discrete time-stepping, and data-driven DeepONet.
\par
For convenience, we use the abstract form
\begin{equation}
M(\bm p)\bm w=\bm y,
\end{equation}
where
\begin{equation}
M(\bm p)=
\begin{cases}
U(\bm p), & \text{(PDAD)},\\
A(\bm p), & \text{(DDAD)},
\end{cases}
\qquad
\bm y=
\begin{cases}
\tilde{\bm y}, & \text{(PDAD)},\\
\hat{\bm u}, & \text{(DDAD)}.
\end{cases}
\end{equation}

\begin{algorithm}[!ht]
\caption{Adaptive Distribution (PDAD/DDAD)}
\label{alg:awvp-ridge}
\begin{algorithmic}[1]
\STATE \textbf{Initialization.}
Choose initial distribution parameters $\bm p_0$, training points $\{(\bm x_i,t_i)\}_{i=1}^n$,
a fixed random seed for generating the basic weights $\bm W_0$ and bias parameters $\bm B$,
and a ridge parameter $\lambda>0$.

\FOR{$k=0,1,\dots,N-1$}
  \STATE 
  Build $\bm W(\bm p_k)$ and biases $\bm b(\bm p_k)$ from $\bm W_0,\bm B$ and $\bm p_k$, and form
  \[
    (\bm M(\bm p_k),\bm y)=
    \begin{cases}
      (\bm U(\bm p_k),\,\tilde{\bm y}), & \text{(PDAD)},\\
      (\bm A(\bm p_k),\,\hat{\bm u}), & \text{(DDAD)}.
    \end{cases}
  \]

  \STATE 
  Compute
  $
    \bm w_\lambda
    \gets \arg\min_{\bm w}
    \frac12\|\bm M(\bm p_k)\bm w-\bm y\|_2^2+\frac{\lambda}{2}\|\bm w\|_2^2.
  $

  \STATE 
  Set
  $
    f_\lambda(\bm p_k)
    = \frac12\|\bm M(\bm p_k)\bm w_\lambda-\bm y\|_2^2
     +\frac{\lambda}{2}\|\bm w_\lambda\|_2^2 .
  $

  \STATE Update $\bm p_{k+1}$ by Adam using $\nabla_{\bm p} f_\lambda(\bm p_k)$.
  
\STATE Clip $\bm p_{k+1}$ to $[\bm p_{\min},\bm p_{\max}]$ (component-wise).
\ENDFOR

\STATE Set $\hat{\bm p}\gets \bm p_{N}$.

\STATE \textbf{Final unregularized solve.}
With $\hat{\bm p}$, form $(\bm M(\hat{\bm p}),\bm y)$ as above and solve
\[
\min_{\bm w}\|\bm M(\hat{\bm p})\bm w-\bm y\|_2
\]
by a QR-based solver to obtain the final $\bm w$.
\end{algorithmic}
\end{algorithm}

\begin{rem}\label{rem3.1}
A key computational feature of AD-RaNN is that the adaptive part of the hidden-feature generation is parameterized only by the low-dimensional distribution parameter vector $\bm{p}\in\mathbb{R}^{d}$. Thus, the number of trainable variables in the outer optimization is typically $\mathcal{O}(d)$, which is far smaller than the number of hidden-layer parameters in fully trained neural solvers. Consequently, only a small number of optimization steps is often sufficient in practice. In particular, when positivity of the scaling parameter is required, one may set
$
\bm p=e^{\bm s},
$
and optimize with respect to $\bm s$ instead. This automatically preserves positivity and often improves numerical robustness when the components of $\bm p$ differ significantly in magnitude.
\end{rem}

\begin{rem}\label{rem3.2}
    The outer optimization over the distribution parameter $\bm{p}$ is generally nonconvex due to the nonlinear dependence of the feature matrix on $\bm{p}$. However, $\bm{p}$ is low-dimensional and the ridge-regularized inner problem is strongly convex with a unique solution, which facilitates efficient gradient-based optimization in practice.
\end{rem}
\begin{rem}\label{rem3.4}
To enhance the expressiveness of the random feature representation, the Gaussian
parameterization of the weight matrix can be extended from a single Gaussian
distribution to multiple Gaussian components. Specifically, fixing a random seed,
we first draw a reference matrix $\bm{W}_0 \sim \mathcal{N}(0, I)$. A Gaussian
parameterization with mean $\bm{\mu}$ and variance $\bm{\sigma}^2$ can then be
expressed as
\[
\bm{W} = \bm{\sigma} \odot \bm{W}_0 + \bm{\mu}.
\]
To further enrich the feature space, multiple Gaussian components can be
concatenated as
\[
\bm{W} =
\begin{bmatrix}
\bm{\sigma}^1 \odot \bm{W}_0 + \bm{\mu}^1 \\
\bm{\sigma}^2 \odot \bm{W}_0 + \bm{\mu}^2 \\
\vdots \\
\bm{\sigma}^h \odot \bm{W}_0 + \bm{\mu}^h
\end{bmatrix},
\]
where $\bm{p}=\{\bm{\mu}^i,\bm{\sigma}^i\}$ are trainable parameters.
\end{rem}

\section{AD-RaNN under Different Computational Frameworks}
\label{AD-RaNN-dif-setting}

\subsection{Adaptive Distribution Framework with a Unified Space-Time Formulation}\label{sec:Aw-ST}
This adaptive distribution framework applies to both time-dependent and
time-independent partial differential equations.
When temporal variables are present, they are treated on an equal footing
with spatial variables within a unified space-time formulation.
For time-independent problems, the temporal dimension is absent, and the same
algorithm naturally reduces to a purely spatial setting without any
modification of the overall procedure.

\par
Linear equations can be solved efficiently, whereas nonlinear equations require nonlinear iterative methods (e.g., Picard's iteration or Newton's iteration) to handle the nonlinear terms. When a predefined criterion is met, we update the parameters using Algorithm \ref{alg:awvp-ridge} and re-solve the resulting problem. We repeat this process until the final criterion is satisfied or the total number of iterations exceeds \(K_{\max}\). For time-dependent equations, the space-time frameworks with PDAD and DDAD are named PDAD-ST and DDAD-ST, respectively. 
\par
\paragraph{Stopping criterion.}
Let $E_k$ denote the root-mean-square residual (RMSE) of the PDE operator at iteration $k$ and define $\Delta_k = \frac{\lvert E_k - E_{k-1} \rvert}{\lvert E_{k-1} \rvert}$.
\begin{equation}
    E_k=\sqrt{\frac{1}{n}\sum_{j=1}^{n}\Big(\mathcal{G}(u_k)-f\Big)^2}\,.
\end{equation}
\par
For a predefined $\tau$, we declare stagnation if equation \eqref{eq:dk_tau} is satisfied, and we use this criterion in the following Algorithm \ref{alg:outer-inner_ST}
\begin{equation}\label{eq:dk_tau}
  \Delta_k < \tau.
\end{equation}
\par
For problems with a large number of neural basis functions, a two-stage strategy can be employed to accelerate the training process. In the training stage, only $m_\lambda$ neural basis functions ($m_\lambda < m_1$) and fewer training (collocation) points are used to reduce the computational cost. In the solving stage, the final solution is reconstructed by refitting the output layer using the full set of $m_1$ neural basis functions on the full set of collocation points, thereby preserving the expressive capacity of the model.

\begin{algorithm}[!ht]
\caption{Adaptive Distribution Framework with a Unified Space-Time Formulation}
\label{alg:outer-inner_ST}
\begin{algorithmic}[1]
\REQUIRE Initial parameters $\bm p_0$; thresholds $\tau_i,\tau_o$; initial residual $E_0, E_{\mathrm{out}}$;
maximum number of outer iterations $K_{\max}$.

\STATE Set $k\gets 1$, $\bm p\gets \bm p_0$.
\STATE \textbf{Initialization:} Solve the (linearized) least-squares problem with $\bm p$ (QR) to obtain $E_1$ and $\bm u_1$.

\WHILE{$k < K_{\max}$}
  \STATE Compute $\Delta_k \gets \dfrac{|E_k-E_{k-1}|}{|E_{k-1}|}$.
  \IF{$\Delta_k < \tau_i$}
    \STATE $E_{\mathrm{out}} \gets E_k$.
    \STATE \textbf{Parameter update:} $\bm p \gets \text{Algorithm}~\ref{alg:awvp-ridge}(\bm p)$.
  \ENDIF

  \STATE Perform one nonlinear iteration and linear solve (with current $\bm p$, QR) to obtain $\bm u_{k+1}$.
  \STATE Compute the new residual $E_{k+1}$.
  \STATE Compute $\Delta_{\mathrm{out}} \gets \dfrac{|E_{k+1}-E_{\mathrm{out}}|}{|E_{\mathrm{out}}|}$.
  \STATE $k\gets k+1$.

  \IF{$\Delta_{\mathrm{out}} < \tau_o$}
    \STATE \textbf{break} \COMMENT{global stagnation across outer cycles}
  \ENDIF
\ENDWHILE
\end{algorithmic}
\end{algorithm}

\subsection{Adaptive Distribution in a Discrete-Time Setting}\label{sec:Aw-DT}

For time-dependent PDEs with large temporal scales, the optimal distribution
parameters $\bm{p}$ may vary significantly across time steps. Consequently,
a single fixed set of random parameters may be insufficient to represent the
solution accurately throughout the entire time domain. In this section, we
adaptively determine a parameter vector $\bm{p}_i$ at each time step $t_i$ (or
every few steps).
\par
Due to the linearization introduced by the time-discretization scheme, a numerical solution
$\bm{\hat{u}}_i$ is obtained at every time step. Hence, the parameter vector
$\bm{p}_i$ can be updated directly using either PDAD or DDAD. We refer to the
resulting discrete-time versions of these methods as PDAD-DT and DDAD-DT. Moreover, $E_k$ (see Section~\ref{sec:Aw-ST}) serves as an indicator for determining whether
$\bm{p}$ should be updated at the current step.
\par

Specifically, an update is triggered when the positive residual increase
\begin{equation}\label{eq:DT-judge}
 \Delta_k^{+}
 =
 \frac{(E_k-E_l)_+}{\max\{E_l,\varepsilon_E\}},
 \qquad
 (a)_+ := \max\{a,0\},
\end{equation}
exceeds a prescribed threshold~$\tau_k$, where $E_l$ denotes a reference
residual. In this case, the distribution parameter $\bm p$ is updated and the
current step is re-solved once using the updated randomized features. To avoid
excessive re-solves and possible cycling, at most one parameter update is allowed
within each time step (remaining localized residuals may be further reduced by
the optional layer-growth refinement).

\par
Choosing the reference residual \(E_l\) is important, since it controls the
sensitivity of the update criterion.
If \(E_l\) is not comparable to the current residual \(E_k\), the criterion may
become either too sensitive or too insensitive. To keep \(E_l\) on the same scale
as the current residual, we update it adaptively: if
\begin{equation}\label{eq:DT-reference}
  \Delta_l =
  \frac{|E_k-E_l|}
  {\max\{\min(|E_l|,|E_k|),\varepsilon_E\}}
  > \tau_l .
\end{equation}
then $E_l$ is replaced by $E_k$. Furthermore, if satisfactory accuracy is not achieved after applying distribution-parameter adaptation, a layer-growth strategy can be incorporated into the iteration.

\begin{algorithm}[!ht]
\caption{Adaptive Distribution in a Discrete-Time Setting}
\label{alg:outer-inner_DT}
\begin{algorithmic}[1]
\REQUIRE Initial parameters $\bm p_0$; thresholds $\tau_k,\tau_l,\varepsilon_E$; number of time steps $N_t$; residual $E_0, E_l$.
\STATE Set $k\gets 1$, $\bm p\gets \bm p_0$, $\hat{\bm u}\gets \bm u_0$, $break\_out\gets 0$.
\WHILE{$k \le N_t$}
  \STATE Solve the LSQ problem (QR) at time step $k$.
  \STATE Compute $E_k$ and $\Delta_k^{+}$.
  \IF{$\Delta_k^{+}<\tau_k$ \OR $break\_out=1$}
    \STATE Update $\hat{\bm u}$.
    \STATE $k\gets k+1$.
    \STATE $break\_out\gets 0$.
  \ELSE
    \STATE Update $\bm p \gets \text{Algorithm}~\ref{alg:awvp-ridge} (\bm p)$.
    \STATE $break\_out\gets 1$. \COMMENT{avoid infinite loops}
  \ENDIF
  \STATE Compute $\Delta_l \gets \frac{|E_k-E_l|}
  {\max\{\min(|E_l|,|E_k|),\varepsilon_E\}}$.
  \IF{$\Delta_l>\tau_l$}
    \STATE $E_l\gets E_k$.
  \ENDIF
\ENDWHILE
\end{algorithmic}
\end{algorithm}

\subsection{The adaptive layer growth strategy for local resolution}

The adaptive-distribution mechanism in AD-RaNN is primarily a \emph{global} feature-selection strategy: it adjusts the sampling distribution of randomized hidden features so that the resulting trial space better captures the overall structure of the target solution. However, for problems with sharp layers, narrow transition regions, or strongly localized residual concentrations, global distribution adaptation alone may still be insufficient. In such cases, we employ the adaptive layer-growth procedure of \cite{dang2024Adaptive} as a \emph{local enhancement module}. In the present work, the local basis centers are still selected from collocation points with large residuals, while the local distribution parameter is determined adaptively through the same reduced optimization principle used in AD-RaNN. In this way, the original trial space is enriched by a second family of localized basis functions, providing a residual-driven local correction mechanism on top of the primary global distribution adaptation.

\begin{figure}[!ht]
	\centering
\includegraphics[width=0.7\textwidth]{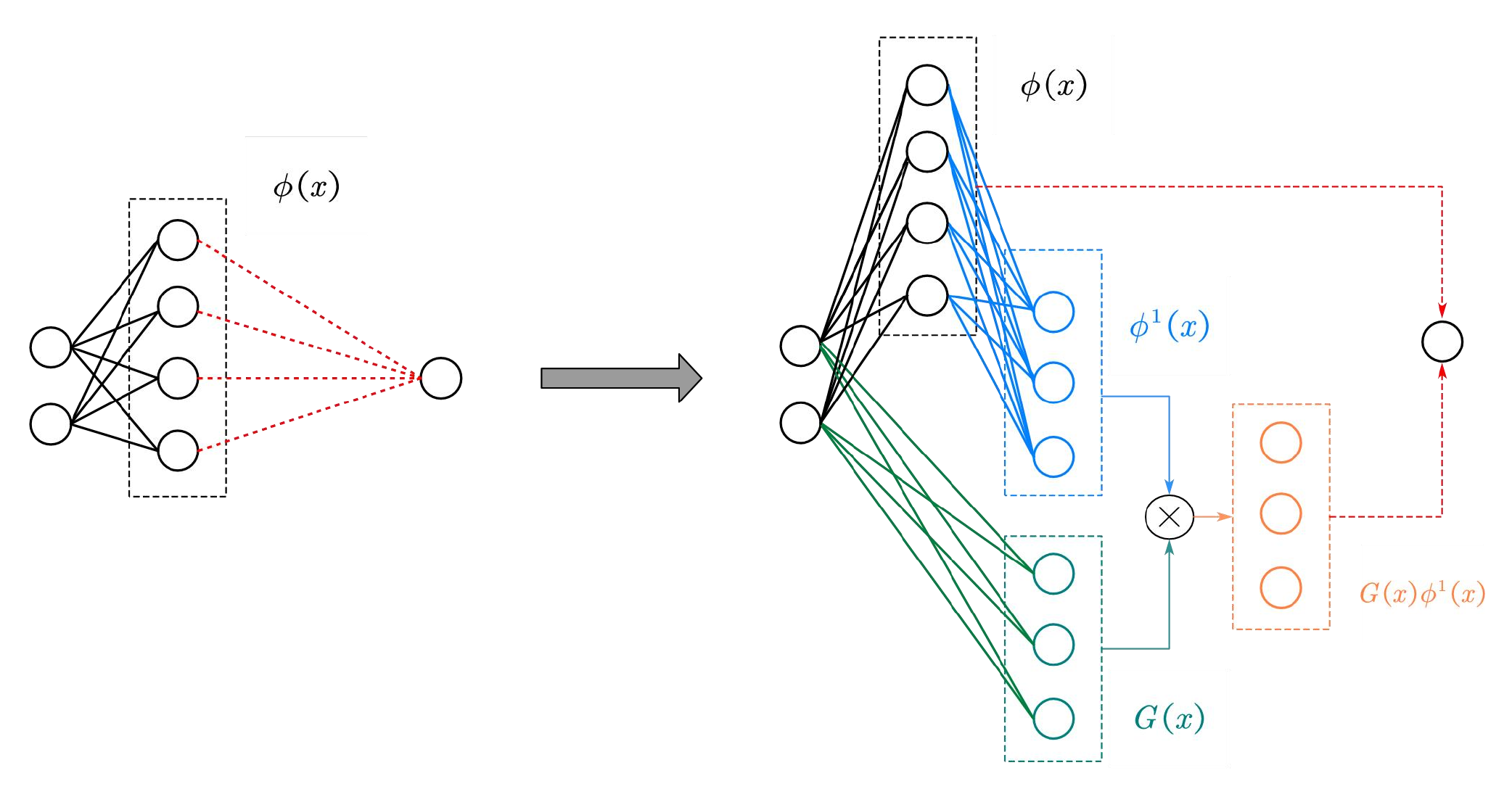}
	\caption{Two-layer RaNN with local basis functions.}
	\label{fig:RNN-two-layer}
\end{figure}
\par
Starting from the collocation points with the largest residuals, we add local basis functions (constructed by the neurons of the new layer) around those points to capture local errors. By combining the original neurons with the newly added ones, the resulting RaNN space is enriched and often yields superior results. The RaNN space after layer growth is defined as
\begin{equation}\label{eq:Vrho1}
\begin{split}
\mathcal{V}_{\rho,1}^{\bm{W},\bm{b}}(D)
=
\Big\{
z(\bm{x})
&=
\bm{w}^{(1)} \bm{\phi}^1(\bm{x})
+
\bm{w}^{(2)} \rho_{2}\!\left(
\bm{W}^{(1)}\bm{\phi}^1(\bm{x})+\bm{b}^{(1)}
\right)\, :  \\
&\quad \bm{x}\in D\subset\mathbb{R}^{d},\;
\forall\,\bm{w}^{(1)}\in\mathbb{R}^{1 \times m_{1}},\;
\forall\,\bm{w}^{(2)}\in\mathbb{R}^{1 \times m_{2}}
\Big\}.
\end{split}
\end{equation}
where $\bm{\phi}^1(\bm{x}) =
\rho_1\left(\bm{W}^{(0)}\bm{x}+\bm{b}^{(0)}\right)$ represents a set of basis functions,
with $\bm{W}^{(0)}$ and $\bm{b}^{(0)}$ determined in the previous step.
The objective is to determine suitable $\bm{W}^{(1)}$ and $\bm{b}^{(1)}$.
\par
After a numerical solution \(\hat{u}_0=\bm{w}^0\bm{\phi}^1\) is obtained using the first layer,
the pointwise residual indicators are computed at the training points, and the top
$m_2$ points
\(X_{\mathrm{err}}=(\bm{x}_1,\ldots,\bm{x}_{m_2})\in\mathbb{R}^{d\times m_2}\)
with the largest errors are recorded. Each point corresponds to a new neuron, and the weights and biases of the second layer are constructed as follows, according to \cite{dang2024Adaptive}.
\begin{equation}
\begin{gathered}
\bm{W}^{(1)} = \bm{H}\,\bm{w}^{0}, \qquad
\bm{b}^{(1)} = -\,\bm{H} \odot [\hat{u}_0(X_{\mathrm{err}})]^{\top},\\
\bm{H} = \sqrt{\bm{H}_{0}\odot \bm{H}_{0}}\,\bm{1}_{d\times 1},\;
\bm{H}_{0} = (\bm{h}_{1}^\top,\ldots,\bm{h}_{m_{2}}^\top)^{\top}\in\mathbb{R}^{m_2\times d},\;
\bm{h}_{i}\sim\mathcal{H}(\bm{p}^{(1)}),\;
\bm{h}_{i}\in\mathbb{R}^{1\times d}.
\end{gathered}
\end{equation}
where $\sqrt{\cdot}$ denotes the element-wise square root.
The activation functions generated by the second layer are first written as
\begin{equation}
\widetilde{\psi}_j(\bm{x})
=
\rho_2\!\left(
\bm{W}^{(1)}_{j}\bm{\phi}^1(\bm{x})+b_j^{(1)}
\right),
\qquad j=1,\ldots,m_2 .
\end{equation}
To enhance locality, we further multiply each second-layer basis function by a
Gaussian localization factor:
\begin{equation}
\begin{gathered}
\psi_j(\bm{x}) = G_j(\bm{x})\widetilde{\psi}_j(\bm{x}),\\
G_j(\bm{x}) =
\exp \left(
- \frac{
\left\| \bm{h}_j \odot (\bm{x}-\bm{x}_j)^{\top} \right\|_2^{2}
}{2}
\right),
\qquad j=1,\ldots,m_2 .
\end{gathered}
\end{equation}

Here $\{\bm{x}_j\}_{j=1}^{m_2}$ is the set of selected large-residual collocation
points, and $G_j(\bm{x})$ is a radial localization function centered at
$\bm{x}_j$. Thus, $\widetilde{\psi}_j$ provides the nonlinear second-layer feature,
while $G_j$ restricts its effective support to a neighborhood of the corresponding
large-error point.
 Once Algorithm~\ref{alg:outer-inner_ST} has been completed without reaching the desired accuracy, we further invoke the layer-growth strategy.
In contrast to manually tuned local basis constructions, the local distribution parameter $\bm p^{(1)}$ is selected adaptively through Algorithm~\ref{alg:awvp-ridge}.
In the DDAD framework, this is implemented by solving a ridge-regularized least-squares problem using the fixed first-layer features together with the newly added local features.
From this perspective, the layer-growth strategy serves as a secondary local refinement mechanism built on top of the primary adaptive-distribution framework.

\subsection{Adaptive distribution RaNN-DeepONet}\label{sec:AD-RaNN-DeepONet}

The preceding sections focus on PDE solvers in which adaptive distribution learning is driven by residual systems or numerically generated solution data. We now show that the same underlying principle extends beyond direct PDE solution and can also be used in operator learning. The purpose of this subsection is therefore not to introduce a disconnected additional model, but to demonstrate that \emph{distribution-level adaptation of randomized hidden features} is a reusable mechanism that transfers naturally from PDE solvers to randomized operator-learning architectures.

Deep operator networks (DeepONets) aim to learn a nonlinear solution operator
\(\mathcal{G}: a \mapsto u\), where \(a\) denotes an input function (e.g., a source term, an initial condition,
or a coefficient field) and \(u(\cdot)\) is the corresponding solution field.
In the vanilla DeepONet, the input function \(a\) is discretized at \(k\) sensors and represented by a vector
\(\bm a\in\mathbb{R}^k\), while the trunk net takes the coordinate \(\bm z\in\mathbb{R}^{d_z}\) as input.
(For spatio-temporal problems, one may take \(\bm z=(\bm{x},t)\).)
The DeepONet output is given by the dot product of the branch and trunk features
\begin{equation}\label{eq:deeponet}
\mathcal{G}(\bm a)(\bm z)=\sum_{i=1}^{{m_t}} \beta_i(\bm a)\, \tau_i(\bm z),
\end{equation}
where \(\beta(\bm a)=(\beta_1,\dots,\beta_{m_t})^\top\in\mathbb{R}^{m_t}\) and \(\tau(\bm z)=(\tau_1,\dots,\tau_{m_t})^\top\in\mathbb{R}^{m_t}\)
are the outputs of the branch and trunk networks, respectively.

\subsubsection{RaNN-DeepONet}

 To reduce the optimization cost of training deep networks, RaNN-DeepONet replaces the branch net by a randomized neural network (RaNN) and uses a fixed randomized trunk feature map (\cite{jiang2026deeponet}).
Concretely, the branch net is a
single hidden layer RaNN
\begin{equation}\label{eq:branch_rann}
\beta(\bm a)=W_b^{2}\,\rho(W_b^{1}\bm a + \bm b_b^{1}),\qquad \beta(\bm a)\in\mathbb{R}^{m_t},
\end{equation}
where \(W_b^{1}\in\mathbb{R}^{m_{b}\times k}\) and \(\bm b_b^{1}\in\mathbb{R}^{m_{b}}\) are randomly initialized and fixed,
and \(W_b^{2}\in\mathbb{R}^{m_t\times m_{b}}\) is the only matrix that needs to be solved.
The trunk net is taken as a fixed feature map
\begin{equation}\label{eq:trunk_fixed}
\tau(\bm z)=\rho(W_t^{1}\bm z+\bm b_t^{1}),\qquad \tau(\bm z)\in\mathbb{R}^{m_t},
\end{equation}
where \(W_t^{1}\in\mathbb{R}^{m_t\times d_z}\) and \(\bm b_t^{1}\in\mathbb{R}^{m_t}\) are randomly initialized and fixed.

With \eqref{eq:deeponet}-\eqref{eq:trunk_fixed}, the RaNN-DeepONet output can be rewritten as a linear model
with respect to the entries of \(W_b^2\).
Let \(\phi_j(\bm a)\) denote the \(j\)-th hidden neuron output in \(\rho(W_b^1\bm a+\bm b_b^1)\),
and let \(\alpha_{ij}\) denote the \((i,j)\)-th entry of \(W_b^2\). Then
\begin{equation}\label{eq:rann_deeponet_linear}
\mathcal{G}(\bm a)(\bm z)=\sum_{i=1}^{m_t} \beta_i(\bm a)\,\tau_i(\bm z)
=\sum_{i=1}^{{m_t}}\sum_{j=1}^{m_{b}}\alpha_{ij}\, \phi_j(\bm a)\, \tau_i(\bm z).
\end{equation}
\par
Given selected pairs \(\{\bm a^{(n)}, (\bm z_j^{(n)},u^{(n)}(\bm z_j^{(n)}))\}_{j=1}^{q}\) for \(n=1,\dots,N\),
the computation of RaNN-DeepONet reduces to a linear least-squares problem for \(\alpha=\{\alpha_{ij}\}\)
\begin{equation}\label{eq:lsq_rann_deeponet}
\min_{\alpha}\; \frac{1}{Nq}\sum_{n=1}^{N}\sum_{j=1}^{q}
\Big(\mathcal{G}(\bm a^{(n)})(\bm z_j^{(n)})-u^{(n)}(\bm z_j^{(n)})\Big)^2.
\end{equation}
\par
The architecture of RaNN-DeepONet is illustrated in Fig. \ref{fig:RaNN-DeepONet}, where the {black solid lines} represent the fixed random mapping, and the {red dashed lines} highlight the learnable parameters solved via least-squares.

\begin{figure}[!ht]
	\centering
    \includegraphics[width=0.7\textwidth]{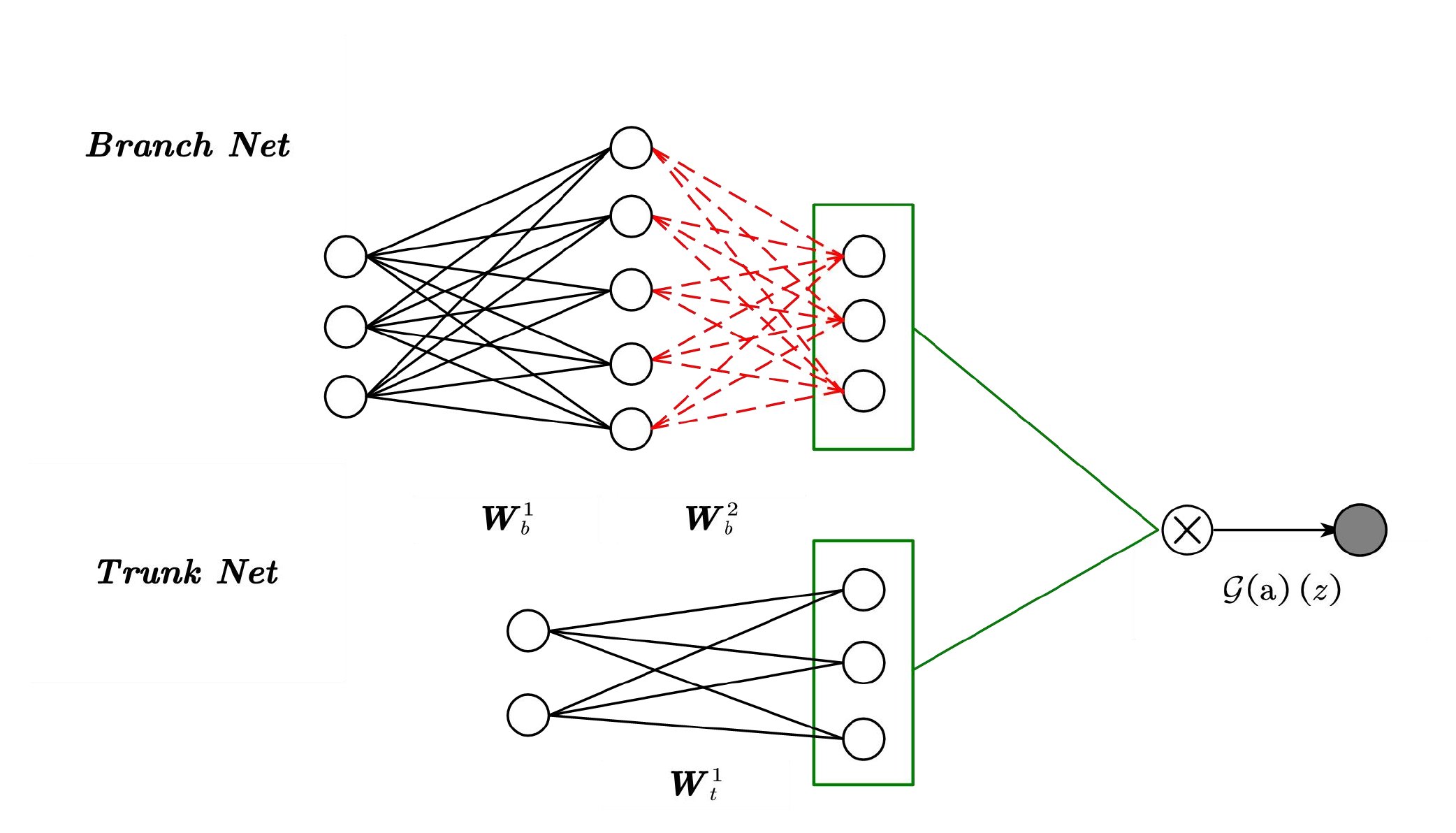}
	\caption{Architecture of RaNN-DeepONet.}
	\label{fig:RaNN-DeepONet}
\end{figure}

\subsubsection{AD-RaNN-DeepONet}

The performance of RaNN-DeepONet depends strongly on the random distributions used to generate the branch and trunk hidden features. This is precisely the same structural issue encountered earlier for randomized neural PDE solvers: once the hidden representation is randomized and frozen, the quality of the approximation is largely determined by how the random feature family is distributed. We therefore apply the same AD-RaNN principle here and replace manually chosen branch/trunk feature distributions by a low-dimensional adaptive distribution parameterization.

Following the adaptive-distribution strategy in AD-RaNN,
we sample
\begin{equation}\label{eq:param_dist}
W_b^{1}\sim U(-r_b,r_b),\qquad W_t^{1}\sim U(-\bm{r}_t,\bm{r}_t),
\end{equation}
and collect the distribution parameters as
\begin{equation}\label{eq:p_rst}
\bm p=[r_b,r_{t_1},r_{t_2},\dots,r_{t_{d_z}}].
\end{equation}
The adaptive distribution RaNN-DeepONet (AD-RaNN-DeepONet) optimizes \(\bm p\) to improve approximation quality. 
After updating \(\bm p\), the output-layer
least-squares problem \eqref{eq:lsq_rann_deeponet} is solved by an unregularized least-squares method to obtain the optimal \(W_b^2\).
Moreover, when only limited supervised data are available while the governing PDE is known, one may incorporate physics-informed constraints and obtain a corresponding AD-PI-RaNN-DeepONet variant. Since this variant follows the same distribution-learning principle, we include it only as an additional demonstration and do not discuss its training mechanism in detail here.
This extension highlights that the contribution of AD-RaNN is not tied to a particular PDE discretization or residual formulation. Rather, its main value lies in a transferable distribution-learning principle for randomized hidden representations, which can be reused in both direct PDE solvers and operator-learning models.

\begin{rem}\label{rem4.1}
 Global basis functions often struggle to approximate solutions with sharp or nearly discontinuous behavior. To address this, we can construct a local randomized neural network (LRaNN), in which the computational domain is partitioned into several subdomains, each equipped with an independent neural network. Each network is then able to capture distinct local features, and the above Algorithm~\ref{alg:outer-inner_ST} can be applied independently to each subnetwork.
\end{rem}
\begin{rem}\label{rem4.2}
For AD-RaNN, a suitable initialization \(\bm{p}_0\) is important. In particular, the Frequency-Based Parameter Initialization Strategy \cite{dang2024Adaptive} can be used. In addition, one may employ a coarse-grid search with a small number of collocation points and basis functions to obtain a suitable initial guess for \(\bm p\), which is then used as \(\bm p_0\).
\end{rem}

\section{Analysis of the Reduced Optimization Problem}
\label{sec:Theoretical-Analysis}

This section gives a basic analysis of the reduced optimization problem
arising in AD-RaNN. The purpose is not to provide an a priori error estimate
for the PDE approximation, but to justify the well-posedness of the reduced
distribution-parameter training problem, the consistency of ridge-regularized
training as the ridge parameter tends to zero, and the gradient formula used in
the outer optimization.

Let $\mathcal P\subset\mathbb R^d$ be the admissible set of distribution
parameters. For $\bm p\in\mathcal P$, let
\[
M(\bm p)\in\mathbb R^{n\times m_1},\qquad \bm y\in\mathbb R^n,
\]
denote the reduced least-squares matrix and the corresponding right-hand side,
where $n$ is the number of sampled equations or data points and $m_1$ is the
number of hidden features. In the usual randomized neural network collocation setting considered here, we
take $n\ge m_1$, so that the reduced systems are square or overdetermined,
and they are typically overdetermined in the numerical experiments.
Here $M(\bm p)=U(\bm p)$ in the PDE-driven case and $M(\bm p)=A(\bm p)$ in the
data-driven case. We define the unregularized and ridge-regularized objectives by
\begin{equation}\label{eq:unregularized_objective_new}
J_u(\bm p,\bm w)
:=
\frac12\|M(\bm p)\bm w-\bm y\|_2^2,
\qquad
f_u(\bm p)
:=
\min_{\bm w\in\mathbb R^{m_1}}J_u(\bm p,\bm w),
\end{equation}
and, for $\lambda>0$,
\begin{equation}\label{eq:ridge_objective_new}
J_\lambda(\bm p,\bm w)
:=
\frac12\|M(\bm p)\bm w-\bm y\|_2^2
+
\frac{\lambda}{2}\|\bm w\|_2^2,
\qquad
f_\lambda(\bm p)
:=
\min_{\bm w\in\mathbb R^{m_1}}J_\lambda(\bm p,\bm w).
\end{equation}

\paragraph{Standing assumptions.}
We use the following assumptions.

\begin{itemize}
\item \textbf{(A1) Compact admissible set.}
The admissible parameter set $\mathcal P\subset\mathbb R^d$ is compact.

\item \textbf{(A2) Continuity.}
The matrix-valued map
\[
\bm p\mapsto M(\bm p)
\]
is continuous from $\mathcal P$ to $\mathbb R^{n\times m_1}$.

\item \textbf{(A3) Differentiability.}
For the gradient formula, we assume that there exists an open set
$\mathcal O\subset\mathbb R^d$ with $\mathcal P\subset\mathcal O$ such that
\[
M\in C^1(\mathcal O;\mathbb R^{n\times m_1}).
\]

\item \textbf{(A4) Uniform full column rank.}
For the unregularized consistency result, we assume that
\[
\operatorname{rank}M(\bm p)=m_1,\qquad \forall \bm p\in\mathcal P.
\]
Equivalently, since $\mathcal P$ is compact and $M$ is continuous, there exists
a constant $\sigma_*>0$ such that
\[
\sigma_{\min}(M(\bm p))\ge \sigma_*,
\qquad \forall \bm p\in\mathcal P.
\]
\end{itemize}

Assumption \textup{(A4)} is not needed for the ridge-regularized inner problem.
It is only needed when the unregularized least-squares minimizer is required to
be unique and uniformly bounded with respect to $\bm p$.

\subsection{Well-posedness of the reduced objectives}

\begin{thm}[Well-posedness and continuity of the ridge-reduced objective]
\label{thm:ridge_wellposed}
Assume \textup{(A1)}--\textup{(A2)}. For every $\lambda>0$ and every
$\bm p\in\mathcal P$, the inner ridge problem
\[
\min_{\bm w\in\mathbb R^{m_1}}
\left\{
\frac12\|M(\bm p)\bm w-\bm y\|_2^2
+
\frac{\lambda}{2}\|\bm w\|_2^2
\right\}
\]
has a unique minimizer $\bm w_\lambda(\bm p)$. Moreover,
\[
\bm w_\lambda(\bm p)
=
\big(M(\bm p)^T M(\bm p)+\lambda I_{m_1}\big)^{-1}
M(\bm p)^T\bm y ,
\]
the map $\bm p\mapsto \bm w_\lambda(\bm p)$ is continuous on $\mathcal P$, and
$f_\lambda$ is continuous on $\mathcal P$. Consequently,
\[
\arg\min_{\bm p\in\mathcal P} f_\lambda(\bm p)\neq\emptyset .
\]
\end{thm}

\begin{proof}
Fix $\lambda>0$ and $\bm p\in\mathcal P$. The Hessian of
$J_\lambda(\bm p,\cdot)$ with respect to $\bm w$ is
\[
M(\bm p)^T M(\bm p)+\lambda I_{m_1}.
\]
Since $\lambda>0$, this matrix is symmetric positive definite, independently
of whether $M(\bm p)$ has full column rank. Hence
$J_\lambda(\bm p,\cdot)$ is strongly convex and admits a unique minimizer.
The first-order optimality condition gives
\[
\big(M(\bm p)^T M(\bm p)+\lambda I_{m_1}\big)\bm w_\lambda(\bm p)
=
M(\bm p)^T\bm y,
\]
which yields the stated formula.

Because $\bm p\mapsto M(\bm p)$ is continuous and matrix inversion is continuous
on the set of nonsingular matrices, the map
$\bm p\mapsto \bm w_\lambda(\bm p)$ is continuous on $\mathcal P$. Therefore
\[
f_\lambda(\bm p)
=
J_\lambda\big(\bm p,\bm w_\lambda(\bm p)\big)
\]
is also continuous on $\mathcal P$. Since $\mathcal P$ is compact by
\textup{(A1)}, the Weierstrass theorem implies that $f_\lambda$ attains its
minimum on $\mathcal P$.
\end{proof}

\begin{thm}[Well-posedness of the unregularized reduced objective]
\label{thm:unregularized_wellposed}
Assume \textup{(A1)}, \textup{(A2)}, and \textup{(A4)}. Then, for every
$\bm p\in\mathcal P$, the unregularized least-squares problem
\[
\min_{\bm w\in\mathbb R^{m_1}}
\frac12\|M(\bm p)\bm w-\bm y\|_2^2
\]
has a unique minimizer
\[
\bm w_u(\bm p)
=
\big(M(\bm p)^T M(\bm p)\big)^{-1}M(\bm p)^T\bm y .
\]
Moreover, $\bm p\mapsto\bm w_u(\bm p)$ and $f_u$ are continuous on
$\mathcal P$, and
\[
\arg\min_{\bm p\in\mathcal P} f_u(\bm p)\neq\emptyset .
\]
\end{thm}

\begin{proof}
By \textup{(A4)}, $M(\bm p)$ has full column rank for every
$\bm p\in\mathcal P$. Hence $M(\bm p)^T M(\bm p)$ is symmetric positive
definite, and the unregularized least-squares problem has the unique minimizer
given by the normal equation
\[
M(\bm p)^T M(\bm p)\bm w_u(\bm p)=M(\bm p)^T\bm y .
\]
The continuity of $M$ and the uniform nonsingularity implied by \textup{(A4)}
give the continuity of $\bm w_u(\bm p)$. Therefore
\[
f_u(\bm p)=J_u\big(\bm p,\bm w_u(\bm p)\big)
\]
is continuous on $\mathcal P$. Since $\mathcal P$ is compact, $f_u$ attains its
minimum on $\mathcal P$.
\end{proof}

\subsection{Consistency of ridge-regularized minimizers}

The next result explains why the ridge-regularized reduced objective is a
consistent stabilization of the unregularized reduced objective as
$\lambda\to0$. The statement is formulated in terms of accumulation points,
because the reduced objective is generally nonconvex and may have multiple
global minimizers.

\begin{lemma}[Uniform comparison between $f_\lambda$ and $f_u$]
\label{lem:uniform_sandwich}
Assume \textup{(A1)}, \textup{(A2)}, and \textup{(A4)}. Then, for every
$\bm p\in\mathcal P$ and every $\lambda>0$,
\begin{equation}\label{eq:sandwich_new}
f_u(\bm p)
\le
f_\lambda(\bm p)
\le
f_u(\bm p)+\frac{\lambda}{2}\|\bm w_u(\bm p)\|_2^2 .
\end{equation}
Moreover, there exists a constant $C_w>0$ such that
\[
\|\bm w_u(\bm p)\|_2\le C_w,\qquad \forall \bm p\in\mathcal P.
\]
Consequently,
\begin{equation}\label{eq:uniform_convergence_flambda}
0\le f_\lambda(\bm p)-f_u(\bm p)\le \frac{\lambda}{2}C_w^2,
\qquad \forall \bm p\in\mathcal P,
\end{equation}
and hence $f_\lambda\to f_u$ uniformly on $\mathcal P$ as $\lambda\to0$.
\end{lemma}

\begin{proof}
For any fixed $\bm p$, the first inequality follows because the ridge objective
contains the nonnegative term $\frac{\lambda}{2}\|\bm w\|_2^2$. For the second
inequality, evaluate the ridge objective at the unregularized minimizer
$\bm w_u(\bm p)$:
\[
f_\lambda(\bm p)
\le
J_\lambda(\bm p,\bm w_u(\bm p))
=
J_u(\bm p,\bm w_u(\bm p))
+
\frac{\lambda}{2}\|\bm w_u(\bm p)\|_2^2
=
f_u(\bm p)
+
\frac{\lambda}{2}\|\bm w_u(\bm p)\|_2^2 .
\]
This proves \eqref{eq:sandwich_new}.

By Theorem~\ref{thm:unregularized_wellposed}, $\bm w_u$ is continuous on the
compact set $\mathcal P$. Therefore $\|\bm w_u(\bm p)\|_2$ is bounded on
$\mathcal P$, which gives the constant $C_w$. The uniform estimate
\eqref{eq:uniform_convergence_flambda} follows immediately.
\end{proof}

\begin{thm}[Consistency of ridge-regularized minimizers]
\label{thm:ridge_consistency}
Assume \textup{(A1)}, \textup{(A2)}, and \textup{(A4)}. Let
$\lambda_k\downarrow0$, and for each $k$ choose
\[
\bm p_k\in\arg\min_{\bm p\in\mathcal P} f_{\lambda_k}(\bm p).
\]
Then every accumulation point of $\{\bm p_k\}$ belongs to
$\arg\min_{\bm p\in\mathcal P} f_u(\bm p)$. In particular, if the unregularized
reduced objective $f_u$ has a unique minimizer $\bm p_u$, then
\[
\bm p_k\to \bm p_u
\qquad \text{as } k\to\infty .
\]
\end{thm}

\begin{proof}
Since $\mathcal P$ is compact, the sequence $\{\bm p_k\}$ has at least one
accumulation point. Let $\bm p_{k_j}\to\bar{\bm p}$ be any convergent
subsequence.

Let $\bm p_u\in\arg\min_{\bm p\in\mathcal P}f_u(\bm p)$, whose existence is
guaranteed by Theorem~\ref{thm:unregularized_wellposed}. By the optimality of
$\bm p_{k_j}$ for $f_{\lambda_{k_j}}$ and by Lemma~\ref{lem:uniform_sandwich},
we have
\[
f_u(\bm p_{k_j})
\le
f_{\lambda_{k_j}}(\bm p_{k_j})
\le
f_{\lambda_{k_j}}(\bm p_u)
\le
f_u(\bm p_u)+\frac{\lambda_{k_j}}{2}C_w^2 .
\]
On the other hand, since $\bm p_u$ minimizes $f_u$,
\[
f_u(\bm p_u)\le f_u(\bm p_{k_j}).
\]
Therefore,
\[
f_u(\bm p_u)
\le
f_u(\bm p_{k_j})
\le
f_u(\bm p_u)+\frac{\lambda_{k_j}}{2}C_w^2 .
\]
Letting $j\to\infty$ gives
\[
f_u(\bm p_{k_j})\to f_u(\bm p_u).
\]
Since $f_u$ is continuous and $\bm p_{k_j}\to\bar{\bm p}$,
\[
f_u(\bar{\bm p})=f_u(\bm p_u).
\]
Thus $\bar{\bm p}\in\arg\min_{\bm p\in\mathcal P}f_u(\bm p)$.

If $f_u$ has a unique minimizer $\bm p_u$, then every convergent subsequence of
$\{\bm p_k\}$ has the same limit $\bm p_u$. Compactness of $\mathcal P$ then
implies convergence of the whole sequence to $\bm p_u$.
\end{proof}

\begin{rem}[Interpretation of the consistency result]
Theorem~\ref{thm:ridge_consistency} is a statement about global minimizers of
the reduced objectives. It does not imply that a practical nonconvex optimizer,
such as Adam or L-BFGS, must find a global minimizer of $f_\lambda$. Rather, it
shows that the ridge-regularized reduced problem is a consistent stabilization
of the unregularized reduced problem at the level of global minimizers.
\end{rem}

\subsection{Gradient formula for the reduced ridge objective}

We now derive the gradient formula used in the outer optimization over
$\bm p$. This formula avoids differentiating the least-squares solution
$\bm w_\lambda(\bm p)$ explicitly.

\begin{thm}[Gradient formula via the envelope theorem]
\label{thm:gradient_formula}
Assume \textup{(A3)}. Let $\lambda>0$. Then $f_\lambda$ admits a $C^1$
extension to $\mathcal O$. For every $\bm p\in\mathcal P$ and
$i=1,\ldots,d$,
\begin{equation}\label{eq:gradient_formula_new}
\frac{\partial f_\lambda}{\partial p_i}(\bm p)
=
\bm r_\lambda(\bm p)^T
\left(
\frac{\partial M(\bm p)}{\partial p_i}
\right)
\bm w_\lambda(\bm p),
\end{equation}
where
\[
\bm r_\lambda(\bm p)
:=
M(\bm p)\bm w_\lambda(\bm p)-\bm y.
\]
\end{thm}

\begin{proof}
For fixed $\lambda>0$, the ridge minimizer is
\[
\bm w_\lambda(\bm p)
=
\big(M(\bm p)^T M(\bm p)+\lambda I_{m_1}\big)^{-1}
M(\bm p)^T\bm y .
\]
Since $M\in C^1(\mathcal O;\mathbb R^{n\times m_1})$ and the matrix
$M(\bm p)^T M(\bm p)+\lambda I_{m_1}$ is uniformly positive definite,
$\bm w_\lambda$ is continuously differentiable with respect to $\bm p$.

Using
\[
f_\lambda(\bm p)=J_\lambda(\bm p,\bm w_\lambda(\bm p)),
\]
the chain rule gives
\[
\frac{\partial f_\lambda}{\partial p_i}(\bm p)
=
\frac{\partial J_\lambda}{\partial p_i}
\big(\bm p,\bm w_\lambda(\bm p)\big)
+
\nabla_{\bm w}J_\lambda
\big(\bm p,\bm w_\lambda(\bm p)\big)^T
\frac{\partial \bm w_\lambda}{\partial p_i}(\bm p).
\]
The second term vanishes because $\bm w_\lambda(\bm p)$ is the unique minimizer
of $J_\lambda(\bm p,\cdot)$ and therefore satisfies
\[
\nabla_{\bm w}J_\lambda
\big(\bm p,\bm w_\lambda(\bm p)\big)=\bm 0.
\]
It remains to compute the explicit partial derivative with respect to $p_i$,
holding $\bm w$ fixed. Since the regularization term does not depend explicitly
on $\bm p$,
\[
\frac{\partial J_\lambda}{\partial p_i}(\bm p,\bm w)
=
\big(M(\bm p)\bm w-\bm y\big)^T
\left(
\frac{\partial M(\bm p)}{\partial p_i}
\right)\bm w .
\]
Setting $\bm w=\bm w_\lambda(\bm p)$ gives
\eqref{eq:gradient_formula_new}.
\end{proof}

\begin{rem}\label{Practical evaluation}
When the derivatives $\partial M(\bm p)/\partial p_i$ are available
analytically, formula \eqref{eq:gradient_formula_new} provides an efficient
gradient evaluation without differentiating through the linear solver. If these
derivatives are expensive or unavailable, they can be approximated by finite
differences. In that case, the formula should be interpreted as the target
continuous gradient, while the implemented gradient is its finite-difference
approximation.
\end{rem}

\subsection{A practical lower-bound estimate for the ridge parameter}

The ridge parameter $\lambda$ is used to stabilize the inner least-squares solve
during distribution-parameter training. Mathematically, any $\lambda>0$ makes
$M(\bm p)^T M(\bm p)+\lambda I$ positive definite. Numerically, however, an
excessively small $\lambda$ may not sufficiently improve conditioning when
$M(\bm p)$ is severely ill-conditioned.

Let $\epsilon_{\rm mach}$ denote the machine precision. A common practical
criterion is to require
\begin{equation}\label{eq:cond_criterion_new}
\kappa_2\!\left(M^T M+\lambda I\right)\epsilon_{\rm mach}
<
\tau_\lambda ,
\end{equation}
where $\tau_\lambda\in(0,1)$ is a prescribed tolerance. The following estimate
does not rely on Assumption~\textup{(A4)}; it is intended to quantify the
stabilizing effect of ridge regularization even when the least-squares matrix is
ill-conditioned or rank deficient.

\begin{thm}[Conditioning-based lower bound for $\lambda$]
\label{thm:lambda_lower_bound}
Let $M\in\mathbb R^{n\times m_1}$ be fixed, with $n\ge m_1$. Let the singular
values of $M$ be
\[
\sigma_1\ge \sigma_2\ge\cdots\ge \sigma_{m_1}\ge0,
\]
and set $\sigma_{\max}=\sigma_1$ and $\sigma_{\min}=\sigma_{m_1}$. In this
conditioning estimate, $\sigma_{\min}=0$ is allowed, corresponding to a
rank-deficient least-squares matrix.
Assume $\tau_\lambda>\epsilon_{\rm mach}$. If
\begin{equation}\label{eq:lambda_lower_bound_new}
\lambda > \max\left\{
0,\,
\frac{
\epsilon_{\rm mach}\sigma_{\max}^2
-
\tau_\lambda\sigma_{\min}^2
}{
\tau_\lambda-\epsilon_{\rm mach}
}
\right\},
\end{equation}
then
\[
\kappa_2(M^T M+\lambda I)\epsilon_{\rm mach}<\tau_\lambda .
\]
In particular, in the nearly rank-deficient case
$\sigma_{\min}^2\approx0$, this bound becomes
\begin{equation}\label{eq:lambda_simplified_new}
\lambda
>
\frac{\epsilon_{\rm mach}}{\tau_\lambda-\epsilon_{\rm mach}}
\sigma_{\max}^2
\approx
\frac{\epsilon_{\rm mach}}{\tau_\lambda}
\sigma_{\max}^2,
\end{equation}
where the approximation uses $\epsilon_{\rm mach}\ll\tau_\lambda$.
\end{thm}

\begin{proof}
Since the eigenvalues of $M^T M$ are $\sigma_i^2$, the eigenvalues of
$M^T M+\lambda I$ are $\sigma_i^2+\lambda$. Hence, for $\lambda>0$,
\[
\kappa_2(M^T M+\lambda I)
=
\frac{\sigma_{\max}^2+\lambda}{\sigma_{\min}^2+\lambda}.
\]
Therefore the condition
\[
\kappa_2(M^T M+\lambda I)\epsilon_{\rm mach}<\tau_\lambda
\]
is equivalent to
\[
\epsilon_{\rm mach}(\sigma_{\max}^2+\lambda)
<
\tau_\lambda(\sigma_{\min}^2+\lambda).
\]
Since $\tau_\lambda>\epsilon_{\rm mach}$, this is equivalent to
\[
\lambda
>
\frac{
\epsilon_{\rm mach}\sigma_{\max}^2
-
\tau_\lambda\sigma_{\min}^2
}{
\tau_\lambda-\epsilon_{\rm mach}
}.
\]
Combining this inequality with the requirement $\lambda>0$ for ridge
regularization gives \eqref{eq:lambda_lower_bound_new}.

If $\sigma_{\min}^2\approx0$, the bound reduces to
\[
\lambda
>
\frac{\epsilon_{\rm mach}}{\tau_\lambda-\epsilon_{\rm mach}}
\sigma_{\max}^2 .
\]
Since typically $\epsilon_{\rm mach}\ll\tau_\lambda$, this yields the practical
approximation
\[
\frac{\epsilon_{\rm mach}}{\tau_\lambda-\epsilon_{\rm mach}}
\sigma_{\max}^2
\approx
\frac{\epsilon_{\rm mach}}{\tau_\lambda}
\sigma_{\max}^2.
\]
\end{proof}

\begin{rem}[Role of the lower bound]
The estimate in Theorem~\ref{thm:lambda_lower_bound} should be understood as a
practical conditioning guideline rather than an approximation-error optimal
choice of $\lambda$. A larger $\lambda$ generally improves conditioning but
introduces stronger bias in the reduced training objective. Therefore, in the
two-stage AD-RaNN strategy, ridge regularization is used only during the
distribution-parameter training stage, while the final output coefficients are
obtained by solving the unregularized least-squares problem with the learned
distribution parameter.
\end{rem}

\begin{rem}\label{rem:On-the-full-rank-assumption}
Assumption \textup{(A4)} is a nondegeneracy condition on the admissible
distribution-parameter set. In randomized feature constructions, full column
rank is often expected for generic collocation sets and nondegenerate random
features, but it is not guaranteed merely because the sampled hidden parameters
are distinct. Nearly dependent features can still lead to severe ill-conditioning
in finite precision. Ridge regularization is introduced precisely to stabilize
the inner problem in such regimes. The consistency result above uses
\textup{(A4)} because it compares the ridge problem with the unregularized
least-squares problem through the unique and uniformly bounded minimizer
$\bm w_u(\bm p)$.
\end{rem}

\begin{rem}[Scope of the analysis]
The results in this section concern the reduced optimization problem generated
after collocation, discretization, or data sampling. They do not by themselves
constitute a PDE-level convergence theory for AD-RaNN. Such a theory would also
need to quantify the approximation power of the randomized feature space, the
effect of collocation or quadrature, the stability of the underlying PDE
formulation, and the behavior of the nonconvex outer optimizer. The present
analysis instead justifies the reduced ridge training problem and its use as a
stable distribution-learning mechanism.
\end{rem}

\section{Numerical Experiments}\label{sec:Numerical_experiment}

Throughout this section, the hidden parameters are sampled from the uniform distribution $\mathcal{U}(-\bm{r},\bm{r})$, and the activation function is
taken to be the Gaussian function, except for the Black-Scholes equation and the
diffusion-reaction dynamics solved by the AD-RaNN-DeepONet model, where the $\tanh$
activation is used.
Unless otherwise stated, all neural-network experiments are implemented in MATLAB
and conducted on a workstation equipped with an Intel(R) Xeon(R) Gold 6240 CPU
and an NVIDIA RTX A6000 GPU. The results are visualized using MATLAB and Python
with Matplotlib.

\par
The accuracy is measured by the relative $\ell_2$ error (hereafter referred to as the $\ell_2$ error):
\begin{equation}
\ell_2 \text{ error} = \sqrt{\frac{\sum_{j=1}^{n_p} |\hat{u}(x_j) - u_r(x_j)|^2}{\sum_{j=1}^{n_p} |u_r(x_j)|^2}}
\end{equation}
where $\hat{u}$ denotes the computed solution and $u_r$ is either the exact
solution (when available) or a high-accuracy reference solution obtained by a
traditional numerical method.

Unless otherwise stated, the compared randomized neural network methods within each experiment use the same collocation sets, the same residual formulation, and comparable numbers of hidden features, so that the reported differences primarily reflect the effect of distribution adaptation rather than changes in the underlying discretization setup.

\par
To avoid the memory and computational cost of backpropagation, all derivatives
of the Gaussian activation function
\begin{equation}\label{eq:activation_fun}
    \bm{\phi} = \exp\!\big(-(\bm{W} \bm{x} + \bm{b})^2\big),
\end{equation}
are computed analytically. Let $\bm{z} = \bm{W} \bm{x} + \bm{b}$. All operations on $\bm{z}$ are understood element-wise. Then
\begin{equation}
    \bm{\phi} = \exp\left(-\bm{z}^2\right).
\end{equation}
The first and second derivatives are given by
\begin{align}
\frac{\partial \bm{\phi}}{\partial x_i}
  &= -2 \bm{W}_i \bm{z}\,\bm{\phi}, \\
\frac{\partial^2 \bm{\phi}}{\partial x_i^2}
  &= 2 \bm{W}_i^2 \,(2\bm{z}^2 - 1)\,\bm{\phi}, \\
\frac{\partial^2 \bm{\phi}}{\partial x_i \partial x_j}
  &= 2 \bm{W}_i \bm{W}_j \,(2\bm{z}^2 - 1)\,\bm{\phi}.
\end{align}
\par
Unless otherwise stated, all penalty parameters $\eta_i$ and $\eta_b$ are
set to $100$. Nonlinear iterations are performed using Newton's iteration. The
layer-growth strategy (denoted by LG) is employed to enhance local resolution
near steep gradients. The reported LG time is the sum of the baseline runtime (AD-RaNN or Fre-based) and the additional time incurred by the layer-growth strategy.
\par
    For the discrete-time framework,  the convergence
order in time is computed as
\begin{equation}
    \mathrm{order}
    \approx
    \frac{
      \log\!\left( E_{N_t} / E_{2N_t} \right)
    }{
      \log\!\left( \Delta t_{N_t} / \Delta t_{2N_t} \right)
    },
\end{equation}
where $E_{N_t}$ denotes the numerical error evaluated with $N_t$ time steps.
\par
Random sampling is performed using MATLAB's \texttt{rand} function. The discrete
Fourier transform (DFT) in the frequency-based (Fre-based) initialization is
computed using MATLAB's \texttt{fft} function on a $100\times 100$ sampling
grid.
\par
In a two-dimensional domain $D = (x_{11},x_{12})\times (x_{21},x_{22})$, a tensor grid of
training points with resolution $n_1\times n_2$ is generated as
\begin{equation}
\begin{aligned}
x_{1k} &= x_{11} + \epsilon + (k-1)\frac{x_{12}-x_{11}-2\epsilon}{n_1-1}, \\
x_{2l} &= x_{21} + \epsilon + (l-1)\frac{x_{22}-x_{21}-2\epsilon}{n_2-1},
\end{aligned}
\end{equation}
where $\epsilon$ is a small offset introduced to avoid boundary effects, and
$\bm{x}_{k,l} = (x_{1k},x_{2l})$ denotes the collocation points. The quantity
$m_i$ denotes the number of neurons in the $i$-th hidden layer. 
\par
We use $m_\lambda$ to denote the number of neurons employed during the training stage. Unless otherwise specified, $m_\lambda$ is taken to be equal to $m_1$, the number of neurons used in the solution stage. Moreover, the maximum number of nonlinear iterations is set to $K_{\max}=15$.

\subsection{Effectiveness of proposed methods}\label{sec:Effectiveness_of proposed_methods}
In this section, we consider three test cases with exact solutions:
\begin{itemize}
    \item \textbf{Case 1}: Compares PDAD with the PINN, SNN, Fre-based, and Evolution method, demonstrates the effectiveness of AD-RaNN, and investigates the influence of the ridge regularization parameter $\lambda$.
    \item \textbf{Case 2}: Combines layer growth for nonlinear PDEs.
    \item \textbf{Case 3}: Presents the performance of PDAD-DT and DDAD-DT on time-dependent nonlinear PDEs.
\end{itemize}
The related hyperparameters, including the learning rate, the number of collocation points,
the number of training iterations, and the number of neurons $m_1$, are summarized in Table~\ref{tab:sub1_para}.

\begin{table}[!ht]
\centering
\caption{Hyperparameter settings for Section~\ref{sec:Effectiveness_of proposed_methods}.}
\label{tab:sub1_para}
\begin{tabular}{|c|c|c|c|c|}
\hline
Case & Learning rate & Collocation points & Training iterations & Neurons ($m_1$) \\
\hline
1 &1 & $30\times 80$  & 20 & 600 \\
2 & 0.5 & $100\times 100$& 10 & 2000, 4000 \\
3 & 0.5 & $80\times 50$  & 15 & 1200 \\
\hline
\end{tabular}
\end{table}

\par
\textbf{Case 1: Two-dimensional linear Poisson equation}

In Section~\ref{sec:example}, we examined the effect of manually chosen distribution parameters for the Poisson equation.
Here, we further consider the same two-dimensional Poisson problem~\eqref{eq:poisson} on $\Omega=(-1,1)^2$ and compare PDAD with several representative baselines, including PINN, SNN, the Fre-based strategy, and a differential-evolution-based parameter search.

A suitable initial guess for $\bm p$ can significantly reduce the computational cost.
Following Remark~\ref{rem4.2}, we adopt the frequency-based initialization strategy, which gives $\bm p_{\rm Fre}=(r_x,r_y)=(6.00,12.00)$.
The Fre-based method uses this parameter directly without further adaptation.
For the evolution-based baseline, the same Fre-based parameter is used to initialize the population, and the differential-evolution search is performed for 20 generations.
For PDAD, the same parameter is used as the initial value of the reduced distribution optimization.

The results in Table~\ref{tab:poisson_mannal_PDAD} show that PDAD consistently improves upon the fixed Fre-based strategy.
With the same number of randomized basis functions, the Fre-based method yields a relative $\ell_2$ error of $2.68\times10^{-3}$, whereas PDAD reduces the error to $3.63\times10^{-9}$ when $\lambda=10^{-7}$.
The evolution-based search also identifies a substantially better distribution parameter than the fixed Fre-based choice, but at a noticeably higher runtime than PDAD in this test.

We also include PINN and SNN as representative neural-network baselines under the training settings specified below Table~\ref{tab:poisson_mannal_PDAD}.
Specifically, PINN is trained with 3000 Adam iterations followed by 3000 L-BFGS iterations, while SNN is trained with 1200 Adam iterations.
Under these settings, PDAD achieves higher accuracy and a shorter runtime in this example.
Although such cross-method comparisons inevitably depend on implementation details and hyperparameter choices, the results still indicate that distribution-level adaptation can efficiently identify effective anisotropic distribution parameters for this problem.
The results also suggest that, in this example, smaller ridge parameters tend to yield more accurate PDAD solutions.

\begin{table}[!ht]
\centering
\caption{Comparison of PINN, SNN, Fre-based, differential-evolution-based, and PDAD methods for the Poisson problem~\eqref{eq:poisson}. 
For the Fre-based, differential-evolution-based, and PDAD methods, we report the hidden-layer neurons (HLN), distribution parameters $\bm p=(r_x,r_y)$, runtime (s), and relative $\ell_2$ error; the ridge parameter $\lambda$ is reported only for PDAD. For PINN and SNN, the HLN column reports the network architecture, together with the corresponding runtime and relative $\ell_2$ error.}
\label{tab:poisson_mannal_PDAD}
\setlength{\tabcolsep}{12pt}
\renewcommand{\arraystretch}{1.2}
\begin{tabular}{cccccc}
\hline
Method & $\lambda$ & HLN & $\bm{p}=(r_x,r_y)$ & time (s) & Relative $\ell_2$ error \\

\hline
PINN & -- & $100\times100\times100$ & -- & 145.69 & 4.15e-03 \\
\hline
SNN & -- & $100\times100\times100\times600$ & -- & 68.44 & 8.83e-05 \\
\hline
Fre-based & -- & 600 & $(6.00,12.00)$ & 2.53 & 2.68e-03 \\
\hline
Evolution & -- & 600 & $(1.42,5.44)$ & 16.92 &  1.17e-08 \\
\hline
\multirow{4}{*}{PDAD}
& 1e-01 & 600 & $(2.82,7.02)$ & 5.02 & 8.04e-07 \\
& 1e-03 & 600 & $(1.15,5.44)$ & 5.04 & 1.07e-08 \\
& 1e-05 & 600 & $(1.46,6.16)$ & 5.05 & 5.33e-09 \\
& 1e-07 & 600 & $(1.31,6.19)$ & 5.04 & 3.63e-09 \\
\hline
\end{tabular}
\end{table}
\par

\textbf{Case 2: Nonlinear PDE with a sharp layer}
\par
In this case, the PDAD-ST and DDAD-ST methods are applied to the nonlinear equation~\eqref{eq:2d_non_para_lg} on $\Omega=(0,1)^2$
\begin{equation}\label{eq:2d_non_para_lg}
\begin{cases}
-u_{xx}-u_{yy}+u^2 = f(x,y), & (x,y)\in\Omega,\\[2pt]
u(x,y) = g(x,y),           & (x,y)\in\partial\Omega,
\end{cases}
\end{equation}
with the exact solution
\begin{equation}\label{eq:2d_sharp}
\begin{cases}
u(x,y)= \tanh\!\left( \dfrac{r_0 - \sqrt{(x - 0.5)^2 + (y - 0.5)^2}}{\epsilon} \right),\\
r_0=0.5,\quad \epsilon=0.04.
\end{cases}
\end{equation}
\par
This example features a sharp internal layer, which makes it suitable for
testing the effectiveness of the layer-growth strategy.
\par
To assess the performance of the adaptive layer growth strategy (LG), we keep the parameters in the first hidden layer fixed and construct a second hidden layer
consisting of nonlinear local basis functions. The parameters of these local
basis functions are determined based on points with large residuals combined
with adaptive training. The corresponding numerical results are summarized in
Table~\ref{tab:2d_sharp}, and representative solution plots are shown
in Fig.~\ref{fig.case3}.

\begin{table}[!ht]
\centering
\caption{Comparison of the  PDAD-ST and DDAD-ST for the sharp-layer
problem \eqref{eq:2d_sharp}. The table reports the parameters $\bm{p}=(r_x,r_y)$, computational time,
and $\ell_2$ errors, together with the layer-growth results using $m_2=300$.}
    \label{tab:2d_sharp}
\setlength{\tabcolsep}{10pt} 
\begin{tabular}{cccccccc}
\hline
& $m_1$ & $\bm{p}=(r_x,r_y)$ & time (s) & $\ell_2$ error & $r_2$ &LG time (s) &LG $\ell_2$ error \\ \hline

\multirow{2}{*}{PDAD-ST} & 2000 & $(18.36,22.95)$ &29.06&1.92e-03&2.43&38.67&3.13e-05 \\ 
& 4000 & $(24.49, 28.50)$ &153.32&5.85e-05&4.45& 184.75 & 3.39e-06 \\ \hline
\multirow{2}{*}{DDAD-ST} & 2000 & $(17.67,17.41)$ & 11.64 & 2.00e-03&1.22&21.21&1.96e-04 \\ 
& 4000 & $(33.82, 31.40)$ & 49.76 & 9.66e-05 &2.95&84.95&4.88e-06 \\ \hline
\end{tabular}
\end{table}

\begin{figure}[!ht]
	\centering
	\includegraphics[width=0.3\textwidth]{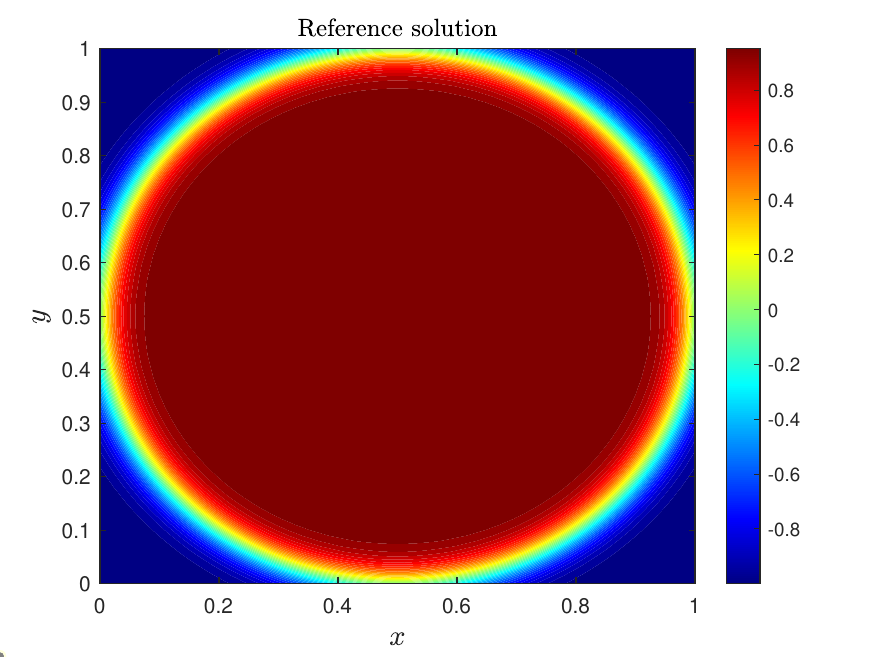}
	\includegraphics[width=0.3\textwidth]{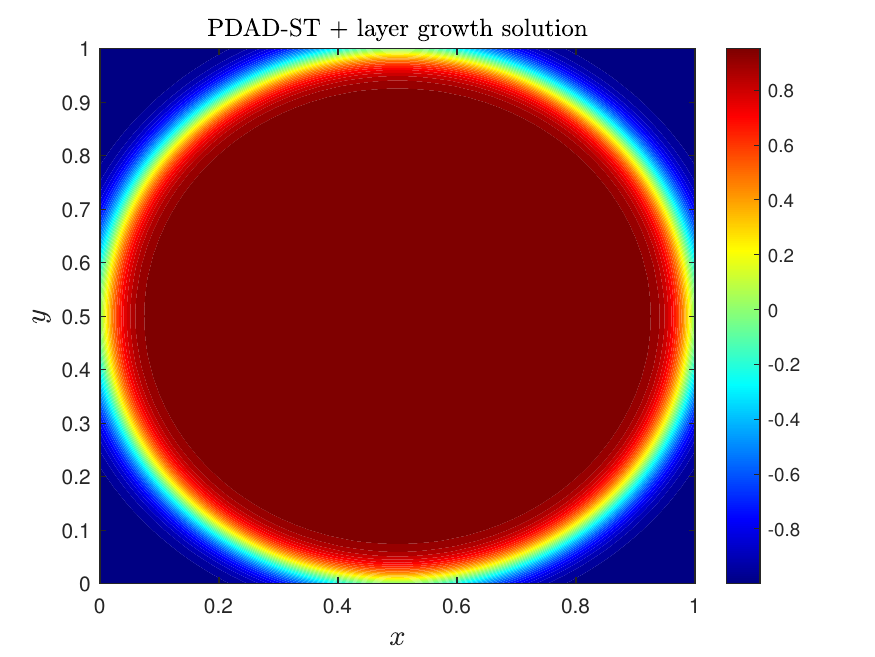}
     \includegraphics[width=0.3\textwidth]{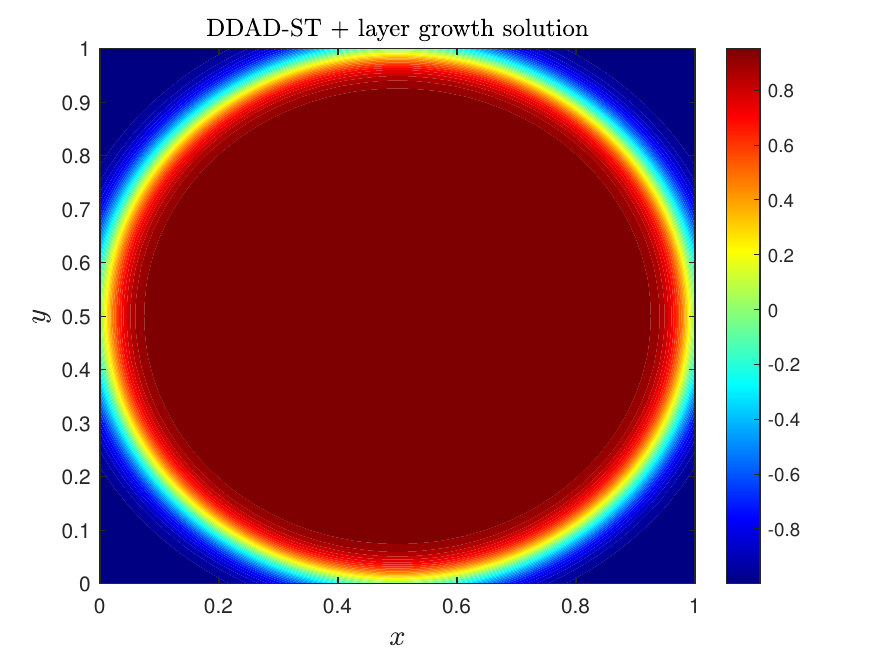}  
	\caption{Reference solution and numerical solutions for the sharp-layer problem~\eqref{eq:2d_sharp}. From left to right: reference solution, PDAD-ST solution, and DDAD-ST solution with $m_1=4000$.}
	\label{fig.case3}
\end{figure}
\par
\textbf{Case 3: Two-dimensional time-dependent nonlinear equation with a large gradient}
\par
In this time-dependent nonlinear example, we employ IMEX-BDF2, treating the
nonlinear term using the previously computed solution, and apply PDAD-DT and
DDAD-DT directly on the domain
\[
\Omega\times(0,T] = (0,1)\times(-1,1)\times(0,1.5].
\]
\par
 We consider the following nonlinear problem with a large
spatio-temporal gradient:
\begin{equation}\label{eq:td_non}
\begin{cases}
u_{t}-\Delta u+u^2 = f(x,y,t), & (x,y,t)\in\Omega\times(0,T],\\[2pt]
u(x,y,t) = g(x,y,t),           & (x,y,t)\in\partial\Omega\times(0,T],\\[2pt]
u(x,y,0) = h(x,y),           & (x,y)\in \Omega,
\end{cases}
\end{equation}
where the exact solution is given by
\begin{equation}
u(x,y,t) = \bigl(1 - y^2\bigr)
           \exp\!\left( \frac{1}{(2x - 1)^2 + e^{-t}} \right).
\end{equation}
\par
The results for PDAD-DT and DDAD-DT are reported in
Table~\ref{tab:td_non}. The column 'order' denotes the observed convergence
order, which is close to $2$ in all refinements and thus confirms the
second-order accuracy of the IMEX-BDF2 scheme. Due to the large temporal range,
a single fixed parameter vector $\bm{p}$ is generally not suitable for all time
steps. Therefore, an adaptive evolution of $\bm{p}$ in time is essential. The
evolution of $\bm{p}=(r_x,r_y)$ for PDAD-DT and DDAD-DT at different time step
numbers $N_t=100,200,400,800$ is illustrated in Fig.~\ref{fig.TD-evo-p}. To reduce the training cost, we set $m_{\lambda}=600$.
\par
As shown in Table~\ref{tab:td_non}, both PDAD-DT and DDAD-DT achieve an observed
temporal convergence order close to $2$ under successive time-step refinement.
This indicates that the proposed adaptive distribution strategies do not destroy the empirical second-order temporal accuracy of the underlying IMEX-BDF2 scheme.
Moreover, the nearly identical errors obtained by PDAD-DT and DDAD-DT indicate
that adaptive parameter evolution in time does not compromise numerical accuracy.

\begin{figure}[!ht]
	\centering
	\includegraphics[width=0.24\textwidth]{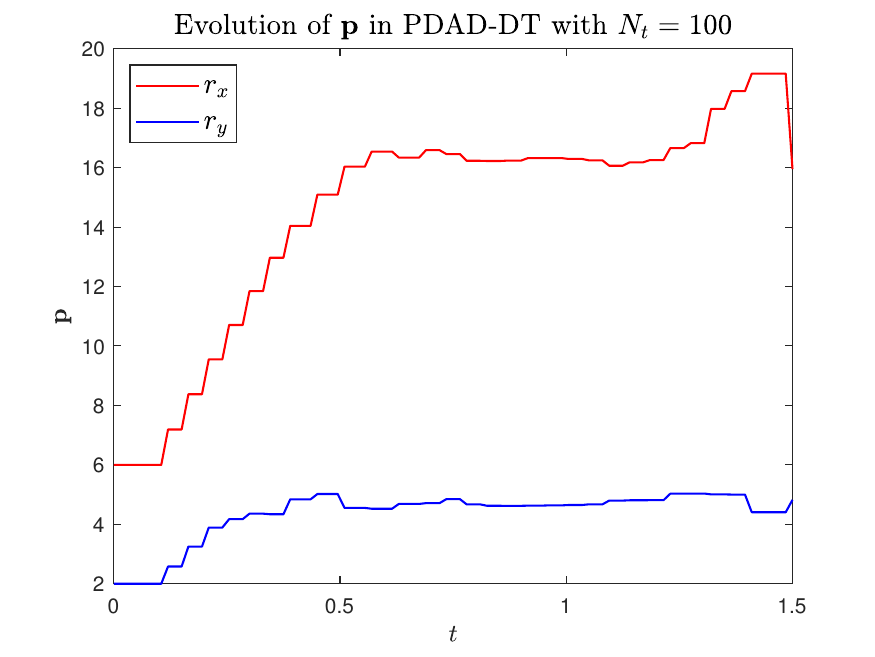}
	\includegraphics[width=0.24\textwidth]{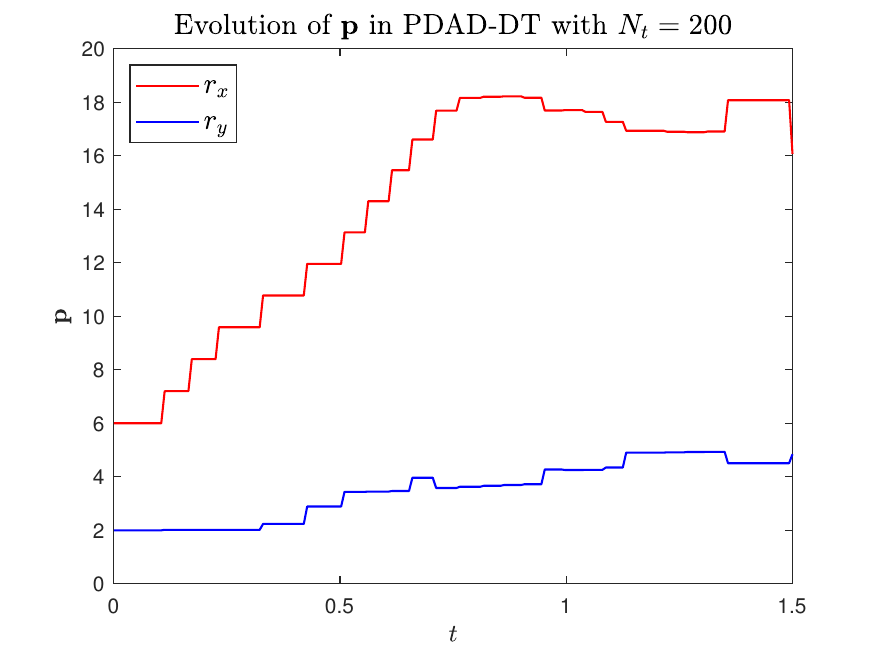}
	\includegraphics[width=0.24\textwidth]{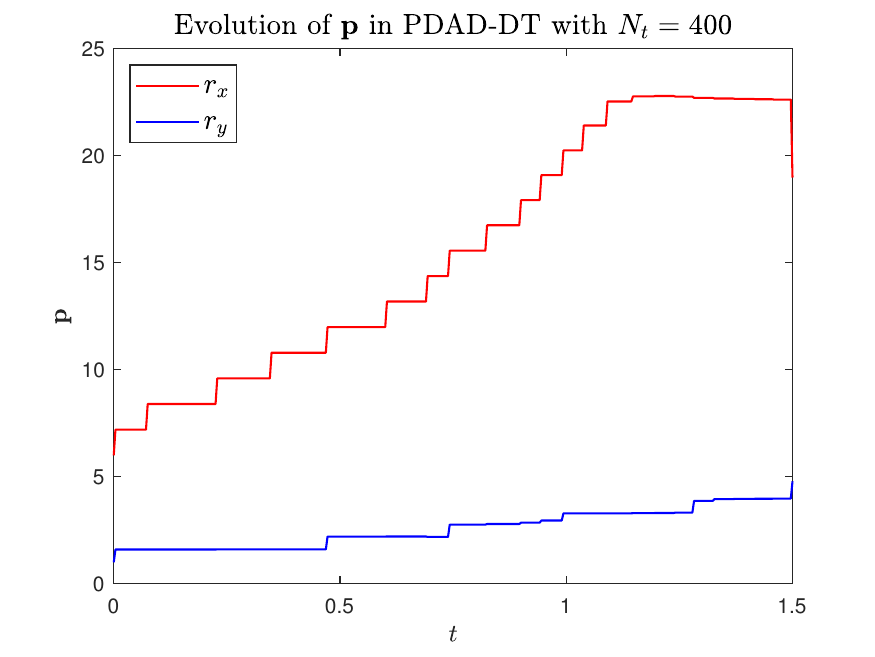}
     \includegraphics[width=0.24\textwidth]{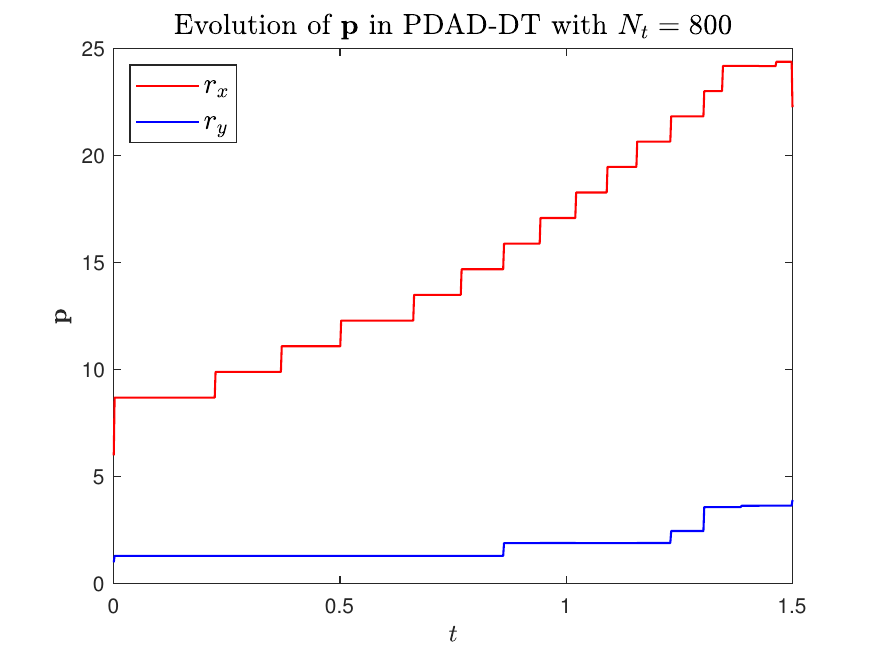}  \\
	\includegraphics[width=0.24\textwidth]{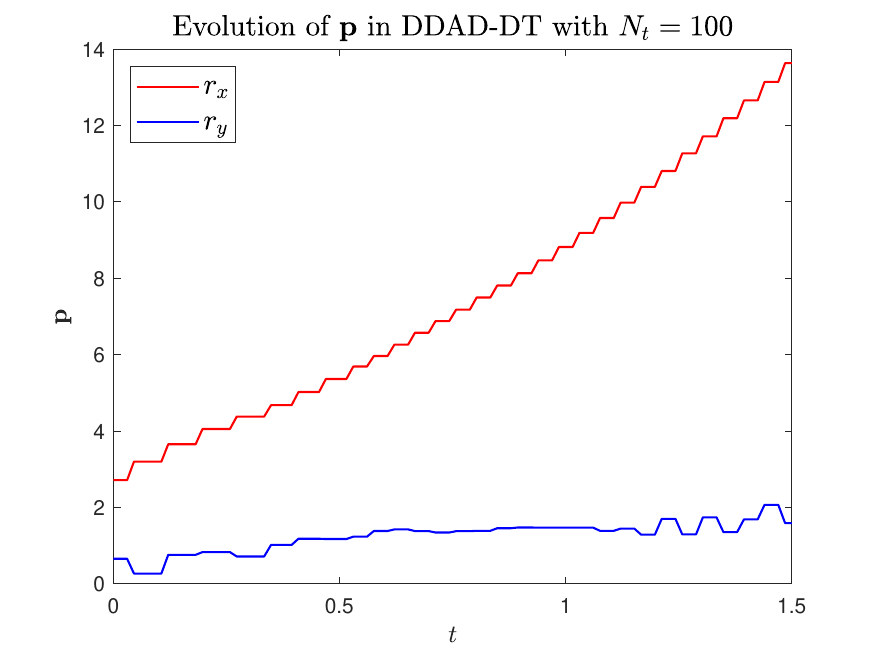}
	\includegraphics[width=0.24\textwidth]{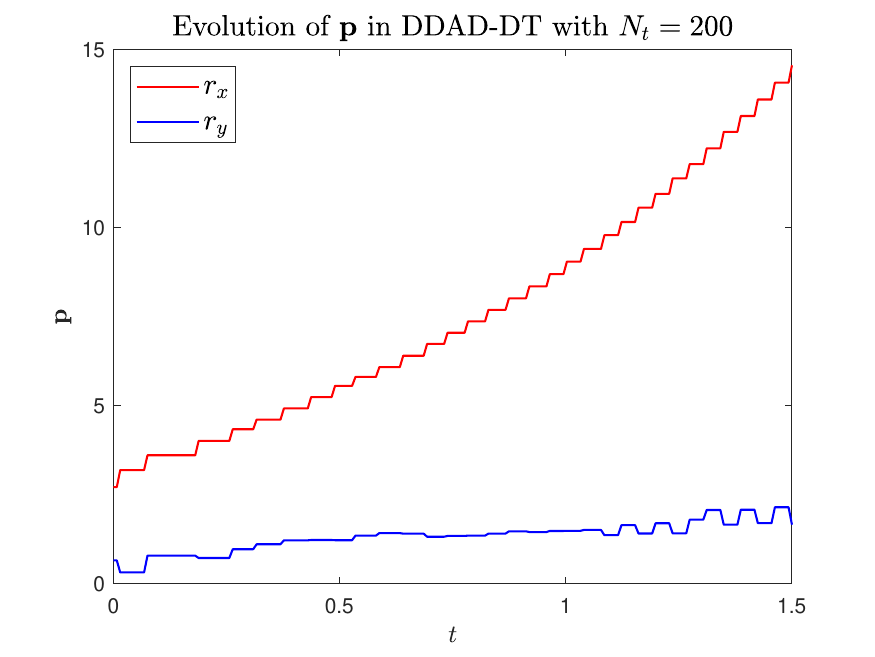}
	\includegraphics[width=0.24\textwidth]{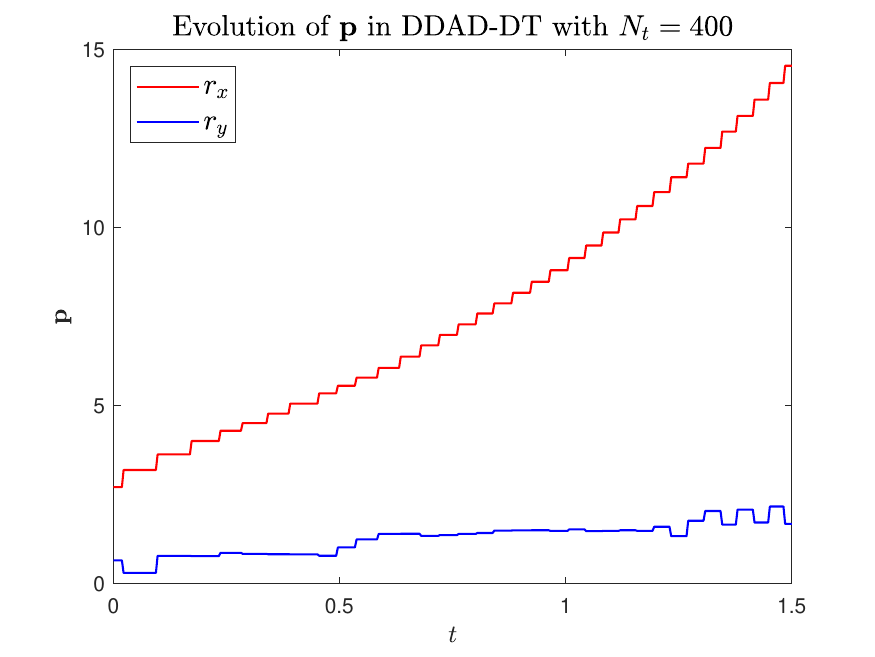}
    \includegraphics[width=0.24\textwidth]{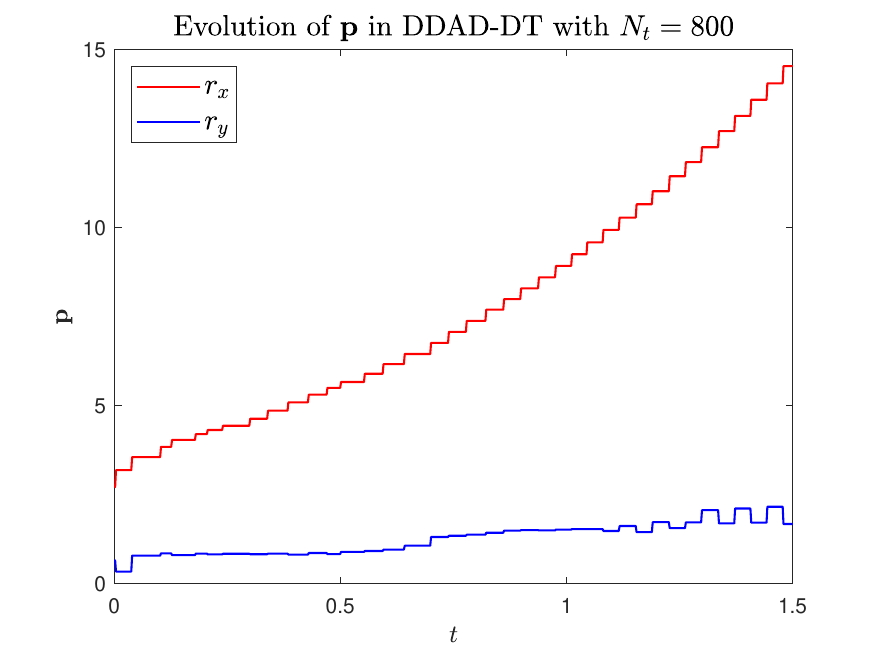}
	\caption{ Evolution of $\bm p=(r_x,r_y)$ for PDAD-DT (top) and DDAD-DT (bottom) with $N_t=100,200,400,800$.}
	\label{fig.TD-evo-p}
\end{figure}

\begin{table}[!ht]
\centering
\caption{Comparison of the PDAD-DT and DDAD-DT methods for the time-dependent
nonlinear equation~\eqref{eq:td_non} with $m_1=1200$. The table reports the runtime, $\ell_2$ errors, and convergence orders under time-step refinement.}
    \label{tab:td_non}
\setlength{\tabcolsep}{18pt} 
\begin{tabular}{ccccc}
\hline
& $N_t$ &  time (s) & $\ell_2$ error & order\\ \hline
\multirow{4}{*}{PDAD-DT} & 100 & 59.68 & 5.61e-03 & -\\ 
& 200& 65.31& 1.48e-03 & 1.92\\
 & 400& 74.01 & 3.84e-04 & 1.95\\ 
& 800 & 103.32 & 9.78e-05 & 1.97 \\ \hline
\multirow{4}{*}{DDAD-DT} & 100 &  11.98 & 5.60e-03& - \\ 
& 200& 20.16 & 1.48e-03 & 1.92\\
& 400 & 36.27 & 3.84e-04  &1.95 \\ 
& 800 & 70.59 & 9.78e-05  & 1.97 \\ \hline
\end{tabular}
\end{table}

\subsection{Helmholtz equation}

The Helmholtz equation is an elliptic partial differential equation that
arises in many applications \cite{wang2024ApracticalPINN,urban2025Unveiling}. In this section, we use it to evaluate
the performance of the proposed adaptive methods. The model problem is
\begin{equation}\label{eq:Helmholtz}
\begin{cases}
\Delta u(x,y) + k^{2} u(x,y) = q(x,y), & (x,y)\in\Omega=(-1,1)^2,\\[2pt]
u(x,y)=h(x,y), & (x,y)\in\partial\Omega.
\end{cases}
\end{equation}
We take $k=1$ and prescribe the exact solution
\begin{equation}
u(x,y)=\sin(a_{1}\pi x)\sin(a_{2}\pi y),
\end{equation}
from which the source term $q$ and boundary condition $h$ are obtained analytically.
\par
Two parameter sets $(a_1,a_2)$ are considered: $(6,6)$ and $(1,20)$. For the case
$(6,6)$, a uniform training grid of size $121\times121$ is used, whereas for the
case $(1,20)$, we employ a highly anisotropic grid of size $66\times300$ to
resolve the strong oscillations in the $y$-direction. The learning rate is set to
$1$, and the maximum number of iterations is $15$. To reduce the training cost, we set $m_{\lambda}=1000$.
\par
Table~\ref{tab:Helmholtz} reports the performance of the Fre-based and PDAD
methods, including the number of neurons, the trained parameters
$\bm{p}=(r_x,r_y)$, and the resulting accuracy. For $(a_1,a_2)=(6,6)$, PDAD
converges from the initial value $\bm{p}=(12.00,12.00)$ to the optimized
parameters $\bm{p}=(6.01,5.98)$ and achieves an $\ell_2$ error of $1.48\times 10^{-9}$.
For $(a_1,a_2)=(1,20)$, the method adjusts the parameters from
$\bm{p}=(6.00,36.00)$ to $\bm{p}=(1.27,31.91)$, reaching an error level of 
$2.61\times 10^{-9}$. Fig.~\ref{fig:Helmholtz} shows the PDAD solutions obtained with $m_1=3000$ neurons.

\begin{table}[!ht]
\centering
\caption{Comparison of Fre-based and PDAD for the equation \eqref{eq:Helmholtz}: optimized parameters $\bm{p}$, runtime, and $\ell_2$ errors.}
\label{tab:Helmholtz}
\setlength{\tabcolsep}{16pt} 
\begin{tabular}{cccccc}
\hline
$(a_1, a_2)$ & Method & $m_1$ & $\bm{p}=(r_x,r_y)$ & time (s) & $\ell_2$ error \\
\hline
\multirow{4}{*}{(6, 6)} 
    & \multirow{2}{*}{Fre-based} & 1500 & $(12.00,12.00)$ & 1.36 & 3.81e-03 \\
    &                            & 3000 & $(12.00,12.00)$ & 4.20 & 2.21e-07 \\
\cline{2-6}
    & \multirow{2}{*}{PDAD}   & 1500 & $(6.01,5.98)$   & 24.61 & 2.19e-06 \\
    &                            & 3000 & $(6.01,5.98)$  & 25.69 & 1.48e-09 \\
\hline
\multirow{4}{*}{(1, 20)} 
    & \multirow{2}{*}{Fre-based} & 1500 & $(6.00,36.00)$  & 2.98  & 1.82e-01 \\
    &                            & 3000 & $(6.00,36.00)$  & 8.80  & 3.01e-04 \\
\cline{2-6}
    & \multirow{2}{*}{PDAD}   & 1500 & $(1.27,31.91)$  & 37.40 & 2.00e-05 \\
    &                            & 3000 & $(1.27,31.91)$  & 39.48& 2.61e-09 \\
\hline
\end{tabular}
\end{table}

\begin{figure}[!ht]
	\centering
	\includegraphics[width=0.45\textwidth]{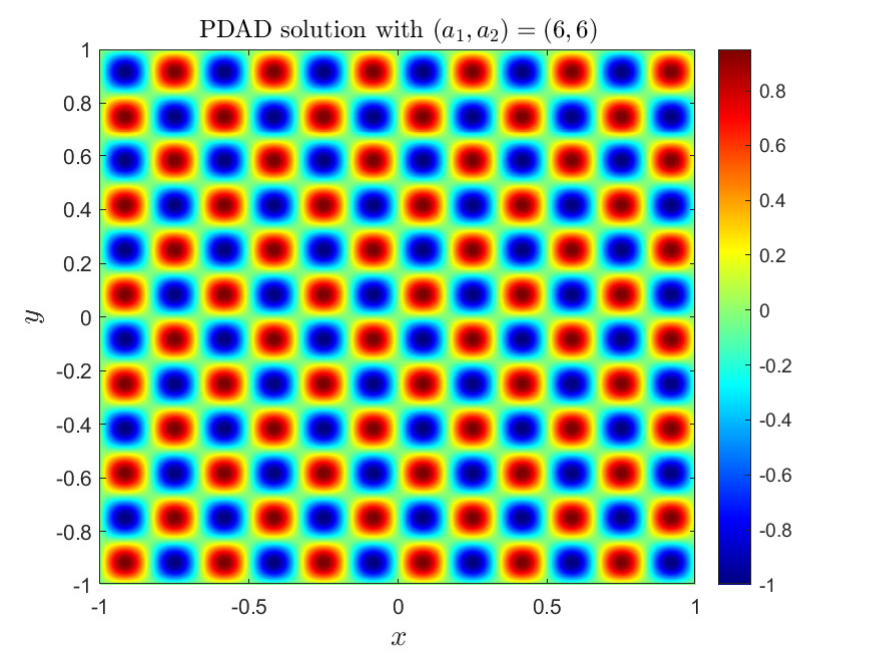}
	\includegraphics[width=0.45\textwidth]{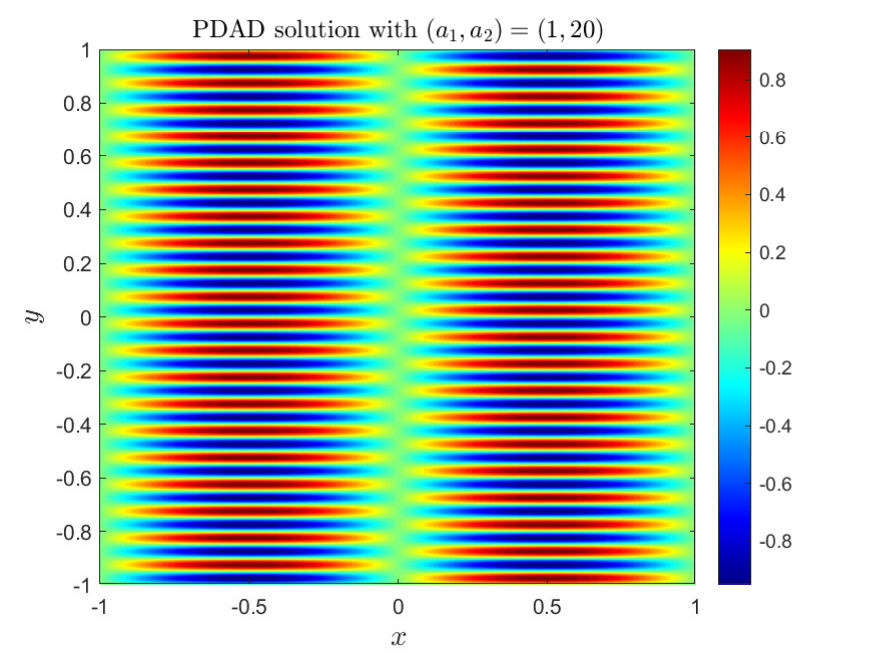}

	\caption{Numerical solutions of~\eqref{eq:Helmholtz} obtained by the PDAD method with $(a_1,a_2)=(6,6)$ (left) and $(a_1,a_2)=(1,20)$ (right). Here $m_1=3000$ and $m_\lambda=1000$.}
	\label{fig:Helmholtz}
\end{figure}

\subsection{Burgers' equation}
Burgers' equation is an important partial differential equation that captures the
interaction between nonlinear advection and viscous diffusion, and serves as a
simplified model in many fluid-mechanical applications. To assess the
performance of the proposed adaptive methods, we consider both one- and
two-dimensional Burgers' equations.
\par
\textbf{Case 1: One-dimensional Burgers' equation.}
\begin{equation}\label{eq:Burgers-1D}
\begin{cases}
u_t + u\,u_x - \dfrac{0.01}{\pi}\,u_{xx} = 0, & x\in(-1,1),\ t\in(0,1],\\[3pt]
u(x,0) = -\sin(\pi x),\\[2pt]
u(-1,t)=u(1,t)=0.
\end{cases}
\end{equation}
\par
 It is evident that the spatial structure of the Burgers' solution evolves significantly over time, which necessitates the use of time-adaptive distribution parameters $\bm{p}$. We employ PDAD-DT and DDAD-DT
with time-step numbers $N_t = 125, 250, 500$, and $1000$. Each experiment uses a
fixed learning rate of $0.5$, $15$ training iterations, $1500$ uniformly
sampled collocation points, and $700$ neurons.
\par
A variety of neural-network-based solvers have been reported for this 1D Burgers benchmark. 
As reported in the literature, SA-PINN achieves an $\ell_2$ error of $1.01\times 10^{-4}$, 
RAD-PINN achieves $2\times 10^{-4}$, 
and gPINN+RAR achieves $2\times 10^{-3}$ (\cite{mcClenny2023Selfadaptive,wu2023Acomprehensive,yu2022Gradientenhanced}). 
By comparison, our DDAD-DT attains an $\ell_2$ error of $5.85\times 10^{-6}$ with a runtime of $45.48$~s. 
Table~\ref{tab:Burgers1d-DT} indicates that the discrete-time framework substantially reduces the computational cost 
while achieving an empirical temporal convergence rate close to the theoretical second-order accuracy of the adopted semi-implicit scheme. 
The temporal evolution of the parameter $r_x$ for PDAD-DT and DDAD-DT is shown in Fig.~\ref{fig.bg1d-DT-p}; 
$r_x$ increases over time, which is consistent with the progressive steepening of the solution profile.

\begin{table}[!ht]
\centering
\caption{Comparison of the PDAD-DT and DDAD-DT methods for Burgers' equation~\eqref{eq:Burgers-1D} with $m_1=700$. The table reports the runtime, $\ell_2$ errors, and convergence orders under time-step refinement.}
    \label{tab:Burgers1d-DT}
\setlength{\tabcolsep}{18pt} 
\begin{tabular}{ccccc}
\hline
& $N_t$ &  time (s) & $\ell_2$ error & order\\ \hline
\multirow{4}{*}{PDAD-DT} & 125 & 40.19 & 4.56e-04 & -\\ 
& 250& 65.98 & 1.24e-04 & 1.88\\
& 500&  89.10 & 2.94e-05  & 2.08\\ 
& 1000 &  119.60 & 5.72e-06 & 2.36 \\ \hline
\multirow{4}{*}{DDAD-DT} & 125 & 14.97 & 4.60e-04 & -\\ 
& 250&21.29 & 1.25e-04 & 1.88\\
& 500 & 30.08 & 2.93e-05  &2.09 \\ 
& 1000 & 45.48 & 5.85e-06  & 2.32 \\ \hline
\end{tabular}
\end{table}

\begin{figure}[!ht]
	\centering
	\includegraphics[width=0.24\textwidth]{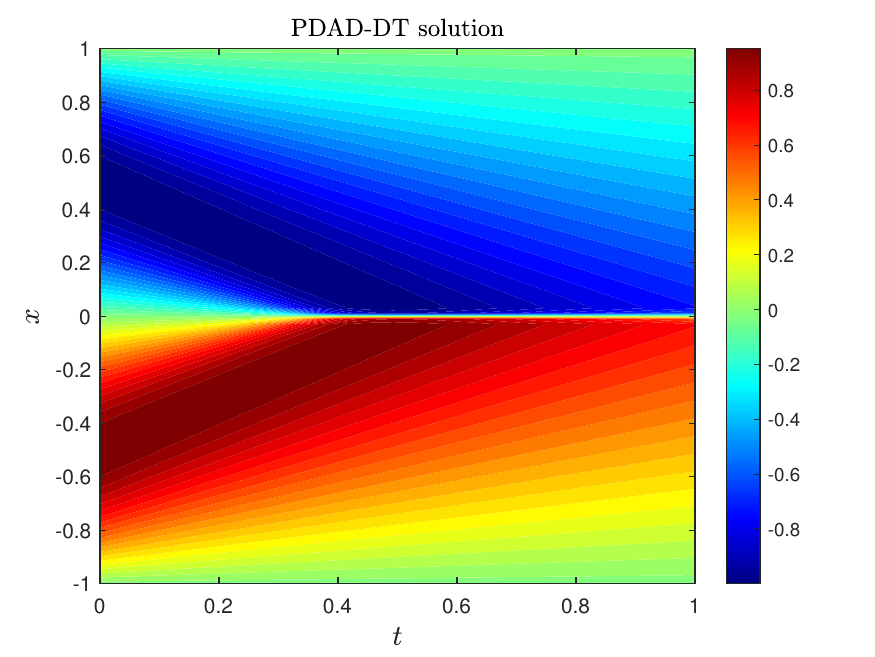}
	\includegraphics[width=0.24\textwidth]{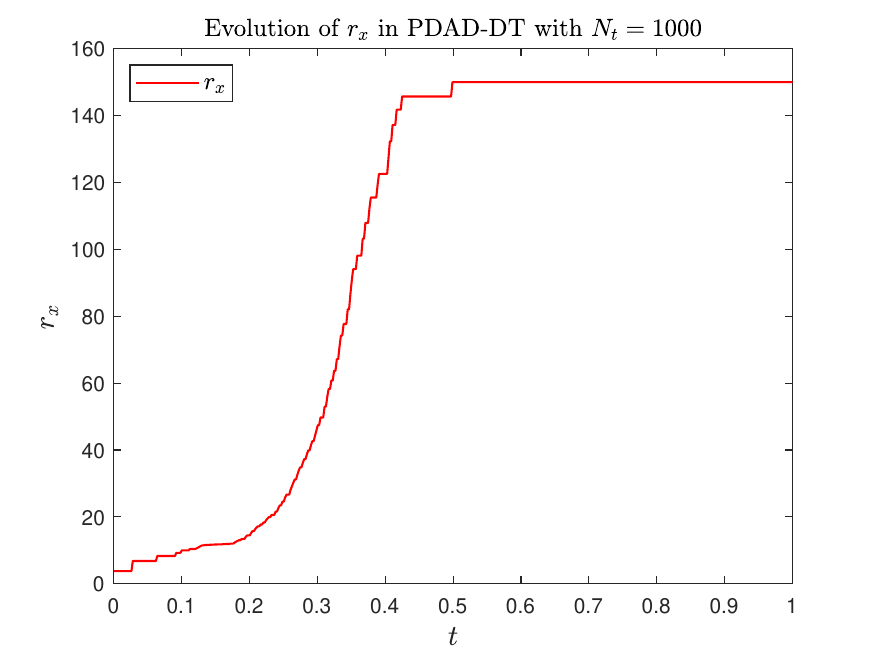}
	\includegraphics[width=0.24\textwidth]{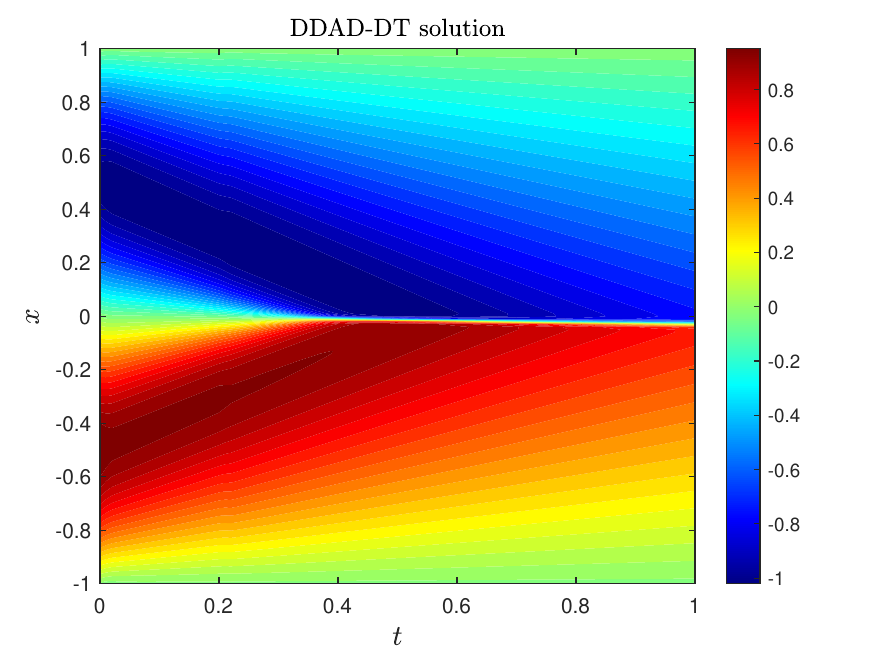}
     \includegraphics[width=0.24\textwidth]{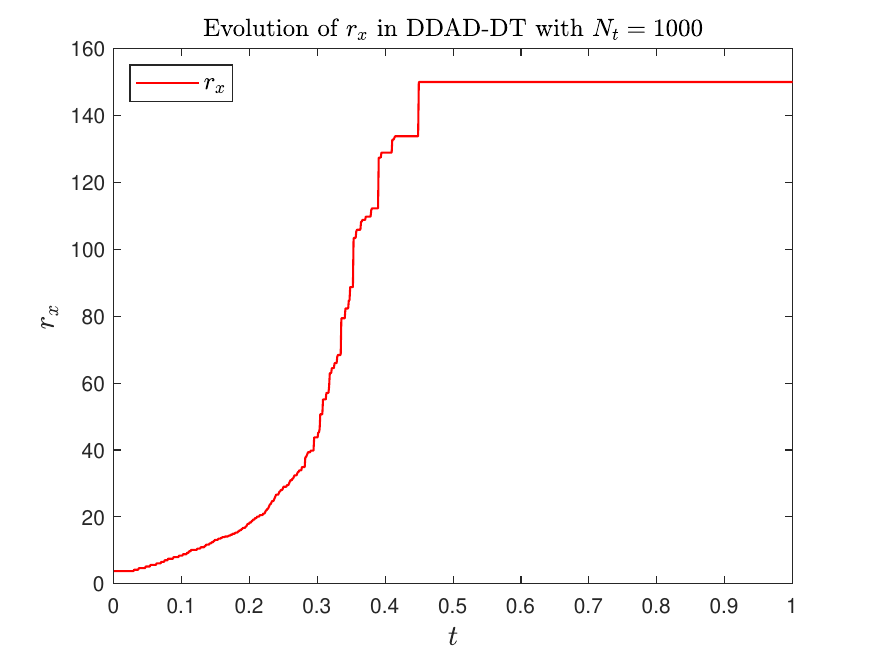}  
	\caption{Results for the 1D Burgers' equation~\eqref{eq:Burgers-1D} computed using PDAD-DT and DDAD-DT with $N_t=1000$.
From left to right, the columns show: the PDAD-DT solution, the optimized parameters $\bm p$ of PDAD-DT, the DDAD-DT solution, and the optimized parameters $\bm p$ of DDAD-DT.}
	\label{fig.bg1d-DT-p}
\end{figure}

\textbf{Case 2: Two-dimensional Burgers' equation.}
\par
We now extend Burgers' equation to two dimensions and consider the problem
\begin{equation}\label{eq:Burgers-2D}
\begin{cases}
u_t + u\,(u_x + u_y) - \varepsilon \Delta u = 0, 
    & (x,y)\in\Omega,\ t\in(0,1], \\[2pt]
u(x,y,t) = g(x,y,t), & (x,y)\in\Gamma_D,\ t\in(0,1], \\[2pt]
u(x,y,0) = h_0(x,y), & (x,y)\in\Omega,
\end{cases}
\end{equation}
where $\Omega=(0,1)^2 \subset \mathbb{R}^2$ and $\Gamma_D=\partial\Omega$.  
The exact solution is
\begin{equation}\label{eq:exact_solution}
u(x,y,t)
= \frac{1}{1+\exp\!\left(\dfrac{x+y-t}{2\varepsilon}\right)},
\end{equation}
which represents a smooth traveling-wave profile that steepens as 
$\varepsilon$ decreases.
\par
The two-dimensional problem was previously studied in~\cite{shang2023randomized} using a
Randomized Neural Network with a Petrov-Galerkin discretization (RaNN-PG),
where an $\ell_2$ error of $5.19\times 10^{-3}$ was reported for $\varepsilon=0.1$.
Here, we solve the same problem using the proposed PDAD-DT and DDAD-DT
frameworks. All experiments use a fixed learning rate of $0.5$, 
$10$ optimization iterations, a training grid of $70\times70$ points, and
$1000$ neurons.
\par
As summarized in Table~\ref{tab:Burgers2d-DT1}, DDAD-DT achieves an 
$\ell_2$ error of $7.04\times 10^{-6}$ with a total runtime of $32.04$~s at 
$\varepsilon=0.1$, representing an improvement of nearly three orders of magnitude compared with
the reported RaNN-PG result.

\begin{table}[!ht]
\centering
\caption{Comparison of the PDAD-DT and DDAD-DT methods for the 2D Burgers' equation~\eqref{eq:Burgers-2D} with $\epsilon=0.1$ and $m_1=1000$. The table reports the runtime, $\ell_2$ errors, and convergence orders under time-step refinement.}
    \label{tab:Burgers2d-DT1}
\setlength{\tabcolsep}{18pt} 
\begin{tabular}{ccccc}
\hline
& $N_t$ & time (s) & $\ell_2$ error & order\\ \hline
\multirow{4}{*}{PDAD-DT} & 50 & 13.47 & 4.57e-04 & -\\ 
& 100& 23.58 & 1.09e-04 & 2.07\\
 & 200& 56.34 & 3.05e-05 & 1.84\\ 
& 400 & 70.78 & 7.06e-06 & 2.11 \\ \hline
\multirow{4}{*}{DDAD-DT} 
& 50 &  5.53& 4.43e-04& - \\ 
& 100 &  8.67 &9.67e-05& 2.20 \\ 
& 200& 18.58 & 2.75e-05& 1.81\\
& 400 & 32.04 &7.04e-06  &1.97  \\ \hline
\end{tabular}
\end{table}

\begin{figure}[!ht]
	\centering
	\includegraphics[width=0.24\textwidth]{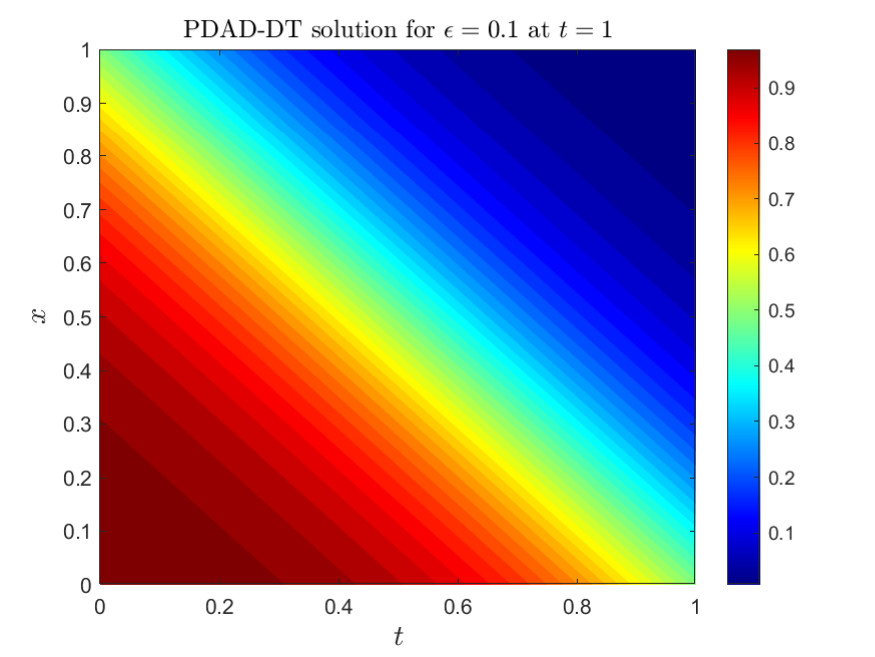}
	\includegraphics[width=0.24\textwidth]{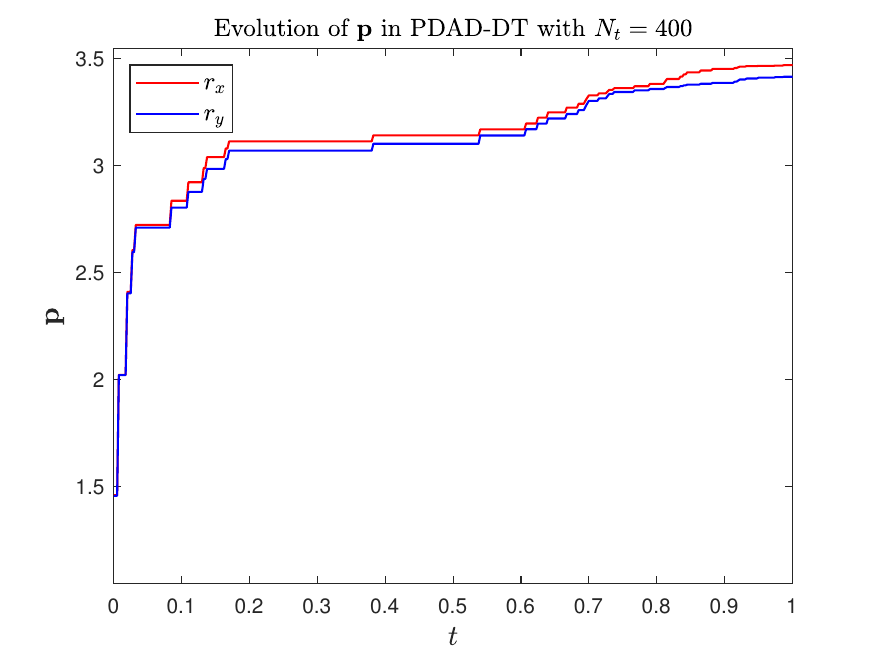}
	\includegraphics[width=0.24\textwidth]{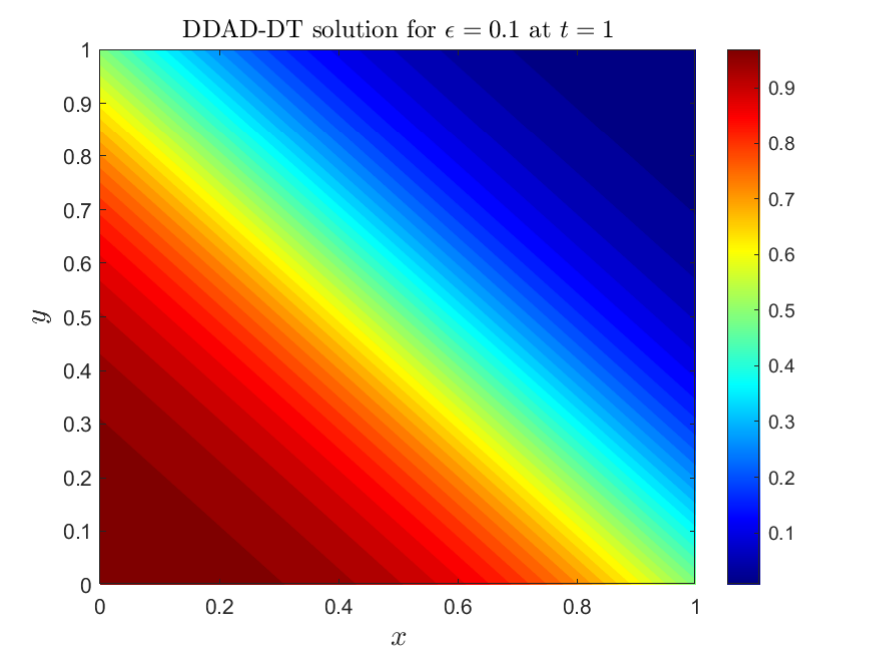}
     \includegraphics[width=0.24\textwidth]{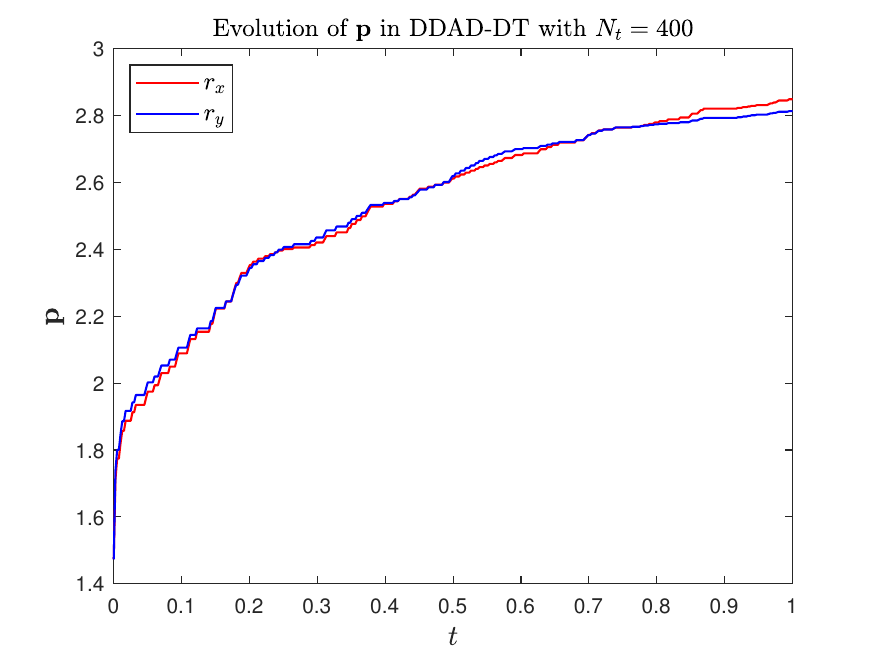}  
	\caption{Results for the 2D Burgers' equation~\eqref{eq:Burgers-2D} with $\varepsilon=0.1$ and $N_t=400$ at $t=1$.
From left to right: the PDAD-DT solution, the optimized parameters $\bm p=(r_x,r_y)$ for PDAD-DT, the DDAD-DT solution, and the optimized parameters $\bm p=(r_x,r_y)$ for DDAD-DT.}
	\label{fig.bg2_0.1}
\end{figure}
\par
To further assess the performance of the discrete-time AD-RaNN, we reduce $\varepsilon$ to $0.01$, which produces a much sharper shock
layer and significantly increases the numerical difficulty of the problem. As
shown in Table~\ref{tab:Burgers2d-DT2}, both PDAD-DT and DDAD-DT exhibit
second-order behavior when $N_t\le 100$. However, as the number of time steps
increases, the methods are unable to maintain the theoretical second-order
accuracy and the convergence order deteriorates.
To address this issue, we employ the layer-growth strategy, which augments the
network with locally adaptive basis functions in regions where the residual is
large (the selected residual points at $t=1$ are shown in the second column of
Fig.~\ref{fig.bg2_0.01}). Although this increases the computational cost, 
 it restores an observed convergence behavior close to second order up to $N_t=400$, 
achieving an $\ell_2$ error of $8.25\times 10^{-4}$. 
The temporal evolution of the learned parameters in the first hidden layer,
$\bm{p}=(r_x,r_y)$, and the second-layer parameter $r_2$ is displayed in the
third and fourth columns of Fig.~\ref{fig.bg2_0.01}, respectively.

\begin{table}[!ht]
\centering
\caption{Comparison of the PDAD-DT and DDAD-DT methods for the 2D Burgers'
equation~\eqref{eq:Burgers-2D} with $\epsilon=0.01$ and $m_1=2000$. The table reports the
$\ell_2$ errors, computational times, and observed convergence orders. 
Results with the layer-growth strategy (LG), using $m_2=300$ additional neurons,
are also included.}
    \label{tab:Burgers2d-DT2}
\setlength{\tabcolsep}{8pt} 
\begin{tabular}{cccccccc}
\hline
& $N_t$ & time (s) & $\ell_2$ error & order & LG time (s) & LG $\ell_2$ error & LG order\\ \hline
\multirow{4}{*}{PDAD-DT} & 50 & 126.90 & 1.15e-01 & -&345.14 &1.16e-01&-\\ 
& 100& 161.25 & 2.46e-02 & 2.22 &559.95 &2.31e-02&2.33\\
 & 200& 198.54 & 7.02e-03 & 1.81&787.02&4.01e-03& 2.53\\ 
& 400 & 298.13 & 2.19e-03 & 1.68&1405.27 & 8.28e-04&2.28\\ \hline
\multirow{4}{*}{DDAD-DT} 
& 50 & 37.64 & 1.13e-01&- & 70.73& 1.17e-01& - \\ 
& 100 &56.80 & 2.04e-02 & 2.47 &  101.71 &2.05e-02& 2.51 \\ 
& 200&106.03& 7.51e-03&1.44 & 438.45 & 3.63e-03& 2.50\\
& 400 &194.47& 5.89e-03  & 0.35 & 991.51 &8.25e-04  &2.14  \\ \hline
\end{tabular}
\end{table}

\begin{figure}[!ht]
	\centering
	\includegraphics[width=0.24\textwidth]{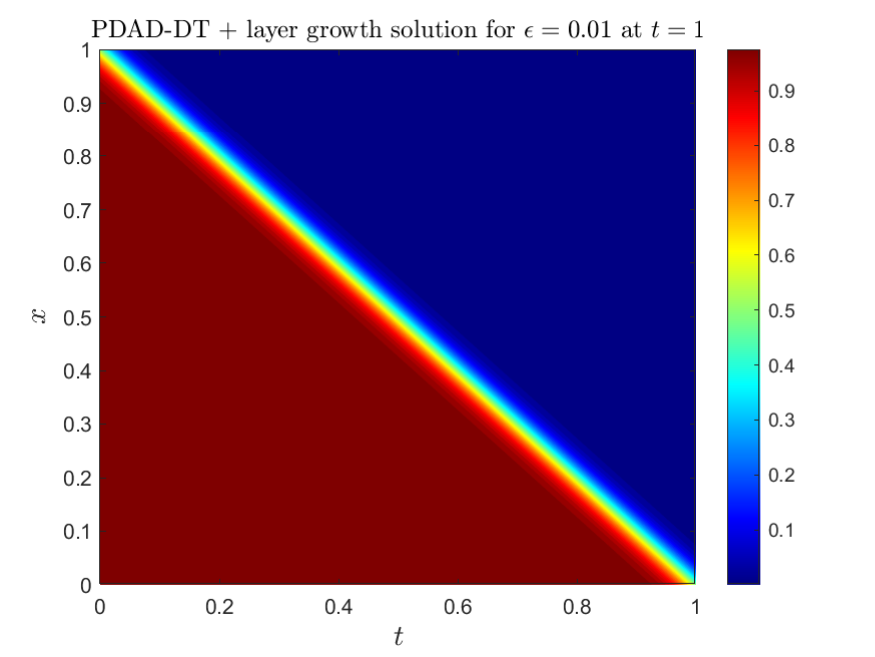}
	\includegraphics[width=0.24\textwidth]{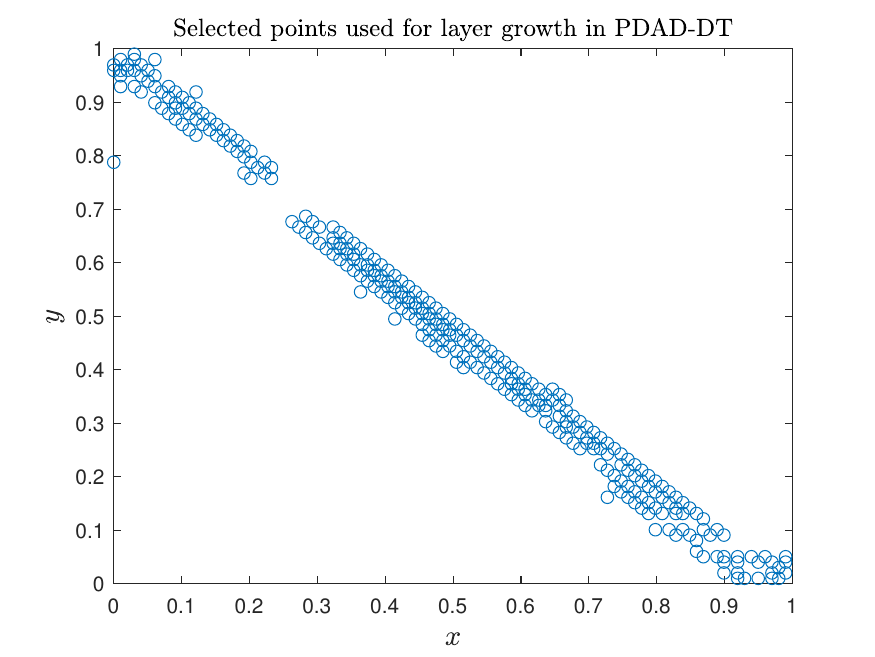}
	\includegraphics[width=0.24\textwidth]{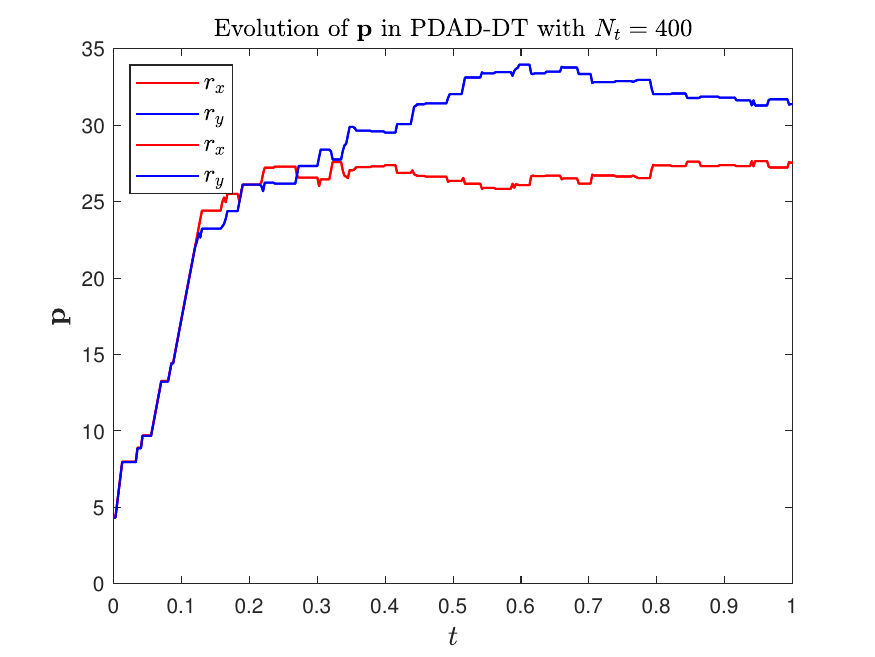}
     \includegraphics[width=0.24\textwidth]{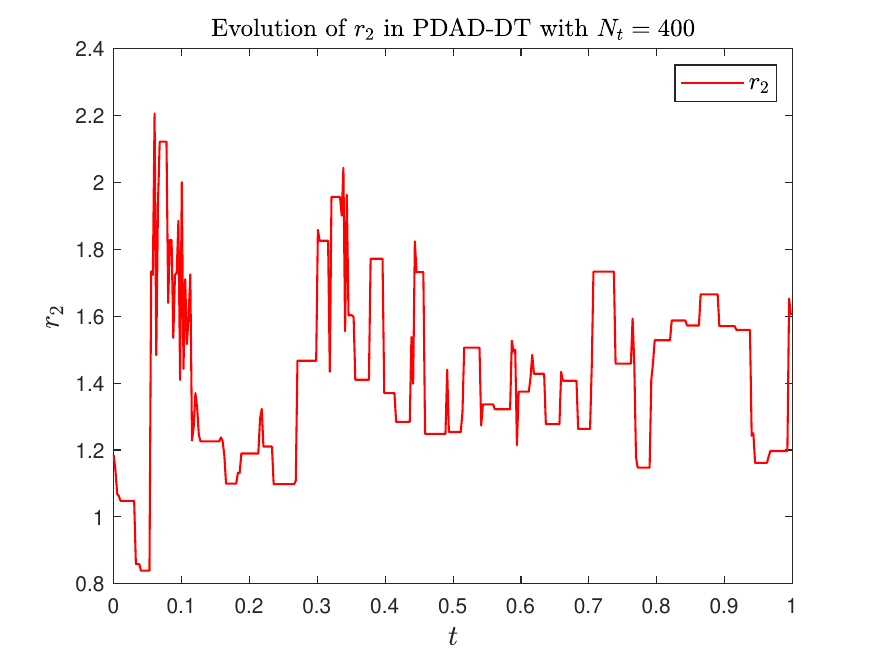}  \\
	\includegraphics[width=0.24\textwidth]{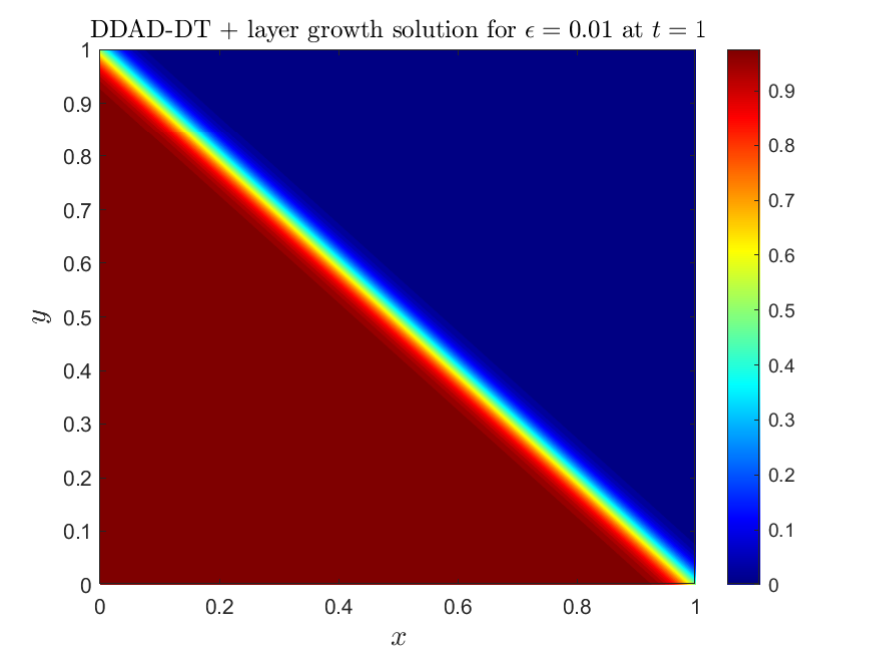}
	\includegraphics[width=0.24\textwidth]{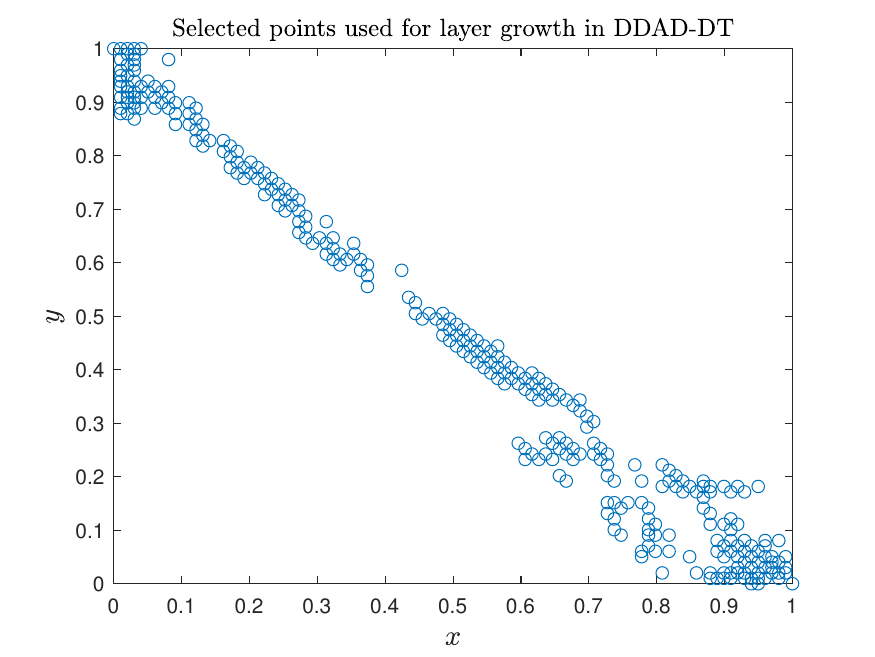}
	\includegraphics[width=0.24\textwidth]{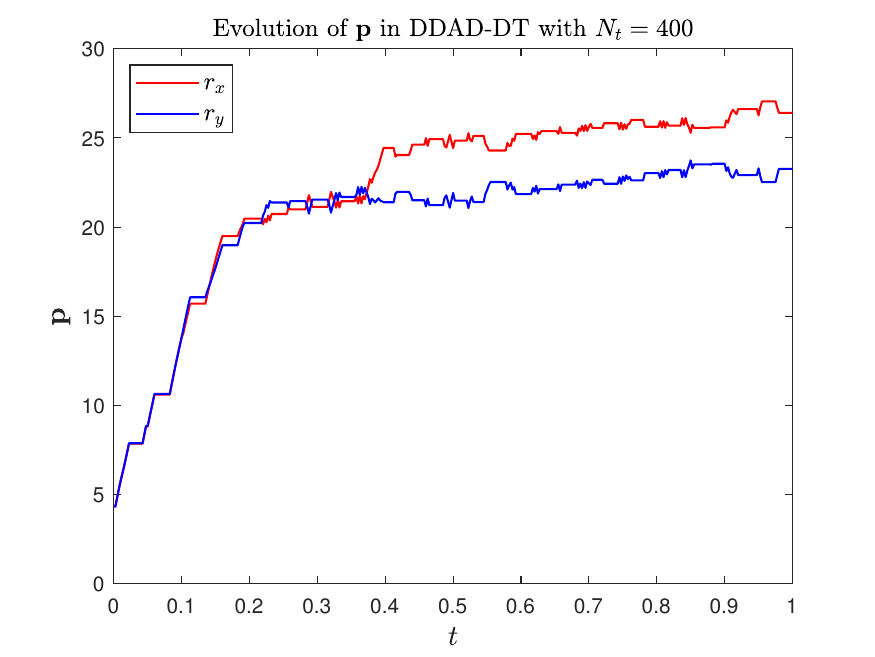}
    \includegraphics[width=0.24\textwidth]{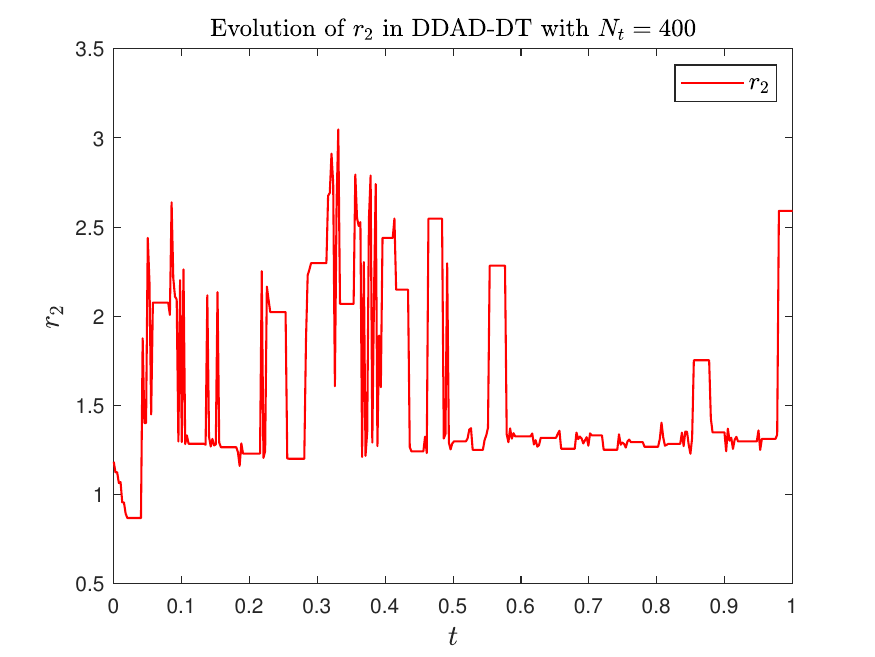}
	\caption{Results for 2D Burgers' equation~\eqref{eq:Burgers-2D} obtained by PDAD-DT (top) and DDAD-DT (bottom) with $\varepsilon=0.01$ and $N_t=400$.
From left to right, the columns show: the numerical solution at $t=1$, the residual points selected for the second layer at $t=1$, the optimized first-layer parameters $\bm p$, and the second-layer parameter $r_2$.
}
	\label{fig.bg2_0.01}
\end{figure}
\subsection{Allen-Cahn equation}
\textbf{Case 1: One-dimensional Allen-Cahn equation}
\par
The Allen-Cahn equation is a nonlinear reaction-diffusion model that describes
phase separation and interface evolution in materials science. We consider the
one-dimensional Allen-Cahn equation with periodic boundary conditions:
\begin{equation}\label{eqn:ac1d}
\begin{cases}
u_t - 0.0001\,u_{xx} + 5(u^3 - u) = 0, 
    & x\in(-1,1),\ t\in(0,1],\\[2pt]
u(x,0) = x^2\cos(\pi x),\\[2pt]
u(-1,t)=u(1,t),\quad u_x(-1,t)=u_x(1,t),
    & t\in(0,1].
\end{cases}
\end{equation}
\par
Since an analytical solution to~\eqref{eqn:ac1d} is not available, a reference
solution is generated using a spectral method with sufficiently fine resolution
to ensure accuracy. This equation is frequently used as a benchmark problem in neural
network-based PDE solvers, such as SA-PINN ($\ell_2$ error $2.1\times 10^{-4}$) and bc-PINN (\cite{mattey2022Anovel}) 
($\ell_2$ error $1.68\times 10^{-4}$).
\par
For the proposed PDAD-DT and DDAD-DT methods, we employ $1000$ training points,
$600$ neurons, a fixed learning rate of $0.5$, and $15$ optimization iterations.
As shown in Table~\ref{tab:AC1D}, DDAD-DT achieves an $\ell_2$ error of
$7.58\times 10^{-6}$ with a training time of $21.21$ seconds. The numerical
solution and the evolution of the distribution parameter $\bm{p}$ are shown in
Fig.~\ref{fig.ac1d}; notably, $\bm{p}$ increases monotonically during time
stepping, exceeding $100$ by the final time, which reflects the sharpening of
the solution profile.

\begin{table}[!ht]
\centering
\caption{Comparison of the PDAD-DT and DDAD-DT methods for the Allen-Cahn equation~\eqref{eqn:ac1d} with $m_1=600$.
The table reports the runtime, $\ell_2$ errors, and convergence orders under time-step refinement.}
    \label{tab:AC1D}
\setlength{\tabcolsep}{18pt} 
\begin{tabular}{ccccc}
\hline
& $N_t$ &  time (s) & $\ell_2$ error & order\\ \hline
\multirow{4}{*}{PDAD-DT} & 125 & 8.27 & 4.49e-04 & -\\ 
& 250& 10.32& 1.14e-04 & 1.98\\
 & 500& 27.13 & 3.05e-05 & 1.90\\ 
& 1000 & 42.36 & 7.92e-06 & 1.95 \\ \hline
\multirow{4}{*}{DDAD-DT} & 125 &  6.05 & 4.49e-04& - \\ 
& 250& 8.53 & 1.13e-04 & 1.99\\
& 500 & 12.69& 2.87e-05  &1.98 \\ 
& 1000 & 21.21 & 7.58e-06  & 1.92 \\ \hline
\end{tabular}
\end{table}

\begin{figure}[!ht]
	\centering
	\includegraphics[width=0.24\textwidth]{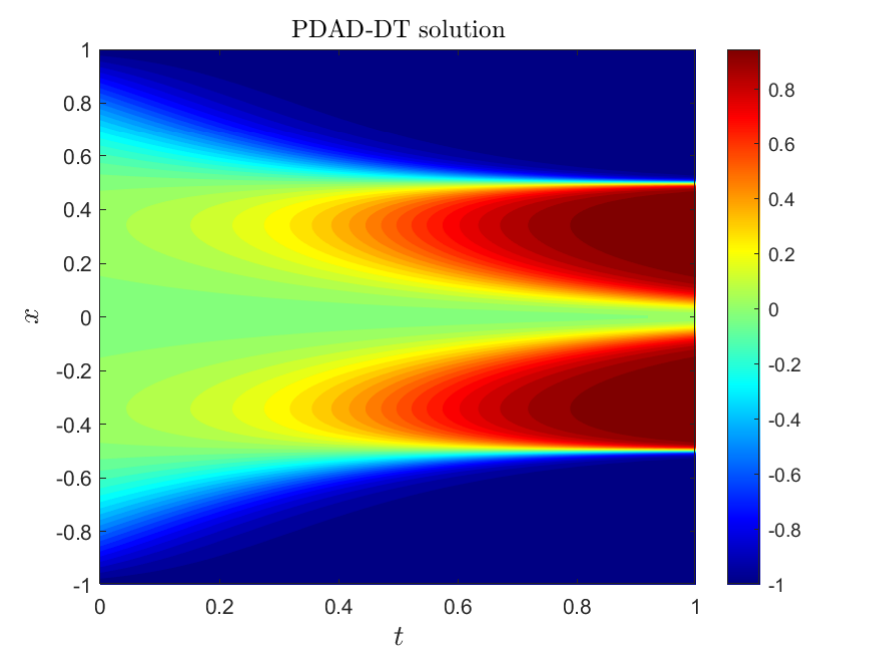}
	\includegraphics[width=0.24\textwidth]{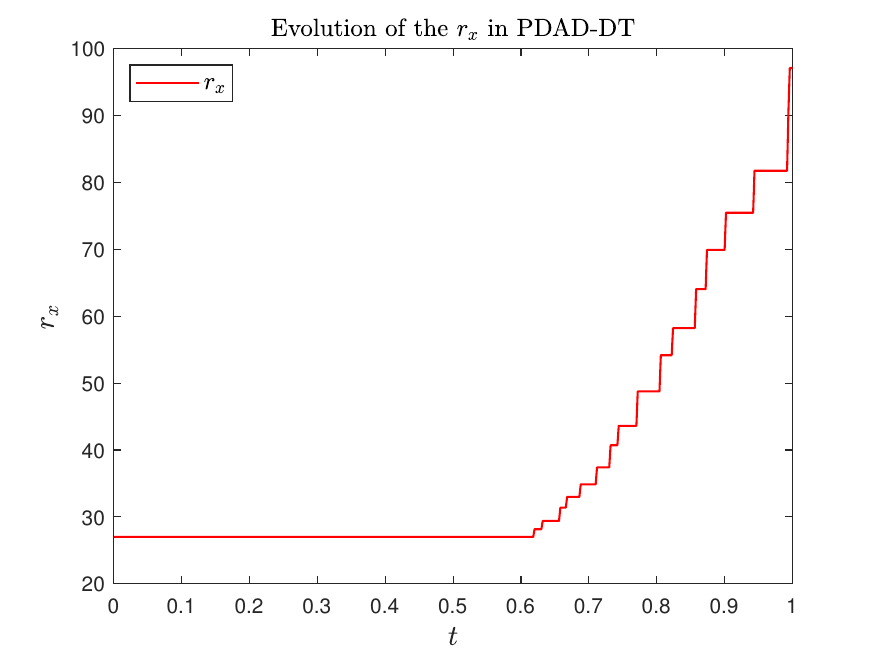}
	\includegraphics[width=0.24\textwidth]{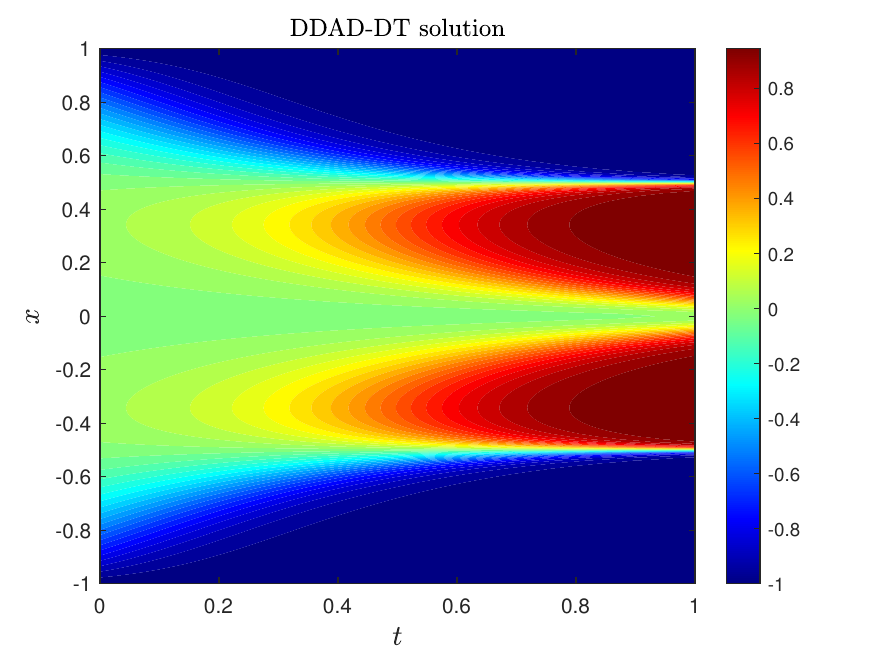}
     \includegraphics[width=0.24\textwidth]{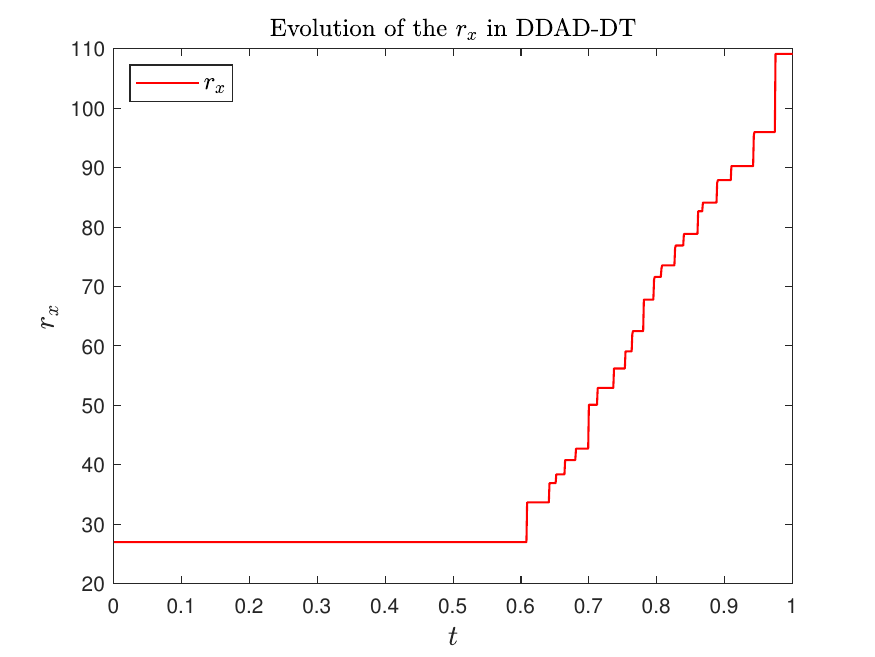}  
	\caption{Results for the Allen-Cahn equation~\eqref{eqn:ac1d} computed using PDAD-DT and DDAD-DT with $N_t=1000$.
From left to right, the columns show: the PDAD-DT solution, the optimized parameters $\bm p$ of PDAD-DT, the DDAD-DT solution, and the optimized parameters $\bm p$ of DDAD-DT.}
	\label{fig.ac1d}
\end{figure}
\par
\textbf{Case 2: Two-dimensional Allen-Cahn equation}
\par
We now extend the Allen-Cahn equation to two dimensions and consider the problem
with Dirichlet boundary conditions:
\begin{equation}\label{eqn:ac2d}
\begin{cases}
u_t - 0.0001\,\Delta u + 2(u^3 - u) = 0, 
    & (x,y)\in(-1,1)^2,\ t\in(0,1],\\[2pt]
u(x,y,0)=\cos\!\left(\tfrac{\pi}{2}x\right)\cos\!\left(\tfrac{\pi}{2}y\right),
    & (x,y)\in\Omega,\\[2pt]
u(x,y,t)=0,
    & (x,y)\in\partial\Omega,\ t\in(0,1].
\end{cases}
\end{equation}
\par
A reference solution is computed using a high-resolution Runge-Kutta method.
For the PDAD-DT and DDAD-DT solvers, we employ an $80\times80$ training grid,
$1000$ neurons, a fixed learning rate of $0.5$, and $10$ optimization
iterations. As shown in Table~\ref{tab:AC2D}, both PDAD-DT and DDAD-DT begin to
lose second-order accuracy when the number of time steps reaches $N_t=400$.
Therefore, the layer-growth strategy is also applied in this case.
\par
With layer growth, the discrete-time AD-RaNN framework  recovers an observed convergence behavior close to second order up to $N_t=400$.
In particular, the DDAD-DT method achieves
an $\ell_2$ error of $4.64\times 10^{-6}$ with a total runtime of $133.97$ seconds.

\begin{table}[!ht]
\centering
\caption{Comparison of the PDAD-DT and DDAD-DT methods for the 2D Allen-Cahn
equation~\eqref{eqn:ac2d} with $m_1=1000$. The table reports the
$\ell_2$ errors, computational times, and observed convergence orders. 
Results with the layer-growth strategy (LG), using $m_2=300$ additional neurons,
are also included.}

    \label{tab:AC2D}
\setlength{\tabcolsep}{8pt} 
\begin{tabular}{cccccccc}
\hline
& $N_t$ &  time (s) & $\ell_2$ error & order & LG time (s) & LG $\ell_2$ error & LG order\\ \hline
\multirow{4}{*}{PDAD-DT} & 50 &   55.14 & 3.14e-04 & -&78.84 &3.06e-04&-\\ 
& 100& 78.55 & 7.54e-05 &2.06 &194.31&7.72e-05&1.99\\
 & 200& 118.10 & 1.76e-05 & 2.10&164.32&1.82e-05& 2.08\\ 
& 400 & 167.31 &1.63e-05 & 0.11&250.18 & 4.81e-06&1.92\\ \hline
\multirow{4}{*}{DDAD-DT} 
& 50 & 10.95 & 2.99e-04&- & 36.16& 3.09e-04& - \\ 
& 100 &16.61 & 7.31e-05 & 2.03 &  50.37 &7.77e-05& 1.99 \\ 
& 200&29.57& 2.90e-05&1.33 & 73.80& 1.92e-05& 2.02\\
& 400 &53.05& 3.19e-05  & -0.14 & 133.97 &4.64e-06  &2.05  \\ \hline
\end{tabular}
\end{table}

\begin{figure}[!ht]
	\centering
	\includegraphics[width=0.24\textwidth]{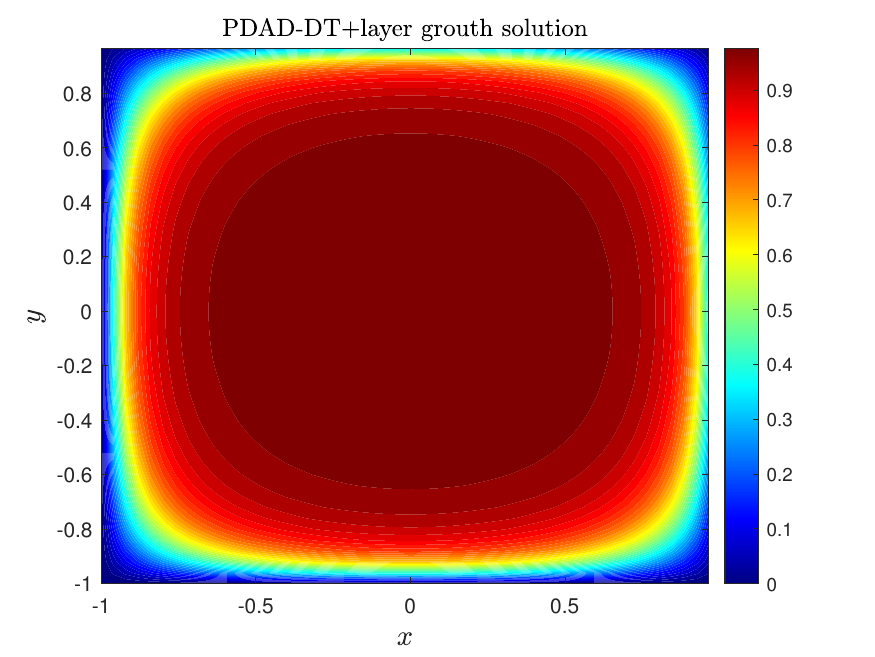}
	\includegraphics[width=0.24\textwidth]{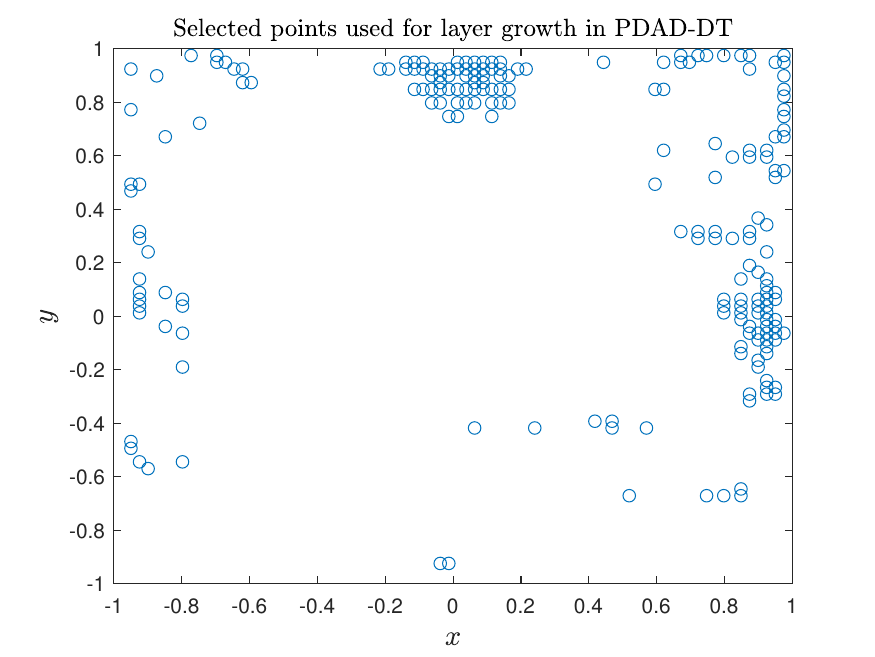}
	\includegraphics[width=0.24\textwidth]{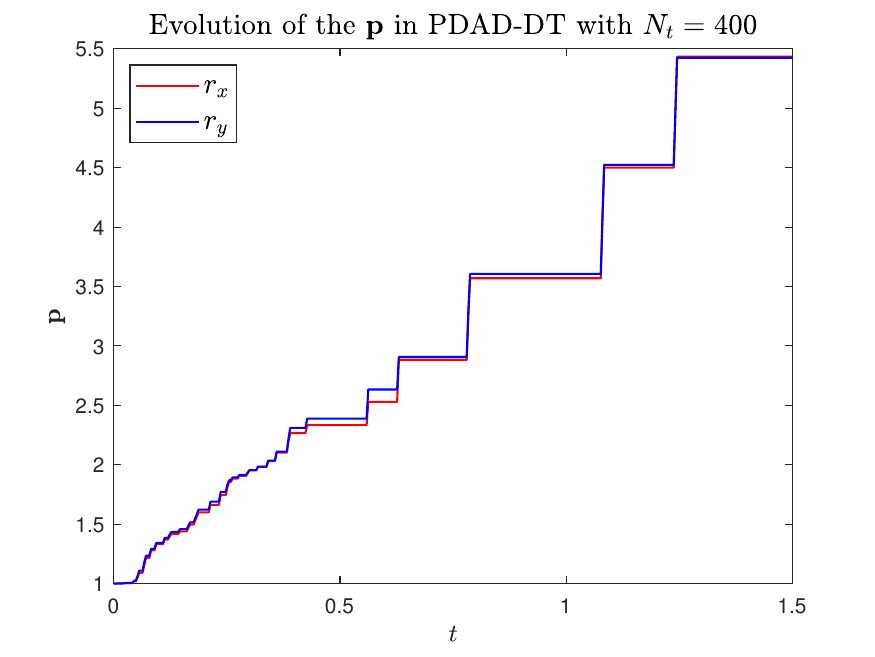}
     \includegraphics[width=0.24\textwidth]{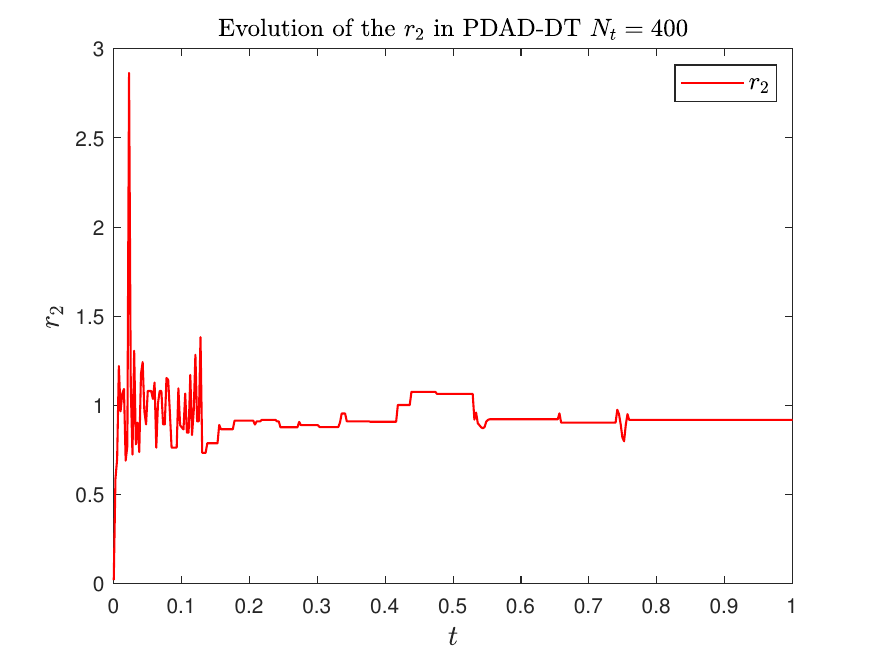}  \\
	\includegraphics[width=0.24\textwidth]{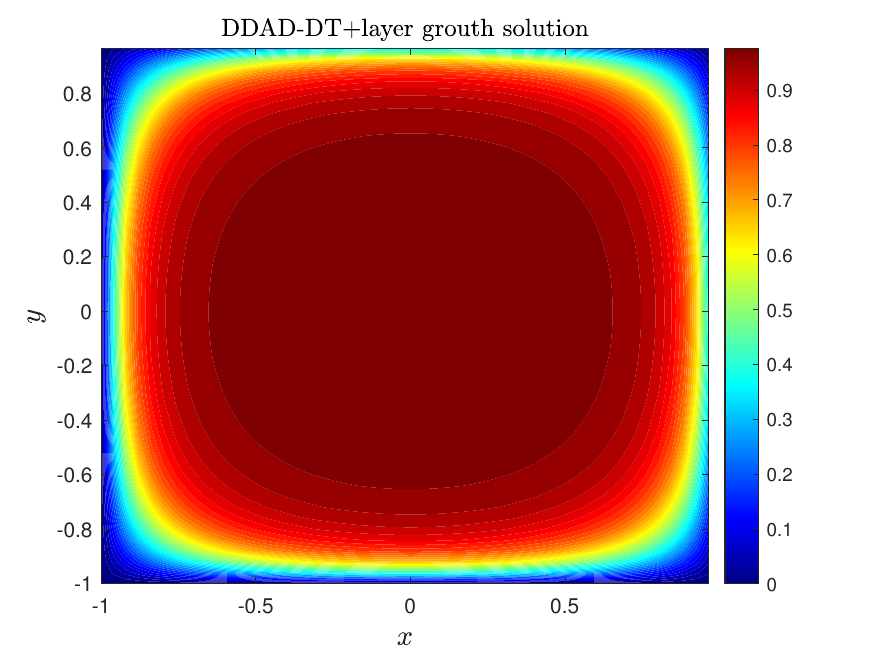}
	\includegraphics[width=0.24\textwidth]{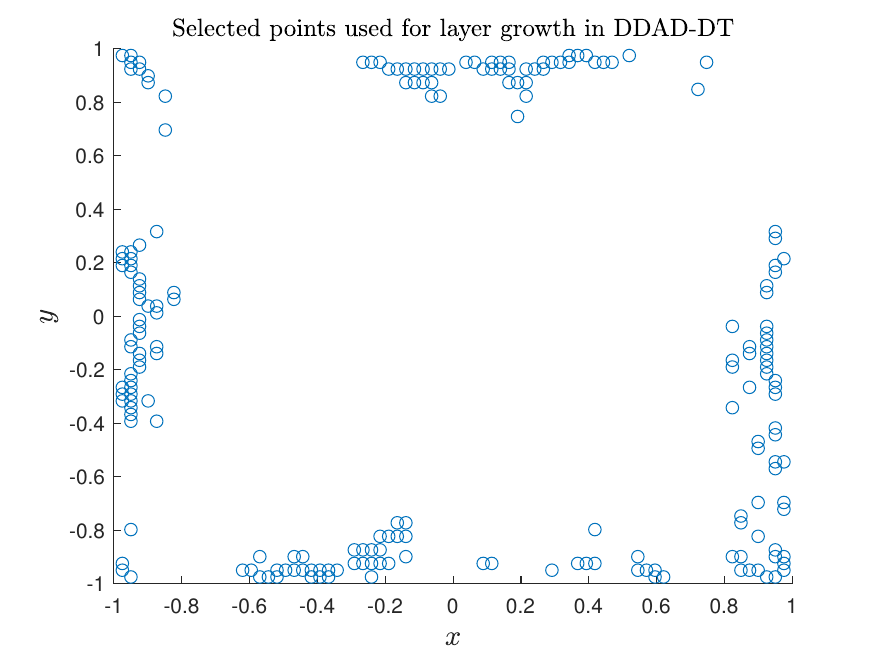}
	\includegraphics[width=0.24\textwidth]{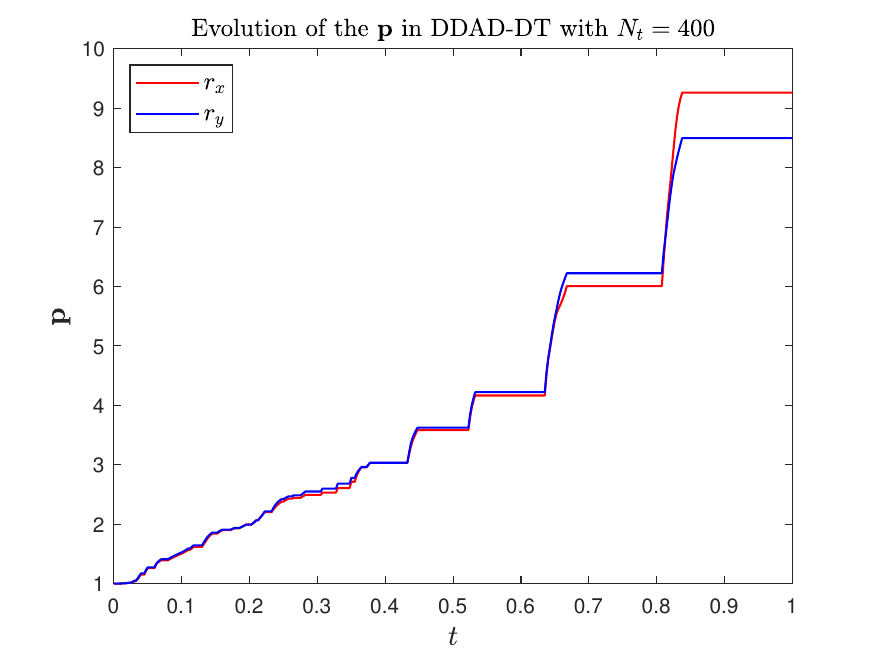}
    \includegraphics[width=0.24\textwidth]{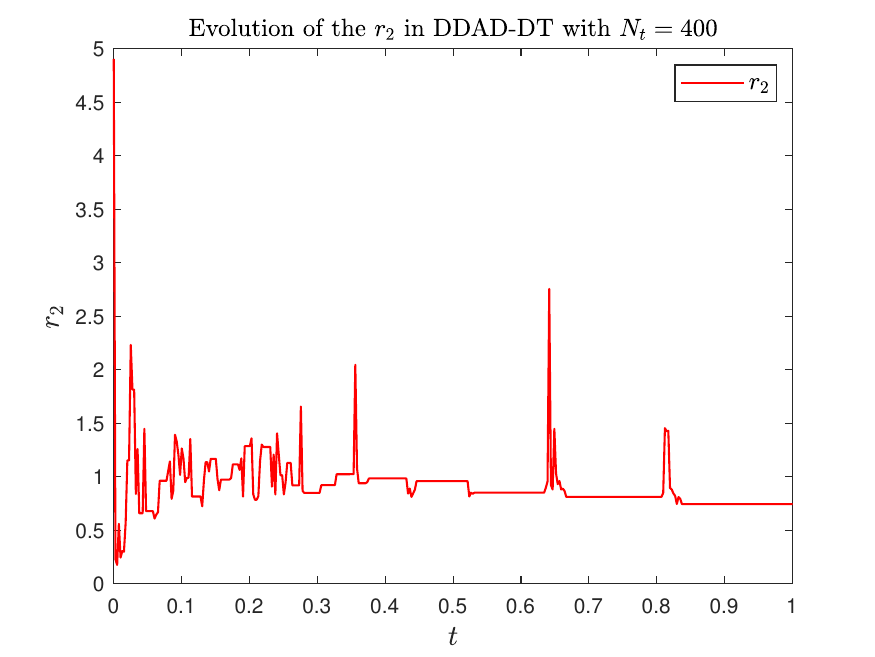}
	\caption{Results for the Allen-Cahn equation~\eqref{eqn:ac2d} obtained by PDAD-DT (top) and DDAD-DT (bottom) with $N_t=400$ at $t=1$.
From left to right, the columns show: the numerical solution, the residual points selected for the second layer, the optimized first-layer parameters $\bm p$, and the second-layer parameter $r_2$.}
	\label{fig.AC2d}
\end{figure}

\subsection{Mean curvature flow}

As a fundamental example of geometric flows, mean curvature flow is a surface-shortening process with important applications in materials science, image processing, and others (\cite{gage1986The,grayson1987The}). Deep learning methods based on the sharp interface model for solving mean curvature flow have practical applications (\cite{li2025Astructurepreserving}).
We consider the following formulation of mean curvature flow:
\begin{equation}\label{eq:MCF}
\partial_t \bm{x} =  \frac{\partial_{\rho\rho} \bm{x}}{|\partial_\rho \bm{x}|^2},
\end{equation}
where $\rho \in(0,1)$. A first-order discrete-time framework is commonly used to solve MCF numerically:
\begin{equation}
 \bm{x}_{i+1}- \frac{\partial_{\rho\rho} \bm{x}_{i+1}}{|\partial_\rho \bm{x}_i|^2}\Delta t =\bm{x}_i,
\end{equation}
\par
 Although a closed curve is considered a manifold without boundary
 in both geometric and topological contexts, periodic boundary conditions must still be applied at the ending
 points during numerical computations. This ensures that the curve remains closed and maintains continuity
 and smoothness at the junction. Periodic boundary conditions can be enforced through both soft and hard constraint approaches. In this work, we adopt the hard constraint approach by constructing a periodic boundary layer that connects the input layer to the first hidden layer of the neural network (\cite{dong2021A}). This architecture ensures that all basis functions inherently satisfy periodicity constraints. As an illustration, consider the following one-dimensional periodic boundary condition:

\begin{equation}\label{eqn:periodic-boundary-conditions}
	\begin{aligned}
		\partial_x^s u(a_1,t)=\partial_x^s u(a_2,t),\qquad s=0,1,\dots,\ t\in(0,T].
	\end{aligned}
\end{equation}
\par
The periodic boundary layer is implemented by applying a nonlinear transformation to the input spatial variable \(x\), mapping it to a set of linear combinations
of periodic functions that satisfy equation \eqref{eqn:periodic-boundary-conditions}. 
 For simplicity, the periodic features in this section are chosen in the form $\sin(2\pi kx+b)$, where the phase shift $b$ is randomly sampled from the uniform distribution $\mathcal{U}(-k\pi, k\pi)$.
\begin{figure}[!ht]
	\centering
	\includegraphics[width=0.5\textwidth]{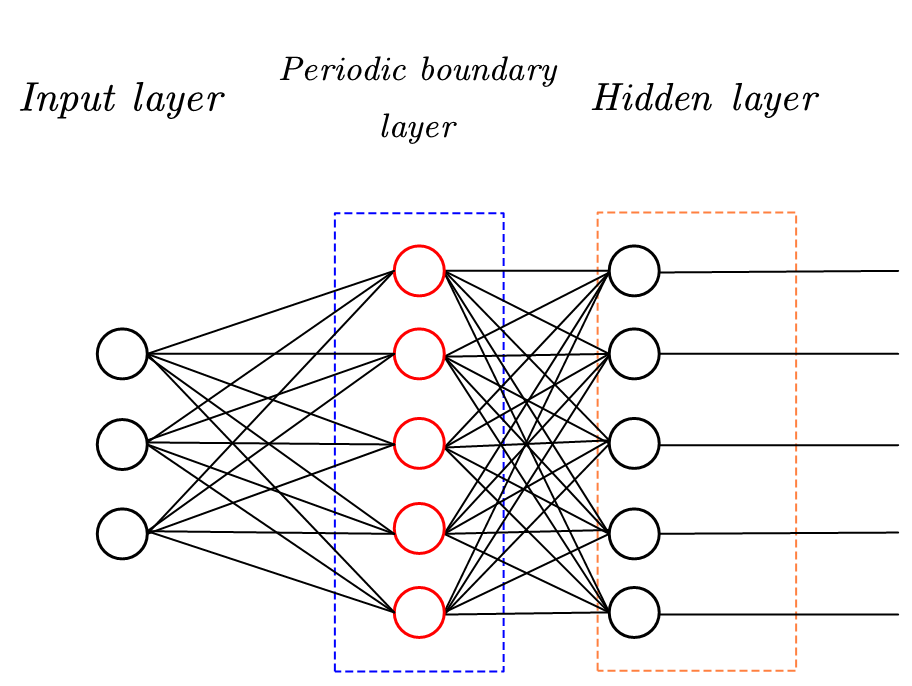}
	\caption{Schematic of the neural network architecture with a hard-constrained periodic boundary layer.}
	\label{fig:PBL}
\end{figure}
\par
In the mean curvature flow (MCF) experiments, we consider the following exact solution:
\begin{equation}\label{eq:mcf_initial_1}
    \bm{x}(\rho,t)=(x(\rho,t),y(\rho,t))=\big(\sqrt{9-2t}\cos{(2\pi \rho)}, \sqrt{9-2t}\sin{(2\pi \rho)}\big).
\end{equation}
\par
Two separate neural networks, each equipped with distinct parameter vectors $\bm{p}$, are employed to approximate $x(\rho,t)$ and $y(\rho,t)$, respectively. For the DDAD-DT method, we utilize $500$ training points, a fixed learning rate of $0.5$, and 15 training iterations. Specifically, $5$ periodic basis functions are implemented to enforce periodic boundary conditions, while the last hidden layer contains 200 neurons. As illustrated in Table \ref{tab:mcf}, DDAD-DT achieves an $\ell_2$ error of $8.65\times 10^{-4}$ with a training time of $8.91$ seconds, and its convergence order is approximately first order.

\begin{table}[!ht]
\centering
\caption{Results of DDAD-DT for mean curvature flow \eqref{eq:MCF} with $m_1=200$.}
    \label{tab:mcf}
\setlength{\tabcolsep}{18pt} 
\begin{tabular}{ccccc}
\hline
& $N_t$ &  time (s) & $\ell_2$ error & order \\ \hline
\multirow{5}{*}{DDAD-DT} & 50 &  0.38 & 1.24e-02& - \\ 
& 100& 1.99 & 6.54e-03 & 0.92\\
& 200 &  2.29& 3.37e-03  &0.96 \\ 
& 400 & 4.85 & 1.71e-03  & 0.98 \\
 & 800 & 8.91 & 8.65e-04  & 0.98 \\ \hline
\end{tabular}
\end{table}

\begin{figure}[!ht]
	\centering
	\includegraphics[width=0.35\textwidth]{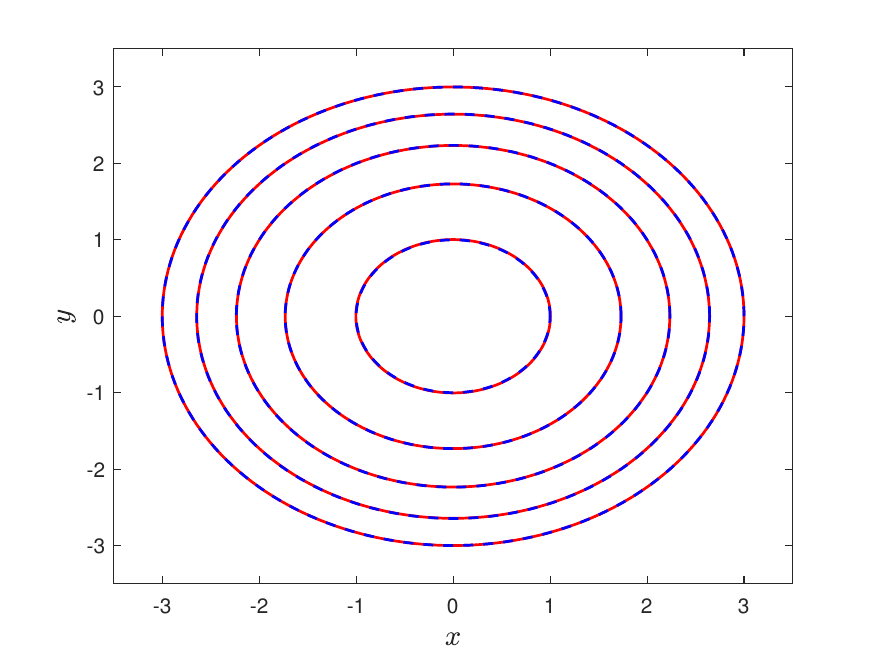}
	\includegraphics[width=0.35\textwidth]{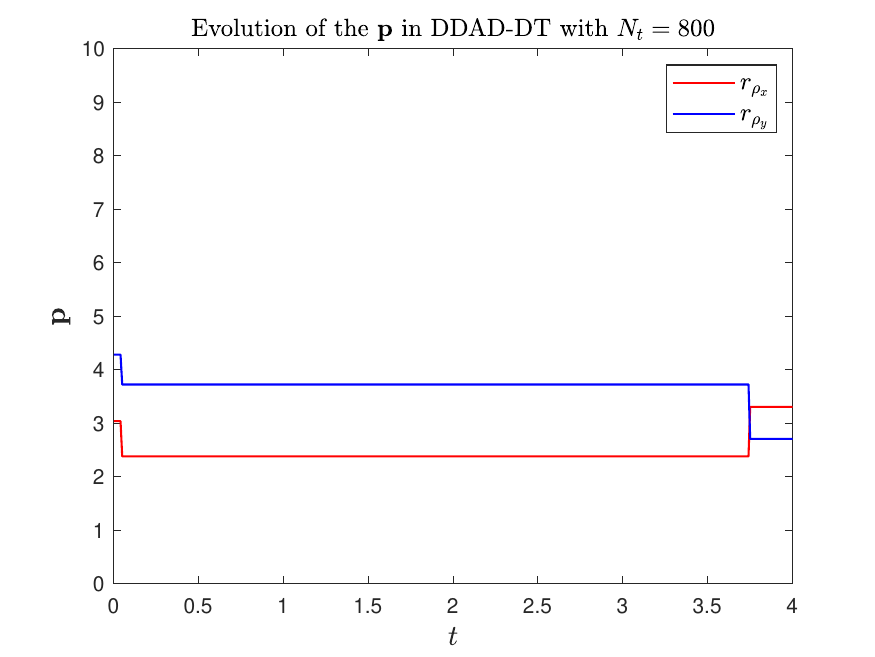}
	
	\caption{Comparison between the reference solution (blue dashed) and the DDAD-DT solution (red solid) for the exact solution~\eqref{eq:mcf_initial_1} with $N_t=800$ at $t=0,1,2,3,4$ (left); evolution of $\bm p=(r_{\rho_x},r_{\rho_y})$ (right).}
	\label{fig:mcf_1}
\end{figure}
\par
To further demonstrate the performance of the proposed method, we consider a more complex initial condition given by \eqref{eq:mcf_initial_2}, and increase the number of periodic basis functions to $20$. The corresponding results are presented in Fig.~\ref{fig:mcf_2}. 
\begin{equation}\label{eq:mcf_initial_2}
    \bm{x}(\rho,0)=(x(\rho,0),y(\rho,0))=\big((2+\cos(12\pi\rho))\cos(2\pi\rho),\ (2+\cos(12\pi\rho))\sin(2\pi\rho)\big).
\end{equation}

\begin{figure}[!ht]
	\centering
	\includegraphics[width=0.35\textwidth]{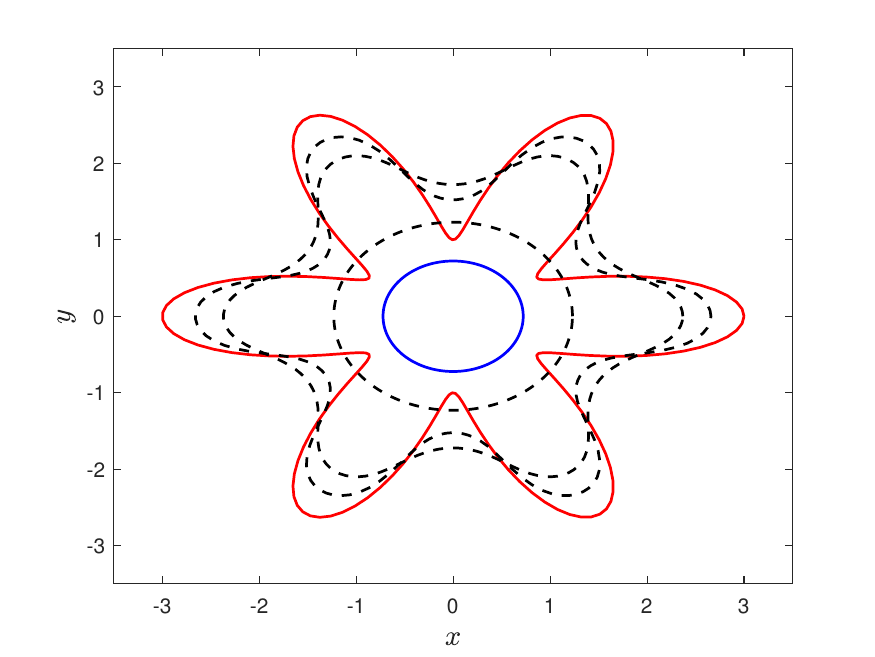}
	\includegraphics[width=0.35\textwidth]{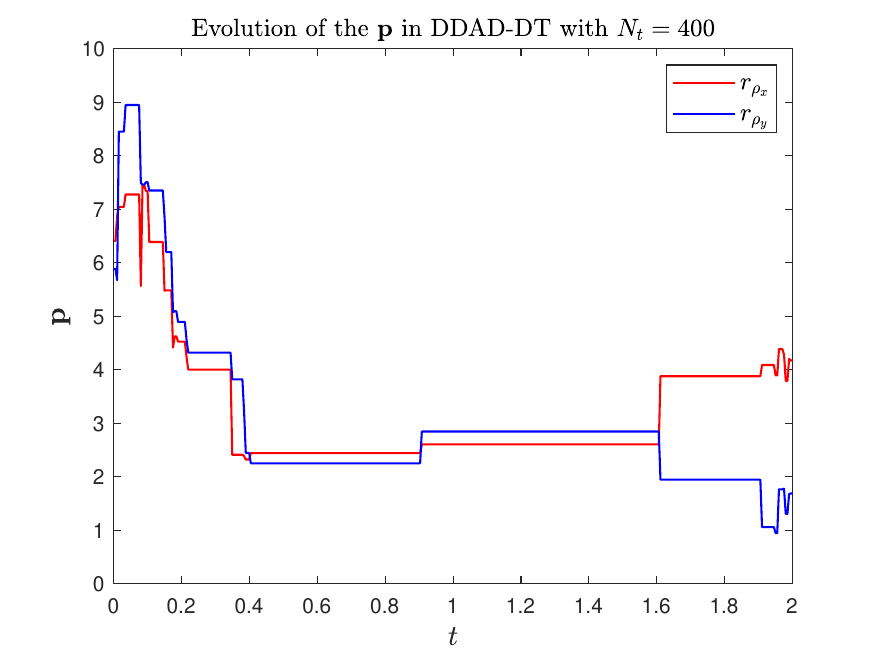}
	
	\caption{DDAD-DT solutions for the initial condition~\eqref{eq:mcf_initial_2} with $N_t=400$ at $t=0,0.1,0.2,1.5,2$ (left); evolution of $\bm p=(r_{\rho_x},r_{\rho_y})$ (right).}
	\label{fig:mcf_2}
\end{figure}

\subsection{High-dimensional Black-Scholes equation}

In this section, we solve the $d$-dimensional Black-Scholes model with uncorrelated noise using the proposed AD-RaNN framework; see \cite{ryck2025Approximation} for related approximation theory and further numerical studies. Let $W(t)=(W_1(t),\dots,W_d(t))$ be a $d$-dimensional standard Brownian motion, $\mu\in\mathbb{R}$ the stock return, and $\sigma_i>0$ the stock volatilities. By It\^o's formula and the Feynman-Kac theorem, for a payoff function $\psi$,
\begin{equation}
u(x,t)=\mathbb{E}\!\left[\psi\!\left(\left[x_i\exp\!\left(\left(\mu-\tfrac12\sigma_i^2\right)t+\sigma_i W_i(t)\right)\right]_{i=1}^d\right)\right],
\end{equation}
which solves the PDE
\begin{equation}\label{eqn:Black-Scholes}
\partial_t u=\frac12\sum_{i=1}^d |\sigma_i x_i|^2\,\partial_{x_ix_i}^2 u+\sum_{i=1}^d \mu x_i\,\partial_{x_i}u,
\end{equation}
with the initial condition $u(x,0)=\psi(x)$ for $x\in\Omega$ and an approximate Dirichlet boundary condition imposed via Monte Carlo:
\begin{equation}
u(x,t)= \frac{1}{N_s}\sum_{n=1}^{N_s}
\psi\!\left(\left[x_i\exp\!\left(\left(\mu-\tfrac12\sigma_i^2\right)t+\sigma_i W_i^{(n)}(t)\right)\right]_{i=1}^d\right),
\quad x\in\partial\Omega,\ t\in(0,T].
\end{equation}
\par
We consider $\Omega=[90,110]^d$, $T=1$, $\mu=-0.05$, and $\sigma_i=\frac{1}{10}+\frac{i}{200}$, $i=1,\dots,d$, with initial condition $\psi(x)=\max(\max_{1\le i\le d}x_i-100,0)$. Following \cite{ryck2025Approximation}, we use $N_s=16384$ Monte Carlo samples for the boundary estimator, and choose $N_{\mathrm{int}}=32768$ interior points, $N_{\mathrm{sb}}=16384$ boundary points, and $N_{\mathrm{tb}}=16384$ initial points. In this experiment, the computational overhead of the solver is evaluated independently of Monte Carlo sampling time, using pre-sampled boundary conditions and reference solutions.
\par
To reduce the computational cost of parameter adaptation, we subsample the collocation set
($3276$ interior points; $1638$ boundary/initial points; $800$ neurons; learning rate $=0.01$) for training stage.
 We share one parameter across all spatial dimensions and use a separate parameter for time, so that only $\bm{p}=[r_s,r_t]$ needs to be optimized, with initialization $\bm{p}_0=(0.1,0.1)$. Related results are reported in Table~\ref{tab:BS}, where AD-RaNN achieves a small relative $\ell_2$ error even for dimension $d=100$, while maintaining a moderate computational cost.

\begin{table}[!ht]
\centering
\caption{Results of the PDAD-ST method for the $d$-dimensional Black-Scholes equation~\eqref{eqn:Black-Scholes}. The table reports the trained parameter $\bm{p}$, the total runtime, and the $\ell_2$ error.}
\label{tab:BS}
\setlength{\tabcolsep}{18pt} 
\begin{tabular}{cccccc}
\hline
& $d$ &$m_1$ &$\bm{p}=(r_s,r_t)$ & time (s) & $\ell_2$ error \\ \hline
\multirow{3}{*}{PDAD-ST} & 20 &  3200 & (0.014,0.221) &28.88& 1.17e-02 \\ 
&50& 3200& (0.023,0.215) & 45.62&2.16e-02 \\
& 100& 3200 & (0.025,0.236)& 64.05&2.33e-02 \\ 
 \hline
\end{tabular}
\end{table}

\subsection{Diffusion-Reaction Dynamics: AD-RaNN-DeepONet Model}\label{subsec:DR_ADRaNNDeepONet}
We consider the nonlinear diffusion-reaction equation with a spatial source term, which is subject to homogeneous initial and boundary conditions
\begin{equation}\label{eqn:DR-AD-RaNN-DeepONet}
\begin{cases}
u_t(x,t)-D\,u_{xx}(x,t)-k\,u^2(x,t)=a(x),& \qquad (x,t)\in(0,1)\times(0,1],\\[2pt]
u(0,t)=u(1,t)=0,
&    \qquad  t\in(0,1],\\[2pt]
u(x,0)=0,
&   \qquad   x\in(0,1),
\end{cases}
\end{equation}
where the diffusion coefficient and reaction rate are set to $D=0.01$ and $k=0.01$.
The goal is to learn the solution operator
\begin{equation}\label{eq:DR_operator}
\mathcal{G}: a(\cdot)\mapsto u(\cdot,\cdot),\qquad 
u(x,t)=\mathcal{G}(a)(x,t).
\end{equation}
\begin{itemize}
    \item \textbf{Data generation:}
Following~\cite{jiang2026deeponet}, the training dataset consists of $N=10{,}000$ random realizations of the input function $a(x)$,
where each $a$ is sampled from a Gaussian random field with length scale $l=0.2$.
For each realization, a reference solution $u(x,t)$ is generated on a $100\times100$ mesh.
The test set contains another $1000$ realizations evaluated on the same $100\times100$ uniform grid.

\item \textbf{Hard enforcement of boundary and initial conditions:}
To satisfy the homogeneous Dirichlet boundary conditions and the zero initial condition in \eqref{eqn:DR-AD-RaNN-DeepONet} exactly, we adopt a hard-constraint formulation
and represent the operator output as
\begin{equation}\label{eq:DR_hbc}
\mathcal{G}(a)(x,t)=c(x,t)\,\widetilde{\mathcal{G}}(a)(x,t),
\qquad c(x,t)=t\,x(1-x),
\end{equation}
so that $\mathcal{G}(a)(x,t)$ automatically vanishes on $t=0$ and $x\in\{0,1\}$.
\end{itemize}
\par
To improve accuracy with minimal optimization cost, the trunk and branch first-layer weights are sampled as
\begin{equation}\label{eq:DR_param_dist}
W_b^{1}\sim U(-r_b,r_b),\qquad W_t^{1}\sim U(-(r_x,r_t),(r_x,r_t)),
\end{equation}
and we collect the distribution parameters as $\bm p=(r_b,r_x,r_t)$.
The AD-RaNN-DeepONet model uses $m_b=120$, $m_t=100$, and $40{,}000$ randomly selected solving pairs $(\bm a,\bm z)$. 
To reduce the training cost during parameter adaptation, we use a smaller configuration with $m_b=60$, $m_t=50$, and $20{,}000$ training pairs. 
The learning rates are set to $0.02$ (branch) and $0.2$ (trunk), and the number of training steps is $80$.
As reported in Table~\ref{tab:DR-AD-RaNN-DeepONet}, compared with the vanilla DeepONet and PI-DeepONet (\cite{wang2021Learning}), AD-RaNN-DeepONet achieves a (relative) $\ell_2$ error of $1.34\times 10^{-3}$ with a total runtime of $87.15$ seconds.
The Vanilla DeepONet and PI-DeepONet baselines are evaluated on the same diffusion-reaction dataset and on the same workstation as our AD-RaNN-DeepONet experiments.
These two baselines are implemented in Python using PyTorch.

\begin{table}[!ht]
\centering
 \caption{Comparison of operator-learning models for the diffusion-reaction problem~\eqref{eqn:DR-AD-RaNN-DeepONet} with $m_b=120$ and $m_t=100$. The table reports the total number of trainable parameters, runtime, relative $\ell_2$ error, and the learned distribution parameter $\bm{p}$ when applicable.}
\label{tab:DR-AD-RaNN-DeepONet}
\setlength{\tabcolsep}{12pt} 
\begin{tabular}{ccccc}
\hline
 model & 
parameters  & time (s) & $\ell_2$ error& $\bm{p}=(r_b,r_x,r_t)$ \\ \hline
  Vanilla DeepONet &  25,600  &2283.35& 2.17e-02 &-\\ 
PI-DeepONet & 25,600 & 4961.64 &3.81e-03 &-\\
 AD-RaNN-DeepONet& 12,000& 87.15&1.34e-03 &(0.02,4.69,1.41)\\
 AD-PI-RaNN-DeepONet& 12,000& 571.76&4.26e-03 &(0.04,3.94,1.52)\\ 
 \hline
\end{tabular}
\end{table}

\begin{figure}[!ht]
	\centering
\includegraphics[width=0.32\textwidth]{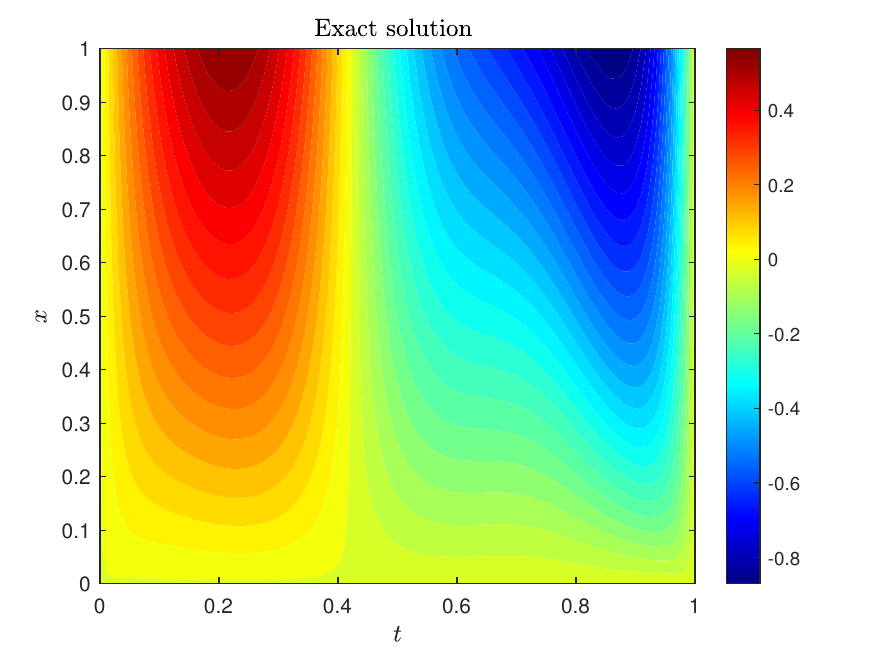}
\includegraphics[width=0.32\textwidth]{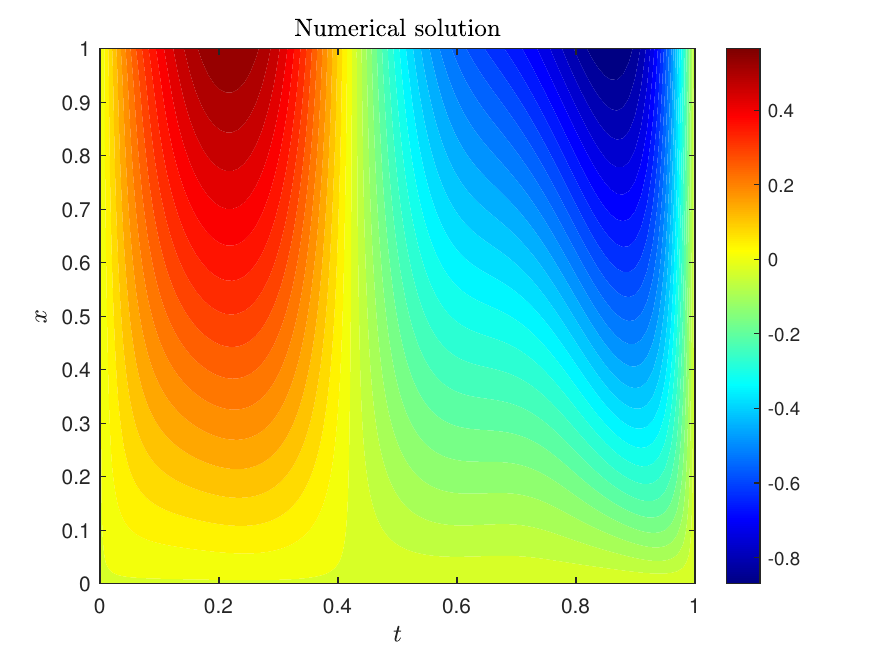}
\includegraphics[width=0.32\textwidth]{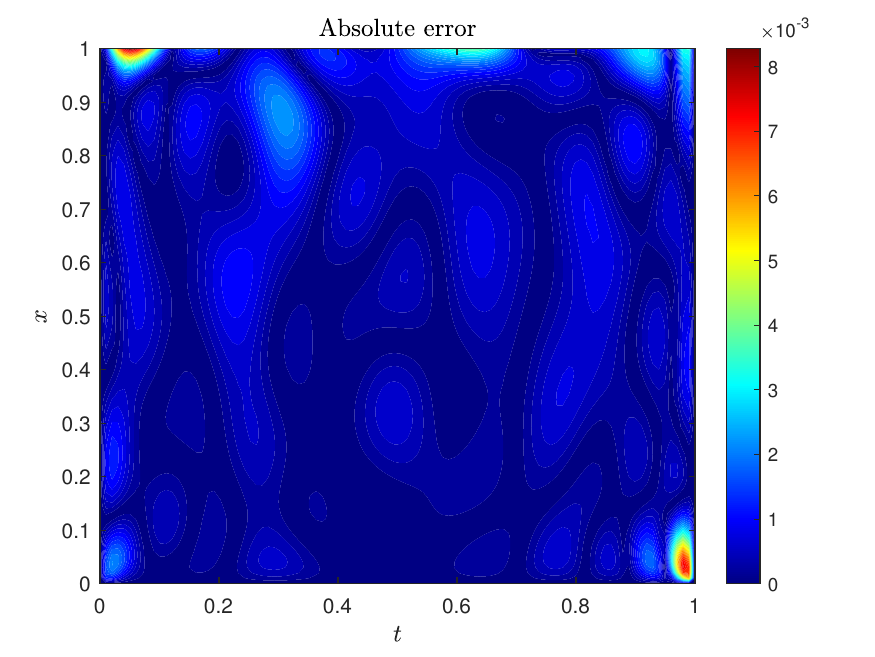}
	
	\caption{ Comparison of exact solutions (left), predictions (middle), and absolute errors (right) for equation~\eqref{eqn:DR-AD-RaNN-DeepONet} solved by AD-RaNN-DeepONet.
}
	\label{fig:DR-AD-RaNN-DeepONet}
\end{figure}

\section{Conclusion}\label{sec:conclusion}
RaNN-based PDE solvers are appealing because they reduce neural PDE approximation to a least-squares problem on a randomized hidden representation. Their central limitation, however, is the strong sensitivity of solution quality to the sampling distribution of the hidden-layer parameters. In this work, we addressed this issue by introducing Adaptive-Distribution Randomized Neural Networks (AD-RaNN), in which randomized feature generation is elevated from a fixed heuristic choice to a low-dimensional adaptive optimization problem. This distribution-level viewpoint provides a structural middle ground between fully fixed random features and fully trained hidden representations, while preserving the computational advantages of least-squares-based randomized solvers.

Viewed from a broader methodological perspective, the main outcome of this paper is a reusable adaptive distribution learning principle for randomized hidden representations, rather than only a collection of problem-specific RaNN variants.
 Within this framework, we developed the PDE-driven and data-driven mechanisms PDAD and DDAD, deployed them in unified space-time and discrete-time PDE solvers, incorporated a local layer-growth enhancement for unresolved localized structures, and extended the same principle to randomized operator learning through AD-RaNN-DeepONet. Across a diverse set of benchmarks, the numerical results show that the proposed framework can substantially reduce sensitivity to manually chosen feature distributions and achieve strong empirical accuracy across a broad range of PDE regimes.

Future work includes several directions. First, more advanced time integrators and stabilization techniques, such as structure-preserving or SAV-type schemes, will be investigated to further improve long-time accuracy and robustness, particularly for stiff or multiscale dynamics. Second, extending the framework with parallel computing and domain decomposition may enable larger-scale and multi-physics applications, such as solid-state dewetting and other complex interface-evolution problems.

\bibliographystyle{plain}

\bibliography{references}

\end{document}